\documentclass[12pt,reqno]{amsart}

\makeatletter
\@namedef{subjclassname@2020}{2020 Mathematics Subject Classification}
\makeatother

\usepackage[paper=a4paper,margin=0.8in,bmargin=1in]{geometry}

\usepackage[
  open,
  openlevel=2,
  atend,
  numbered
]{bookmark}

\usepackage{amsmath}
\usepackage{amssymb}
\usepackage{amsthm}
\usepackage{color}
\usepackage{hyperref}
\definecolor{darkblue}{rgb}{0.0,0.0,0.85}
\definecolor{darkgreen}{rgb}{0,0.5,0}
\hypersetup{colorlinks,breaklinks,
            linkcolor=darkblue,urlcolor=darkblue,
            anchorcolor=darkblue,citecolor=darkblue}
\usepackage{mathrsfs}
\usepackage{enumitem}
\usepackage{graphicx}
\usepackage{comment}

\newcommand{\rvec}{{\mathrm r}}
\def\vsigma{{\vec\sigma}}
\def\ns{_n^\sigma}

\def\omvs{\om^\vsigma_{nkm}}
\def\sumn{ \displaystyle\sum_{n,\sigma}}
\def\sumnkmsc{\displaystyle\sum_{n,k,m, \vsigma;\textnormal{conv}}}
\def\sumcnkm{\displaystyle\sum_{n,k,m;\textnormal{conv}}}
\def\tauO{_{\tau=\Om t}}
\def\etau{e^{\tau\cL}}
\def\emtau{e^{-\tau\cL}}
\def\sL{{\mathsf L}}
\def\sigg{{n,\sigma_1,\sigma_2}}
\def\vr{{\boldsymbol r}}
\def\vrd{{\boldsymbol r}_{\delta}}

\numberwithin{equation}{section}
\newtheorem{theorem}{Theorem}[section]
\newtheorem{proposition}[theorem]{Proposition}
\newtheorem{lemma}[theorem]{Lemma}

\theoremstyle{definition}
\newtheorem{definition}[theorem]{Definition}
\newtheorem{remark}[theorem]{Remark}

\newcommand{\Om}{\Omega}
\def\om{\omega}

\newcommand{\vU}{{\boldsymbol{U}}}
\newcommand{\vV}{{\boldsymbol{V}}}
\def\vpsi{{\boldsymbol\psi}}
\newcommand{\vu}{\boldsymbol{u}}
\newcommand{\vv}{\boldsymbol{v}}
\newcommand{\vh}{{\boldsymbol h}}
\def\eig{{\boldsymbol e}}
\def\tvu{{\wt\vu}}
\def\tvU{{\wt\vU}}
\def\vVR{{\vV}^{R}}
\def\tVR{{V}^{R}}
\def\vw{{\boldsymbol{w}}}
\newcommand{\ck}{\check k}
\newcommand{\cm}{\check m}
\newcommand{\cn}{\check n}
\newcommand{\Jr}{{\mathsf  J}} 
\newcommand{\leray}{{\mathcal P}^{\textnormal{div}}}
\def\cP{{\mathcal P}}
\def\cPl{\cP_{\!<R}}

\def\cL{{\mathcal L}}
\newcommand{\bC}{{\mathbb C}}

\newcommand{\cN}{{\mathcal N}}
\def\cNp{{{\mathcal N}_0}}
\def\cNg{{{\mathcal N}_G}}

\newcommand{\bR}{{\mathbb R}}
\newcommand{\bZ}{{\mathbb Z}}

\newcommand{\bT}{{\mathbb T}}
\newcommand{\bTt}{{\bT^3}}
\def\vzero{{\vec 0}}
\def\bZt{\bZ^3}
\def\bZtno{\bZt\backslash\{\vzero\}} 
\newcommand{\rkj}[1]{r_k^{(#1)}}
\newcommand{\bs}{\boldsymbol}
\def\bp{{\mathbf p}}
\def\bq{{\mathbf q}}
\def\bpsi{{\bs \psi}}

\def\ep{\epsilon}
\def\mq{{\mathsf q}}
\def\mQ{{\mathsf Q}}
\def\sg{{\mathsf g}}
\def\sk{{\boldsymbol{\ell}}} 
\def\An{{\mathcal A}}
\def\cC{{\mathcal C}}
\def\scc{{\mathscr C}}
\def\scm{{\mathrm M}}

\def\rd{{\mathrm d}}
\def\ri{{\mathrm i}}
\def\IL{\Pi_\textnormal{IL}}
\def\SP{\Pi_\textnormal{SP}}
\def\EC{\Pi_\textnormal{EC}}
\def\bpi{{\bs\pi}}
\def\Jset{{\boldsymbol{\mathcal J}}}
\def\jc{{\gamma^{o}}}
\def\Gset{{\boldsymbol{\Gamma}}}

\newcommand{\be}{\begin{equation}}
\newcommand{\ee}{\end{equation}}
\def\pa{\partial}
\def\pD{{\mathsf D}}
\def\cD{{\mathcal D}}
\def\wh{\widehat}
\def\wt{\widetilde}
\def\ind{{\bs 1}}
\def\dcd{\cdot,\cdot}

\newcommand{\chibr}[1]{{{\mathfrak B}}{\left\{#1\right\}}}

\newcommand{\chibrB}[1]{{{\mathfrak B}}{\Big\{#1\Big\}}}
\newcommand{\ip}[1]{\left\langle #1 \right\rangle}
\newcommand{\ipb}[1]{\big\langle #1 \big\rangle}

\newcommand{\ipwk}[1]{\left\langle #1 \right\rangle_{\!wk}}
\newcommand{\qua}[1]{\quad \text{#1} \quad}
\newcommand{\qqua}[1]{\qquad \text{#1} \qquad}
 \newcommand{\sectionb}[1]{\section{\bf{#1}}}
 \newcommand{\pDn}[2]{\|#2\|_{#1}}
 \newcommand{\pDnb}[2]{\big\|#2\big\|_{#1}}

 \title[Near Resonance]{Near Resonance Approximation of Rotating Navier-Stokes Equations}
 
\author{Bin Cheng}
\address{Department of Mathematics, University of
          Surrey, Guildford, GU2 7XH, United Kingdom}
   \email{b.cheng@surrey.ac.uk}
   
\author{Zisis N. Sakellaris}
\address{Department of Mathematics, University of
          Surrey, Guildford, GU2 7XH, United Kingdom}
   \email{z.sakellaris@surrey.ac.uk}
\date{06 October 2021}

\begin{document}
\maketitle

\begin{abstract}
We formalise the concept of near resonance for the rotating Navier-Stokes equations, based on which we propose a novel way to approximate the original PDE. The spatial domain is a three-dimensional flat torus of  arbitrary aspect ratios. We prove that the family of proposed PDEs are globally well-posed for any rotation rate and initial datum of any size in any $H^s$ space with $s\ge0$. Such approximations retain much more 3-mode interactions, thus more accurate, than the conventional exact resonance approach. Our approach is free from any limiting argument that requires physical parameters to tend to zero or infinity, and is free from any small divisor argument (so estimates depend smoothly on the torus' aspect ratios). The key estimate hinges on counting of integer  solutions of Diophantine inequalities rather than Diophantine equations. Using a range of novel ideas, we handle rigorously and optimally challenges arising from the non-trivial irrational functions in these inequalities. The main results and ingredients of the proofs can form part of the mathematical foundation of a non-asymptotic approach to nonlinear oscillatory dynamics in real-world applications. 
\end{abstract}

\section{\bf Introduction}

We investigate near-resonance-based approximations of the rotating Navier-Stokes (RNS) equations, a well-known model of geophysical fluid dynamics (GFD).  We prove the global well-posedness of solutions under very mild conditions for this novel class of PDEs that are characterised by what we call ``bandwidth'', a wavenumber dependent  parameter that allows much more 3-mode interactions to be retained by the near resonance approach than by the conventional {\it exact} resonance approach. We also prove explicit error bounds for such approximations. In a nutshell, we achieve in the proposed NR approximations two desirable properties: global solvability for arbitrary rotation rates and a large set of nonlinear interactions that dominate the dynamics for fast rotation rates.

The spatial domain, denoted by $\bTt$, is a three-dimensional flat torus with anisotropic periods $(2\pi \sL_1, 2\pi \sL_2, 2\pi)$ for any positive constants $\sL_1,\sL_2$. 
For a $\bC^3$-valued, Lebesgue integrable function $\vu$, the div-free condition $\nabla\!\cdot\vu=0$ is understood in the weak sense: that is $\int_{\bT^3}\vu\cdot\!\nabla h\,\rd x=0$ for any $\bC $-valued smooth function $h$. Let $\leray\vu$ denote the Leray projection of  $\bC^3$-valued  $\vu\in L^p(\bTt)$ for $p\in(1,\infty)$ onto the subspace of div-free functions. More precisely, 
\[
\leray\vu:=\vu-\nabla\Delta^{-1}\nabla\cdot\vu.
\]
Then, we can define  bilinear form for $\bC^3$-valued $\vu,\vv\in H^1(\bTt)$,
\[
B(\vu , \vv):=\leray (\vu \cdot \nabla \vv).
\] Let $\Jr$ denote the rotation matrix $\Jr =${\tiny$\begin{pmatrix}0,&-1,&0\\1,&0,&0\\0,&0,&0\end{pmatrix}$} 
and define linear operator 
\[
\cL:= \leray \Jr \,\leray .
\]
The incompressible RNS equation is then formulated for $\bR^3$-valued unknown $\vU(t,x)$ as
\begin{equation} \label{NS}
\partial_t \vU   +B(\vU , \vU) = \Om\cL\vU + \mu \Delta \vU,\qquad \nabla\!\cdot\vU=0.
\end{equation}
 Scalar constant $\mu>0$ denotes viscosity. 
The $\Om\cL\vU$ term represents the Coriolis force due to the frame's rotation where scalar constant $\Omega$ is proportional to the rotation rate (WLOG, let $\Omega>0$) and takes the form of  $1/(Rossby\;number)$ upon nondimensionalisation. 
 We omitted an external forcing term, as its treatment under various regularity assumptions is fairly well-documented.  Initial datum $\vU(0,\cdot)$ is in $L^2(\bTt)$ (or a more regular subspace of it) and div-free. 

The linear term $\Om\cL\vU$ with usually large $\Om$ is responsible for inertial waves\footnote{\,\label{fn:1}Evident from dispersion relation \eqref{om:eig},  
 the Coriolis term does not generate (linear) oscillatory dynamics in horizontal modes. Such a large set of non-oscillatory modes is a common feature that separates many fluid models (even in the inviscid case) from classically known dispersive PDEs, hence requiring different treatments.}.  Early results of inertial waves  by Poincar\'e \cite{Poincare} were discussed and extended in \cite[\S 2.7]{Greenspan} in modern notations. This term generates operator exponential $e^{\Om\cL t}$ which then leads to a transformation of the RNS equations by way of Duhamel's principle, which explicates 3-mode interactions in the eigen-expansion of the ``transformed'' bilinear form -- see \eqref{B:full1}. We then define $\wt B(\dcd)$ to approximate $B(\vu , \vv)=\leray (\vu \cdot \nabla \vv)$ via restricted convolution in eigen/Fourier-expansions. Such restriction retains and discards 3-mode interactions  based on the smallness of the linear combination of 3 frequencies -- which will be called ``triplet value'' --  rather than the zero-ness of the triplet value as used for defining exact resonances. In detail, the triplet value is 
 \be\label{om:star}
\omvs: =\sigma_1 \frac{ \cn_3  }{ |{ \cn}|}+\sigma_2  \frac{\ck_3  }{ |{ \ck}|} +\sigma_3 \frac{ \cm_3  }{ |{\cm}|}\,,
\ee  
as a function of $(n,k,m)\in(\bZt)^3$ for the 3 wavevectors and $\vec\sigma=(\sigma_1,\sigma_2,\sigma_3)\in\{+,-\}^3$ for the 3 signs of temporal phases. The accent $\check{\;}$ indicates a ``domain-adjusted'' wavevector defined as
\[
\begin{aligned}
\cn:=\Big(\,\frac{n_1}{\sL_1},\frac{n_2}{\sL_2},n_3\Big)^T\qua{ for } n=(n_1,n_2,n_3)^T\in\bZt  \,.
\end{aligned}
\]
 The Euclidean length and dot product are applied as usual, e.g. $\cn\cdot\ck=\frac{n_1k_1}{\sL_1^2}+\frac{n_2k_2}{\sL_2^2}+n_3k_3$.
 
 Now, we are ready to define the {\bf near resonance} (NR) set as
\be\label{def:cN}
\cN: = \Big\{ ({n, k,m})\in \big(\bZtno\big)^3 \,:\, \min_{\vsigma\in  \{+,-\}^3} \left|\omvs\right| \le\delta(n,k,m),\;\;n+k+m=\vzero\,\Big\}
\ee
where the {\bf bandwidth} $\delta$ is related to but fundamentally different from small divisor -- see \S \ref{ss:lit}.

 The convolution condition $n+k+m=\vzero$ differs slightly from convention in order to make the indicator\footnote{\,For set $S$, the associated indicator function is defined as 
\(
\ind_S(x)=1\) if $x\in S$
and $\ind_S(x)=0$ if $x\notin S$.
} 
function $\ind_{\cN}(n,k,m)$ symmetric with respect to permutation of all three arguments.

In Fourier modes (\eqref{Fourier:s}--\eqref{Fourier:c}), the bilinear form of RNS equations satisfies the convolution $B(\vU,\vV)=\sum_{k,m} B\big(e^{\ri \ck\cdot x}U_k\,, \, e^{\ri \cm \cdot x} V_m\big)$. Then, postponing the detailed motivation to \S \ref{ss:motivate}, we introduce the {\bf restricted convolution} based on the NR set $\cN$,
\be\label{wt:B:full2}
\wt B(\vU,\vV)=\sum_{k,m\in(\bZt)^2} B\big(e^{\ri \ck\cdot x}U_k\,, \, e^{\ri \cm \cdot x} V_m\big)\,\ind_{\cN}({-k-m,k,m}),
\ee
and arrive at the main equation
\begin{equation} \label{NR}
\partial_t \tvU   +\wt B(\tvU , \tvU) = \Om\cL\tvU + \mu \Delta \tvU,\qquad \nabla\!\cdot\tvU=0,
\end{equation}
with div-free initial datum $\vU(0,\cdot)\in L^2(\bTt)$.
We name this the {\bf near-resonance approximation} of the RNS equations. The resemblence of \eqref{NS} and \eqref{NR} can be also seen in their equivalent, ``transformed'' versions \eqref{m:NS} and \eqref{m:NR} respectively with the same transformation rules as in \eqref{def:m:B} and \eqref{m:B:nr} respectively. Also compare the convolutions in the pre-transformed bilinear forms given above and in the transformed counterparts \eqref{B:full1} and \eqref{wt:B:full1} respectively.

With weak solutions given in Definition \ref{def:wk} and homogeneous $H^s(\bTt)$ norm for real $s$ (denoted by $\pDn{s}{\cdot}$) defined in \S \ref{s:results}, we state our main results. For global well-posedness, the bandwidth should decay like $(wavenumber)^{-1}$ with a minor logarithmic attenuation. 
\begin{theorem}[{\bf Global well-posedness of NR approximation}]\label{thm:NR}
 In the near resonance set $\cN$, let bandwidth $\delta=\delta(n,k,m)\in[0,\frac12)$ satisfy 
 \be\label{de:max}
\delta\text{\, only depends on }\max\{|\cn|,|\ck|,|\cm|\},
\ee
and, for some constant $\hat c_1>0$, 
\be\label{de:up}
\delta \log\tfrac1\delta \le \hat c_1\min\big\{ |\cn|^{-1},  |\ck|^{-1}, |\cm|^{-1}\big\}\,,
\ee
 where $\delta(n,k,m)=0$ is allowed under the convention $0|\log 0|=\displaystyle\lim_{\delta\to 0}\delta\,\log\tfrac1\delta=0$. For any real $s\ge 0$ and any
 $\bR^3$-valued, div-free, zero-mean\footnote{\,The zero-mean assumption is imposed WLOG and this property propagates in time -- see the start of \S \ref{s:results}.} initial datum $\tvU(0,\cdot)=\tvU_0(\cdot)\in H^s(\bTt)$, the NR approximation \eqref{NR} admits a {\bf unique global} weak solution
\[
\wt\vU\in \scc^0([0,\infty); H^{s}(\bTt))\cap L^2([0,\infty); H^{s+1}(\bTt))
\] that satisfies the energy \emph{equality} (rather than inequality)
\be\label{NR:en:L2}
\|\tvU(T,\cdot)\|_{L^2}^2+2\mu\int_0^{T}\|\nabla\tvU\|_{L^2}^2\,\rd t =  \|\tvU_0\|_{L^2}^2\,, \quad\;\forall \,T\ge 0\,,
\ee
and    the global bounds
\be\label{global:Hs}
\max_{t\ge0}\pDn s{\tvU(t,\cdot)}^2\le e^{\mu^{-2}{\hat C}E_{00}}E_{s0}\qua{\textnormal{and}}\mu\int_0^\infty\pDn {s+1}{\tvU}^2\,\rd t \le E_{s0}+\mu^{-2}{\hat C}E_{00}\,e^{\mu^{-2}{\hat C}E_{00}}E_{s0},
\ee
for 
$E_{s0}:=\pDn s{\tvU_0}^2$, $E_{00}:=\pDn 0{\tvU_0}^2$ and constant $\hat C=\hat C(s,\hat c_1,\bTt)$.

The solution Lipschitz-continuously depends on the initial datum in the sense of Lemma \ref{lem:stable}.

\end{theorem}
\begin{remark}\label{p:domain}
In this article, when we say a  constant depends on the torus domain $\bTt$, it means the constant depends {\it smoothly} on its aspect ratios $\sL_1,\sL_2$. Though expected from physics, this is not achievable if small divisor argument is involved.
\end{remark}
The proof is in \S \ref{ss:mainproof}. It crucially depends on the following result that our proposed bilinear form $\wt B$  satisfies the same estimates as the classic $\leray(\vu\cdot\nabla\vv)$ in 2D, even though the velocity field and spatial domain of $\wt B$ are  3D, and there is no apparent decoupling of any kind. Then \eqref{NR} with any rotation rate $\Omega$ enjoys most of the nice properties that 2D Navier-Stokes equations have (\cite{Leray}), which lays the foundation for the proof of Theorem \ref{thm:NR}. Further, such 2D-like feature can imply desirable numerical properties (\cite{Temam}), which will be covered in future work. 

\begin{theorem}[{\bf 2D-like estimates}]\label{thm:2Dlike}
For $\wt B$ in \eqref{wt:B:full2} with the same assumptions on $\cN,\delta$ as Theorem \ref{thm:NR}, and $\bR^3$-valued, div-free, zero-mean  $\vu,\vv,\vw\in H^{s+1}(\bTt)$ (real $s\ge0$), we have 
\be\label{tri:est}\begin{aligned}
\Big|\big\langle \pD^{s}\wt B( \vu,\vv),\pD^s\vw\big\rangle\Big|\lesssim 
2^{s}&\big(\pDn0\vu\,\pDn{s+1}\vv+\pDn{s}\vu\,\pDn{1}\vv\big)\pDn{s+1}\vw\\
+2^s&\big(\pDn{1}\vu\,\pDn{s+1}\vv+\pDn{s+1}\vu\,\pDn{1}\vv\big)\pDn{s}\vw,
\end{aligned}
\ee
and a further estimate with identical second and third arguments
\begin{align}
\nonumber&\Big|\big\langle \pD^{s}\wt B( \vu,\vw),\pD^s\vw\big\rangle\Big|\lesssim\\
\label{comm:est}&\begin{cases}
s\,\pDn1\vu\,\pDn{s}\vw\,\pDn{s+1}\vw,&\textnormal{for real }s\in(0, 1],\\
s2^{s}\big(\pDn1\vu\pDn{s}\vw\pDn{s+1}\vw+
\pDn{s}\vu\pDn1\vw\pDn{s+1}\vw+\pDn{s+1}\vu\pDn{1}\vw\pDn{s}\vw\big),&\textnormal{for real }s>1 .
\end{cases}
\end{align}
The missing constants in the $\lesssim$ notation depend on ${\hat c_1}, \bTt$
\end{theorem}

The proof is at the start of \S \ref{s:mainproofs} and relies on results proven in the bulk of this article including \S \ref{s:con:L}  (linking restricted convolution on torus domain to integer point counting, well-known in harmonic analysis), \S \ref{S:R:sum} (linking point counting to volume of sublevel set), \S \ref{S:vol} (volume estimate) and Appendix \S \ref{s:j:c} (a general result on point counting concerning disjoint Jordan curves). The latter three parts contain the key technical novelties that we employ to resolve a range of challenges with rigour, optimality and a good level of generalisability.  For challenges, we invite the reader to recall definitions \eqref{om:star}--\eqref{def:cN} and consider the number of integer points $k=(k_1,k_2,k_3)$ satisfying
\be\label{om:k:ineq}
f(k_1,k_2,k_3):=C_n+\sigma_2  \frac{\ck_3  }{ |{ \ck}|} -\sigma_3 \frac{ \cn_3+\ck_3  }{ |{\cn+\ck}|}\in[-\delta,\delta],\quad\text{with $k$ localised as }\,\tfrac12|\cn| \le |\ck| \le |\cn|,
\ee
for fixed $0<\delta\ll1$ and $n,\vsigma$ as given parameters and with constant $C_n:=\sigma_1 \frac{ \cn_3  }{ |{ \cn}|}$. 
The intuitive $O(|\cn|^3\delta+|\cn|^2)$ upper bound for the point count is a false claim due to the $O(|\cn|^3\delta\log{1\over\delta}+|\cn|^2)$ {\it lower} bound for certain choices of $n$ -- see Appendix \ref{S:op}. Let us consider the simplest 1D case  $f_1:\bR\mapsto\bR$ and the upper bound of number of integers $k_1$ subject to $|f_1(k_1)| \le \delta$. Then one should at least prove a lower bound of $|f_1'(x)|$ for $x$ near a zero of $f_1$ assuming it is not too close to a stationary point (such lower bounds can appear in literature concerning the van der Corput lemma). However in our 3D case, $f(k_1,k_2,k_3)$ combines square roots in such a fashion (e.g. without obvious convexity) that it is highly non-trivial to prove {\it sharp} lower bounds for its derivatives in the neighbourhood of its zero level set while also addressing degeneracy near stationary points\footnote{\,Let us only consider the volume of sublevel sets for now -- issues from the boundary can be seen in examples \eqref{2D:e}, \eqref{3D:e}. Let $\tilde k_j=\frac{k_j}{|\cn|}\,(j=1,2,3)$  so that $\tilde k_1^2+\tilde k_2^2+ \tilde k_3^2\in[\frac14,1]$ by the last part of \eqref{om:k:ineq}. If the sublevel set locally ``resembles'' 
\(
\big\{(k_1,k_2,k_3):\tilde k_1\tilde k_2\tilde k_3(\tilde k_3-a)\in[-\delta,\delta]\big\}
\) for constant $a$, then the volume of sublevel set localised as in \eqref{om:k:ineq} for the $a\sim O(1)$ is $O(|\cn|^3\delta\log\frac1\delta)$ case but the volume for the $a=0$ case is $O(|\cn|^3\sqrt{\delta})$.}  where in fact higher derivatives of $f$ can possibly get too close to zero as well. Further rationalisation of \eqref{om:k:ineq} would result in an 8th degree polynomial in $(k_1,k_2,k_3)$ with $(\delta,n,\vsigma)$-dependent coefficients, and it is still highly non-trivial to {\it optimally} quantify those lower bounds of derivatives. There is also a lesser issue in topology: the level set of a function at a critical value is not necessarily manifold(s), for example the boundary of the 2D set defined by $0<x^2-y^2<\delta$. But thanks to the {\it inequality} nature of sublevel sets, it can be approximated by a sequence of regular sets $\ep< x^2-y^2<\delta$ with $\ep\to 0$. Also see Figure \ref{fig:top}.

Estimates \eqref{tri:est}, \eqref{comm:est} are ``2D-like'' in the following sense. The simple inequalities \eqref{tri:inf:fi} for {\it generic} convolutions indicate the ``cost'' of  derivatives is usually $\frac32$, namely  $\frac12\times dimension$, in view of the orders of derivatives on the left-hand and right-hand sides. In the above result, however, the cost of derivatives is reduced to 1 (note $\wt B$ contains one derivative).

\begin{theorem}[{\bf $O(\Om^{-1})$ error estimates, independent of $\mu$}] \label{thm:local:e}
Let bandwidth $\delta=\delta(n,k,m)\in(0,1)$ in the near resonance set $\cN$ satisfy the \emph{lower} bound
\be\label{de:lower}
\delta \log\tfrac1\delta \ge \hat c_2\min\big\{|\cn|^{-1},  |\ck|^{-1}, |\cm|^{-1}\big\},
\ee
for some constant $\hat c_2>0$.
 Consider the RNS equations \eqref{NS} and its NR approximation \eqref{NR} with respective initial data $\vU_0$ and $\tvU_0$ that are $\bR^3$-valued, div-free, zero-mean, and satisfy 
\[
\pDn {s}{\tvU_0}^2\;\le E_0\qua{\, and \,}\pDn {s^\flat}{\vU_0}^2\;\le E_0\,,\qua{ for real}s>3,\;s^\flat>\tfrac52.
\]
Suppose real number $s'=0$ or ${s'\in[1,s-3)}$ (if $s>4$) with $s'\le s^\flat$.
Then there exist constants $C(s,s^\flat,s',\bTt,E_0)$ and $c(s,s^\flat,s',\bTt)$ independent of $\mu,\Omega,\ep_0$ so that
\[
\pDn {s'}{\tvU(t,\cdot)-\vU(t,\cdot)}^2\le \pDn {s'}{\tvU_0-\vU_0}^2+C\hat c_2^{-2}\Om^{-2}\qua{ \, for }t\in[0,c/E_0^{\frac12}].
\]
\end{theorem}

 Its proof in \S \ref{ss:errorproof} essentially uses integrating by parts (related to Riemann-Lebesgue lemma)  
\be\label{RLL}
\Big\|\int_0^T h(t)\exp(i\om_*\Omega t)\,\rd t \Big\|\le O\Big(\frac{\|\pa_t h\|T+\|h\|}{|\om_*\Omega|}\Big)\,,\quad\;\text{for } \;\; \om_*\Omega\ne0\,,
\ee
where the norms are possibly different and $\om_*$ is just a placeholder for a triplet value that is {\it outside} the NR bandwidth. 
Based on that proof, it takes very little effort in Remark \ref{re:global:sol} to show solvability of the original RNS equation for all $t\ge0$ with sufficiently large $\Om\gg1$. Compared to error estimates of \cite[Theorem 6.3]{BMN:global:0}, \cite[Lemma 5.2]{BMN:global}, \cite[Theorem 2]{Gallagher}, \cite[Theorem 2]{Gallagher:parabolic} for the {\it exact} resonance-based approximations, the above result is explicit in terms of $\Om^{-1}$, independent of $\mu$ and depends smoothly on domain's aspect ratios $\sL_1,\sL_2$ (see Remark \ref{p:domain}). We also note that the choice of the lower bound condition \eqref{de:lower} on $\delta$ is solely for it to be relatable to the upper bound condition \eqref{de:up}. Removal of the $\log\frac1\delta$ factor is almost trivial. Further, changing the $-1$ power in \eqref{de:lower} will only require reconsidering the derivative gap $s-s'$ without essential modification of the proof. Finally we note that error estimates based on this type of proof are most relevant for $\delta$ values not too small or large because as argued in \S \ref{ss:lit}, having $\delta$ too small amounts to selection based on exact resonances whereas having $\delta\ge3$ apparently recovers the original bilinear form $B$.

 The virtue of studying NR approximations is multi-fold.

For oscillatory dynamics, nonlinearity is the medium which allows modes of the solution to  interact. One of the quantitative indicators of this nonlinear effect is the smallness of the triplet values as named above. The  exact resonance approach is only concerned with zero triplet value which apparently exemplifies this effect. This is all that one can retain by ``thinking infinitely'' and thus studying the limits, often known as singular limits  -- c.f. the seminal work of Schochet in \cite{Schochet:JDE}.  In proposing the NR-based approach however, we are concerned with the reality that physical parameters, large or small, are always {\it finite}, and  not $\infty$ or 0. In GFD applications (\cite{Pedlosky}), small parameters associated with oscillatory dynamics are not that small -- for example the Rossby, Froude and  Mach numbers typically range from 0.01 to 0.1. This means  an asymptotic approach can miss considerable amount of nonlinear coupling that plays non-negligible part in the full dynamics. Therefore we propose to include contributions from  small but finite triplet values, hence the notion of bandwidth for that smallness.  
We note that one approach to complement the study of singular limits is to obtain explicit error bounds, namely convergence rates, in terms of the small/large parameters. Then, one can go even further to study next order corrections (\cite{Gallagher, Gallagher:parabolic}).

The drastic increase of 3-mode interactions included in the NR approximations is evidenced in Appendix \ref{S:op}. There, we show the {\it lower} bound of integer point count arising from our proposed NR sets is far greater than the {\it upper} bound of integer point count in \cite[Lemma 4.1]{KY:number} for exact resonance sets.

Crucially, due to the widening of strict zero-ness to its fuzzy neighbourhood, we increase the dimension of the set of included modes when viewed in continuum. Specifically, we change from level sets to sublevel sets, which fundamentally requires novel ideas. For example, we need to study Diophantine {\it inequalities}. Then, number theoretical tools and small divisor arguments, such as those applied to RNS equations in \cite{BMN:global:0,BMN:global,Gallagher:parabolic,KY:number} (based on exact resonance, thus involving  Diophantine equations) are often not applicable here. 

This challenge is however countered by the benefit of robustness since the integer points in a neighbourhood of a curve/surface, compared to those on the curve/surface itself, form a set that is much more stable  against perturbation. For example, the properties of the Diophantine equations and small divisors in the above works depend non-continuously on the domain's aspect ratios, but this rather unphysical sensitivity is no longer an issue when equations are relaxed to inequalities. 

\subsection{Literature}\label{ss:lit}

 Recent decades have seen the power of PDE analysis in {\it exact} resonances. The list is so long that for brevity, we only mention a partial list \cite{KM:limit, BMN:global:0, BMN:global, Ch:Book, EmMa:CPDE, Gallagher, Gallagher:parabolic} and references therein.   In \cite{Schochet:JDE}, Schochet proves for a large class of multi-time-scale PDEs the existence of a generic formula of singular limits -- effectively exact-resonance-based approximations -- for unprepared initial data. In many works, small divisor condition/estimate is used to quantify how lower bound of {\it non-zero} triplet values (in absolute value) depends on wavenumbers, so in our language, setting bandwidth $\delta$ less than such lower bound will reduce the NR set to the exact resonance set. More recently, results were proven in \cite{Cheng:SIMA:2012, Cheng:SIMA:2014, ChMa:zonal} for singular limit problems with physically relevant boundary conditions and  domain geometry for which Fourier analysis is not applicable. Although exact resonance is considered therein, their convergence rate (i.e. error bound) estimates are explicitly in terms of the small parameter, thus not requiring infinitesimal assumption. 
  In most recent years, there have been the first results in {\it three-scale} singular limit problems \cite{CJS:3,CJS:MHD}.
  
In \cite{BMN:global}, the authors prove global solvability of RNS equations in torus domains of any aspect ratios for sufficiently large rotation rate $\Om$  with certain class of external forcing. For initial data in the critical $H^{\frac12}(\bTt)$ space, global solvability is proven in \cite[Theorem 6.2]{Ch:Book} with the threshold for $\Om$ depending on $\vU_0$ (not just its norm). For RNS equations with fractional Laplacian diffusion, global solvability is proven in \cite{KY:number}. All three proofs are based on the global existence of the {\it exact} resonance approximation in the $\Omega\to\infty$ limit, and rely on certain decoupled structure of the limit equation between horizontal and vertical dimensions. Note however that there is no obvious decoupling in our NR approximations, and the approach in Remark \ref{re:global:sol} does not need the decoupling argument.

In numerical analysis, the results of \cite{HW:pa} showed the success of adapting singular limit theory to parareal algorithms. It employs a finite and discretised form of the infinite integral formula of  \cite{Schochet:JDE}, and its
GFD version in \cite{EmMa:CPDE}. Such adaptation always retains all exact resonances but due to its (numerically necessary) finiteness, near resonances are also retained in some form.

On the physical front,  NR-based study has led to exciting results in GFD, e.g. \cite{Newell, SmLe:NR, LS:zonal} which show that NR is closely linked to some notable features in large-scale GFD such as zonal flows. The work of \cite{WW:fast} quantifies exchanges of energies among various modes, which strongly suggests that one must keep the Rossby and Froude numbers finite in order to explain these phenomena. In another physical area, inertial waves are of particular significance in the study of Earth's fluid interior -- recall   the  Coriolis term in the RNS equations is responsible for inertial waves. Notably, such waves are identified in gravimetric data \cite{AL:Nature} which inspired modelling work in \cite{ALH:earth, RV:shell}. They also play an important role in turbulence theory \cite{Galtier} for GFD applications.

\subsection{Outline} \label{ss:outline}

Detailed motivation of our proposed NR approximation is in \S \ref{s:results}. 
Next in \S \ref{s:con:L}, inspired by \cite[Lemma 3.1]{BMN:global}, we prove a result of estimating restricted convolution (used in our proposed bilinear form) in terms of integer point counting.

From \S \ref{S:R:sum} onwards, we embark on a journey of bounding integer point count in two-sided sublevel sets of a function of $k$, namely $|\om^\vsigma_{nk(-n-k)}|\le\delta$ with given parameters $n,\vsigma,\delta$. Here, we see connections to group theory and topology. The essence of \S \ref{S:R:sum},  as suggested by the coined word ``anti-discretise'' in its title, is to convert integer point counting into volume estimates of sublevel sets. Challenges of integer point counting in the simplest 1D case were already emphasised following Theorem \ref{thm:2Dlike}. For the 2D case, the counting further requires  the set perimeter. For example, the finite set 
\be\label{2D:e}
\big\{\,(x_1,x_2)\in\bR^2: N< x_1\le 2N,\;|x_2|<x_1^{-2}\,\big\}
\ee
for any large $N$ contains $N$ integer points but its area is as small as $O(N^{-1})$. This is indeed the essence of the elementary result of Jarnik and  Steinhaus in \cite{JS:integer} for when the 2D set is the interior region of a {\it single} Jordan curve. We prove in Lemma \ref{lem:m:Jordan} an extended version allowing multiple disjoint Jordan curves  -- this means the set can be in either the interior or the exterior of an individual curve.   Finally, 3D problems require even more information. For example,
\be\label{3D:e}
\big\{\,(x_1,x_2,x_3)\in\bR^3: N< x_1\le 2N,\; x_2^2+x_3^2 <x_1^{-4}\,\big\}
\ee
contains $N$ integer points but its volume is  $ O(N^{-3})$ and the area of its boundary is $O(N^{-1})$. Therefore it is impossible to have 3D counterpart of the previous result in 2D, an issue that we will resolve with a straightforward treatment of the third dimension. 
   
   In \S \ref{S:vol}, we prove a sharp upper bound for the volume of sublevel sets $|\om^\vsigma_{nk(-n-k)}|\le\delta$ in the $k$ space. The volume integral is first rewritten in spherical coordinates. Then, as the triplet value is monotone in the azimuthal angle for half of the $2\pi$-period, we perform a crucial change of coordinates to replace the azimuthal angle with a new coordinate essentially equivalent to the triplet value. Since the aforementioned monotonicity must change due to periodicity, the Jacobian inevitably contains singular points. In fact, there are 4 singularities of ``$\frac1{\sqrt0}$'' nature (in the radial coordinate), which gives rise to elliptic integrals. Estimates of the integrals not only can depend on the ordering of the singular points, but can potentially degenerate if the singular points cross each other, i.e. when the order changes. Although there are 4$!$=24 ways of ordering, we will show there are essentially 2 cases using simple group theory. In fact, elliptic integrals with 4 singularities of $\frac1{\sqrt0}$ type are invariant under double transpositions (and identity permutation) of the singular points. They form the Klein 4-group $K_4$ which is a normal subgroup of the 4-permutation group $S_4$, so a coset of $K_4$ in $S_4$ is regarded as one equivalent case. Rather nicely, $S_4/K_4\cong S_3$ plays a role in proving Lemma \ref{lem:sort} since an affine transform of the singular points is used to make one of them fixed so that only the ordering of the other three matters.
   
   We will then have all the lemmas needed for proving 2D-like estimates Theorem \ref{thm:2Dlike}, which is at the start of  \S \ref{s:mainproofs}. Then, we prove  Theorems \ref{thm:NR} and  \ref{thm:local:e} using standard arguments.
   
    In Appendix \S \ref{s:j:c}  we prove Lemma \ref{lem:m:Jordan} for counting integer points enclosed by {\it multiple disjoint} Jordan curves, and only use elementary mathematics except the Jordan-Schoenflies theorem from topology. In Appendix \S \ref{s:e:i} we prove elementary estimates for elliptic integrals. In Appendix \S \ref{S:op}, we construct examples of direct integer point counting with {\it lower} bounds that evidence the optimality of the choices of bandwidth $\delta$, including the logarithmic attenuation.

We remark that if the spatial domain is $\bR^3$ instead, the concept of near resonance approximation and many of the above techniques still apply, in fact without the need to anti-discretise. But the dispersive nature of the inertial waves in the whole space  is stronger than in a compact domain, thus this case is usually treated very differently -- see e.g. \cite[Ch. 5]{Ch:Book}.


\section*{\bf Acknowledgement}

Cheng and Sakellaris are supported by the Leverhulme Trust (Award No.\,RPG-2017-098). Cheng is supported by the EPSRC (Grant No.\,EP/R029628/1). We thank Beth Wingate, Colin Cotter, David Fisher and Paul Skerritt for insightful discussion and valuable feedback.


 \section{\bf Motivation for Near Resonance Approximations}\label{s:results}

 First, a few basic details for clarity. 
We let bold italics symbols denote $\bR^3$ or $\bC^3$-valued functions in the $\bTt$ spatial domain which may or may not be time dependent as determined by the context. A smooth function is infinitely many times differentiable.

The $L^2$ inner product of complex-valued functions is
\[
\ip{f(x),g(x)}=\int_{\bTt}f(x)\cdot \overline{g(x)}\,\rd x.
\]
The dot product is unrelated to conjugate, e.g.  
\(
(f_1,f_2,f_3)^T\!\cdot(g_1,g_2,g_3)^T=f_1g_1+f_2g_2+f_3g_3
\).

 For a function $g(x)\in L^p(\bTt)$ for some $p\in(1,\infty)$, its Fourier series expansion is 
\be\label{Fourier:s}
g (x) = \sum_{n \in \bZt }  g_n\,e^{\ri \cn \cdot x}  \,,
\ee 
with coefficients  
\be\label{Fourier:c}
g_n : = \dfrac{\ip{g(x),e^{\ri \cn \cdot x}}}{\ip{e^{\ri \cn \cdot x},e^{\ri \cn \cdot x}}}=\dfrac{1}{|\bTt|}\int_{\bTt}g(x)\,e^{-\ri \cn \cdot x}\,\rd x\,.
\ee
Combining these definitions gives this version of Parseval's theorem/identity:
\be\label{P:thm}
\ip{f(x),g(x)}=|\bTt|\sum_{n \in \bZt }    f_n \cdot \overline{g_n}\,.
\ee

For the Leray projection $\leray$ given in the Introduction, classical theory of elliptic operators in Lebesgue spaces shows it is a bounded operator on $L^p(\bTt)$ for any $p\in(1,\infty)$.
  It is straightforward to show self-adjoint property $\ip{\leray\vu,\vv}=\ip{\vu,\leray\vv}$ for $\bC^3$-valued $\vu\in L^p(\bTt)$ and $\vv\in L^{p\over p-1}(\bTt)$. Operator $\cL$  is skew-self-adjoint w.r.t. the $L^2$ inner product, i.e.,
\be\label{L2:pr}
\int_{\bTt}\ip{\vu,\cL\vu}\,\rd x=0\,,\quad\text{\; for $\bR^3$-valued \, $\vu\in L^2(\bTt)$}.
\ee 
Also note the classic property $\langle B(\vu,\vu'),\vu'\rangle=0$ for $\bR^3$-valued, div-free $\vu,\vu'\in H^1(\bTt)$. 
 
In an equivalent form, we remove the mean drift from the velocity field as follows and assume zero-mean velocity throughout the article. By definition of $\leray$ and identity
\be\label{B:another}
\vU \cdot \nabla \vV=\nabla\cdot(\vU \otimes \vV),\qua{\; for div-free}\vU
\ee
where $H^1$ regularity of $\vU,\vV$ is assumed, we find $B(\vU,\vV)$ is of zero mean. (This still holds for $\vU,\vV\in L^2(\bTt)$ in which case one regards $B(\vU,\vV)$ as a functional on smooth functions. See Definition \ref{def:L2wk} where setting $\vh\equiv 1$ gives the mean. Also see Definition \ref{def:wk} with $\vpsi\equiv 1$.) Thus, for $U_0(t)=\frac1{|\bTt|}\int_{\bTt}\vU(t,x)\,\rd x$,  integrate \eqref{NS} to have ${\rd\over \rd t}U_0 =\Omega\Jr U_0$. Then upon substitution 
\(
\vU(t,x)=\vU'\Big(t,\,x-\int_0^t U_0(s)\,\rd s\Big)+U_0(t)
\) 
in \eqref{NS} followed by dropping the primes, we obtain \eqref{NS} again but with an additional invariant 
\[
\int_{\bTt}\vU(t,x)\,\rd x \equiv 0, 
\quad\text{\; for all }t\ge0.
\] 

For $s\in\bR$, introduce $s$-th order pseudo-differential operator $\pD^s:=(\sqrt{-\Delta})^s$ so that
\[
\pD^s g(x)=\sum_n|\cn|^sg_n\,e^{\ri \cn \cdot x}\,,\quad\text{\; for \,}g\in H^s(\bTt).
\]
 For ease of reading, define shorthand notation
\[
\pDn s\cdot:=\| \pD^s\cdot\|_{L^2(\bTt)}.
\]
We will use it only on zero-mean functions, in which case $\|\cdot\|_s$ and $\|\cdot\|_{H^s(\bTt)}$ are equivalent. 

We let $A\lesssim_{C_1,C_2,\ldots} B$  denote $A \leq C B$, for a nonnegative constant $C$ that depends on quantities $C_1,C_2,\ldots$.   The specific relation between $C$ and $C_1,C_2,\ldots$ or the specific value of $C$ may change from inequality to inequality. For simplicity, we further omit the dependence on the domain $\bTt$ in this notation -- note Remark \ref{p:domain}.

\subsection{Transformed bilinear form}
Introduce the operator exponential 
\(
\etau
\)
 for $\tau\in\bR$ so that, for $\bC^3$-valued, div-free ${\vV_{\!\!in}\in L^2(\bTt)}$, the function  $\vV(\tau)=\etau {\vV_{\!\!in}}$  is the weak solution to 
\[
\partial_{\tau} \vV = \cL\vV\qqua{with} 
  \vV(0)= {\vV_{\!\!in}}.  
\] Existence and uniqueness of solution in  $\scc^\infty(\bR;L^2(\bTt))$  follows from Proposition \ref{prop:L}.

For a solution $\vU(t,\cdot)$ to the RNS equation, perform transformation 
\be\label{vdp}
\vu(t,\cdot) = \big(\emtau\vU(t,\cdot)\big)\Big|\tauO
\ee 
(also c.f.  \eqref{cr:eig}). Then by the usual chain rule,
\begin{align*}
{\pa_t}\vu(t)&= \Omega\Big({\pa_\tau}\big(\emtau\vU(t)\big)\Big)\Big|\tauO+\Big(\emtau\pa_t\vU(t)\Big)\Big|\tauO\\
&=\Big( \emtau\pa_t\vU(t)-\Omega\cL\emtau\vU(t)\Big)\Big|\tauO\,.
\end{align*}
It is straightforward to show that $\etau\vu$ is div-free and $\etau$ commutes with $\cL$ and $\Delta$,
so if $\vU$ with sufficient regularity  satisfies the RNS equation \eqref{NS}, then $\vu$ satisfies  
\be\label{m:NS}
\partial_t \vu +B( \tau; \vu, \vu)\Big|\tauO = \mu \Delta \vu, \qua{ \; with }\vu(0,\cdot)= \vU(0,\cdot),
\ee
where the transformed bilinear form
\be\label{def:m:B}
B({\tau}; \vu, { \vv}):=\emtau B(\etau\vu,\etau\vv).
\ee

 From now on,   a symbol such as $B$ that denotes a bilinear form can also denote its transformed version, with the distinction that we explicitly write the dependence of the latter on the extra variable $\tau$ (which is sometimes referred to as the ``fast time'').  
 
 \subsection{Eigen-basis representation}\label{ss:eig}
For any $n\in \bZtno$ and associated domain-adjusted wavevector $\cn$,  there exists (not uniquely) an orthonormal,  right-hand oriented basis of $\bR^3$  in the form $\frac{\cn}{|\cn|},\tilde\rvec_n,\tilde{\tilde\rvec}_n$ so that $\frac{\cn}{|\cn|}\times\tilde\rvec_n=\tilde{\tilde\rvec}_n$ and $\frac{\cn}{|\cn|}\times\tilde{\tilde\rvec}_n=-\tilde\rvec_n$. Then for  
\[
\rvec_n^-:=\tfrac{1}{\sqrt2}\left(\tilde\rvec_n+ \ri\,\tilde{\tilde\rvec}_n\right),\quad \rvec_n^+:=\overline{\rvec_n^-},
\] 
the vectors $\frac{\cn}{|\cn|},\rvec^-_n,\rvec^+_n$ 
form an orthonormal basis of $\bC^3$ (so e.g. $\rvec^-_n\cdot\overline{\rvec^+_n}=0$).  
 Next, introduce
\[
\eig\ns:= \rvec\ns\, e^{\ri \cn \cdot x} \, ,\quad\;\;\sigma\in\{+,-\},\;n\in \bZtno.
\]
Since combining identity $\vv=\Delta^{-1}\Delta\vv$ for zero-mean $\vv$ with identity $\Delta=\nabla\nabla\cdot-\nabla\times(\nabla\times)$ shows $\leray\vv= -\Delta^{-1}\nabla\times(\nabla\times\vv)$, we show that
 \[
\cL\vU=-\Delta^{-1}\nabla\times\left(\nabla\times(\Jr\vU)\right)=\Delta^{-1}\nabla\times\pa_3\vU,
\quad\text{\,  for  div-free, zero-mean \,}\vU\, .
\]
 Next, it is easy to see $\frac{\cn}{|\cn|}\times\rvec\ns=\sigma\ri\rvec\ns$ 
(so, in fact, $\eig\ns$ is also eigenfunction of  $\nabla\times$). Note that orthogonality $\cn\cdot \rvec\ns = 0$  implies  div-free condition $\nabla\cdot\eig\ns=0$ and that $n\ne0$ implies $\eig\ns$ is zero-mean, so by the above identity, we obtain the following   eigen-pair relation for operator $\cL$,
 \be\label{om:eig}
  \cL\eig\ns = \ri\,\om\ns\eig\ns\,,\qua{\, for dispersion relation}\om\ns:=\sigma {\cn_3\over|\cn|}, \quad\forall\,\sigma\in\{+,-\},\;n\in \bZtno.
\ee
 
 Introduce eigen-projection, as an operator acting on $\bC^3$-valued functions,  
 \be\label{def:cP}
 \cP\ns\vu:=\dfrac{\ip{\vu,\eig\ns}}{\ip{\eig\ns,\eig\ns}}\eig\ns=\big(u_n\cdot\overline{\rvec\ns}\,\big)\,\rvec\ns \,e^{\ri \cn \cdot x}\,,\qua{ for}n\in\bZtno,  
 \ee
 with the $u_n$ notation used in the same fashion as Fourier series \eqref{Fourier:s}. Also define \[\cP^\sigma_0\vu\equiv {\boldsymbol 0}.\]
 Directly from definition, we have orthogonality property
 \be\label{otg}
 \cP\ns\cP_{n'}^{\sigma'}=
 \begin{cases}\cP\ns,&\text{ if }(n,\sigma)=(n',\sigma'),\\
 0,&\text{ otherwise},
 \end{cases}
 \ee
 and  adjoint property
 \begin{align}
\label{ad2} \ip{\cP\ns\vu,\vv}&=\ip{\vu,\cP\ns\vv}=\ip{\cP\ns\vu,\cP\ns\vv}.
 \end{align}

Recall $\frac{\cn}{|\cn|},\rvec^-_n,\rvec^+_n$ (for $n\ne\vec0$) form an orthonormal basis and  recall $\nabla\cdot\vu=0$ iff $\cn\cdot u_n=0$. Therefore 
\[
\cP\ns\leray=\cP\ns=\leray\cP\ns\,,\qua{ for}n\in\bZt,
\]
and 
   \be\label{Fr:P}
 \cP_n^-\vu+\cP_n^+\vu= \leray\big(u_n\,e^{\ri \cn \cdot x}\big),\qua{ for}n\in\bZtno,
 \ee
   for any $\bC^3$-valued $\vu\in L^2(\bTt)$ that is not necessarily div-free.
 Thus by the completeness of the Fourier basis, we have expansion
 \be\label{id:P}
 \leray\vu=\sumn \cP\ns\vu\,,\quad\text{ \, for $\bC^3$-valued, zero-mean \,}\vu\in L^2(\bTt).  
 \ee
Here and below,  $\displaystyle\sumn$ stands for $\displaystyle\sum_{n\in\bZt}\sum_{\sigma\in\{+,-\}}$. Then, by \eqref{otg}, \eqref{ad2}, we have
\be\label{cP:norm:sum}
\| \leray\vu\|^2_{L^2}=\sumn \Big\|\cP\ns\vu\Big\|^2_{L^2}\,,\quad\text{ \, for zero-mean \,}\vu\in L^2(\bTt).  
\ee

 \begin{proposition}\label{prop:L}
 For any $\bC^3$-valued, div-free, zero-mean function $\vu,\vv\in L^2(\bTt)$, we have
 \be\label{e:cL}
 \cL \, \cP\ns\vu={\ri\,\om\ns} \cP\ns\vu= \cP\ns \cL\,\vu,
 \ee
 the eigen-expansion of  the operator exponential
 \be\label{e:eig:exp}
 \etau\vu=\sumn e^{\ri\,\om\ns\tau}\cP\ns\vu,
 \ee
 and the adjoint property
 \be
 \label{ad3} \ip{ \etau\vu,\vv}=\ip{\vu,\emtau\vv}.
 \ee
 \end{proposition}
 \begin{proof}
By \eqref{om:eig}, we show the first equality of \eqref{e:cL}. Combining it with \eqref{otg}, \eqref{id:P}, we show the second equality of \eqref{e:cL}. 
 Recalling the definition of    $\etau$, we use \eqref{id:P}, \eqref{e:cL} to prove \eqref{e:eig:exp}. Combining the latter with \eqref{ad2}, we see adjointness \eqref{ad3} naturally follows.
\end{proof} 
Combining this with \eqref{def:cP} shows operators $\pD^s$, $\cP\ns$ and $\etau$ all commute. Then  
by \eqref{cP:norm:sum},  
\be\label{norm:pr}
\|\pD^s g(x)\|_{L^2(\bTt)}=\|\etau\pD^s g(x)\|_{L^2(\bTt)}=\|\pD^s \etau g(x)\|_{L^2(\bTt)},\qquad\forall\,\tau\in\bR.
\ee
Also, combining \eqref{ad3} with \eqref{def:m:B} shows
 \be\label{ip:uvw}
 \ip{B(\tau;\vu,\vv),\vw}= \ip{B(\etau\vu,\etau\vv),\etau\vw},\quad \text{for $\bC^3$-valued, div-free, zero-mean }\vu,\vv,\vw.
 \ee
 
Finally, the convolution condition $n+k+m=\vec0$ is equivalent to the fact that
\be\label{conv:cond}
\cP_{-n}^{\sigma_1}B(\cP_k^{\sigma_2}\vu,\cP_m^{\sigma_3}\vv)=\boldsymbol 0\,,\qua{ if}n+k+m\ne\vec0.
\ee
Thus, for a mapping $F$ from $(\bZtno)^3$ to a Banach space, let shorthand notation $\displaystyle\sumcnkm F(n,k,m)$ denote $\displaystyle\sum_{n\in\bZt}\sum_{k\in\bZt}F(n,k,-n-k)$ and let $\displaystyle\sumnkmsc$ denote $\displaystyle\sumcnkm\sum_{\vsigma\in\{+,-\}^3}$.
\begin{remark}\label{re:sum}
The {\bf absolute} convergence of the above sums in the Banach space (the codomain of $F$) implies the commutability of the nested sums, and implies symmetry with respect the choice of running indices, for example,
\[\begin{aligned}
\displaystyle\sumcnkm F(n, k, m)=\sum_{n\in\bZt}&\sum_{m\in\bZt}F(n, -n-m, m)=\sum_{m\in\bZt}\sum_{n\in\bZt}F(n,-n-m,m),\text{\; etc.,}\\
\text{thus \quad} \sumcnkm F(n, k, m)&\;=\;\sumcnkm F(n, m, k)\;=\;\sumcnkm F(k, m, n),\text{\; etc.}
\end{aligned}\] 
Without absolute convergence, even the simple sum $\sum_{n\in\bZ}\big(\sum_{k=-1}^1k\big)$ is not commutable. In this article, all nested sums converge absolutely in a Banach space that is evident from context. 
 Also   note the completeness of Fourier basis in $H^s(\bTt)$ spaces.
\end{remark}

For the reader's reference, we also expand \eqref{vdp} for div-free, zero-mean $\vu,\vU$  as follows:
 \be\label{cr:eig}
 \vu(t,x) = \sumn \exp(-\ri\,\om\ns\Omega t)\,\cP\ns\vU(t,x) \; \; \iff \;\; \vU(t,x) = \sumn \exp(\ri\,\om\ns\Omega t)\,\cP\ns\vu(t,x).
 \ee

 \subsection{Motivating NR approximations}\label{ss:motivate}
 Apply expansion \eqref{e:eig:exp} on all three operator exponentials in the transformed bilinear form \eqref{def:m:B}. Then noting the (bi)linearity of $B(\dcd)$ and convolution condition \eqref{conv:cond}, we find   
\begin{align}
\label{B:full1}B(\tau;\vu,\vv)&=\sumnkmsc\cP_{-n}^{\sigma_1}B(\cP_k^{\sigma_2}\vu,\cP_m^{\sigma_3}\vv)\exp(\ri\omvs\tau),\quad\text{ for  div-free, zero-mean }\vu,\vv,
\end{align}
 with the triplet value given by
\[
 \omvs= -\om_{-n}^{\sigma_1}+\om_k^{\sigma_2}+\om_{m}^{\sigma_3} =\sigma_1 \frac{ \cn_3  }{ |{ \cn}|}+\sigma_2  \frac{\ck_3  }{ |{ \ck}|} +\sigma_3 \frac{ \cm_3  }{ |{\cm}|}
 \]
   (note $B(\dcd)$ is zero-mean so the $n=\vec 0$ term is irrelevant) which was already defined in \eqref{om:star}.
Noting the heuristic bound $\vu_t\sim O(1)$ independent of $\Om\gg1$ as seen from the transformed RNS equation \eqref{m:NS}, we  argue that \eqref{B:full1} suggests the applicability of  \eqref{RLL} with $\omega^*=\omvs$
 on the analysis of \eqref{m:NS}. The smaller the triplet value  is, the more important role it plays in the dynamics.

Motivated by such observation, we propose approximations $\wt B(\dcd)$ and $\wt B(\tau;\dcd)$ that retain and discard eigen-mode interactions based on the smallness of $\om^\vsigma_{nkm}$. This prompts the definition of NR set $\cN$, using restriction criterion based on bandwidth, as
in \eqref{def:cN}. For convenience, our version of the NR set $\cN$ is defined without involving the triplet of signs $\vsigma$. An alternative definition of $\cN$ can involve membership of $(n,k,m,\vsigma)$. Most of the results would not be affected by this but their proofs would not have access to the $\vsigma$-independent bilinear form \eqref{wt:B:full2}, hence lengthier.

 Then define the approximation of the transformed bilinear form \eqref{B:full1} by restriction of $\ind_\cN$, 
\begin{align}
\label{wt:B:full1}\wt B(\tau;\vu,\vv) &:= \sumnkmsc\cP_{-n}^{\sigma_1}B(\cP_k^{\sigma_2}\vu,\cP_m^{\sigma_3}\vv)\exp(\ri \omvs\tau)\,\ind_{\cN}({n,k,m}),\\
\intertext{and naturally, define its pre-transformed version}\label{wt:B:def}\wt B(\vU,\vV)&:=\sumnkmsc\cP_{-n}^{\sigma_1}B(\cP_k^{\sigma_2}\vU,\cP_m^{\sigma_3}\vV)\,\ind_{\cN}({n,k,m}).
\end{align}
 Noting eigen-expansion \eqref{e:eig:exp}, we have the following relation analogous to \eqref{def:m:B},
\be\label{m:B:nr}
\wt B({\tau}; \vu, { \vv})=\emtau \wt B(\etau\vu,\etau\vv).
\ee

Finally, in the eigen-expansion \eqref{wt:B:def},  suppress the running indices $\vec\sigma$ via \eqref{Fr:P} and the (bi)linearity of $B(\dcd)$,  which results in the $\vsigma$-independent bilinear form \eqref{wt:B:full2} 
and the PDE \eqref{NR} from the Introduction where we named them the NR approximation of the RNS equations. 

By the same argument as for the RNS equations, the NR approximation \eqref{NR} is equivalent to its transformed version
\be\label{m:NR}
\partial_t \tvu +\wt B( \tau; \tvu, \tvu)\Big|\tauO = \mu \Delta \tvu, \qua{ \; with } \tvu(0,\cdot)= \tvU(0,\cdot),
\ee
where $\tvu(t,\cdot) = e^{-\Omega t\cL}\tvU(t,\cdot).$

Note the identity for $\bC^3$-valued, smooth functions $\vu,\vv,\vw$
\be
\label{tri:e}\big\langle\wt B(\vu,\vv),\vw\big\rangle=|\bTt|\sumcnkm  (u_k\cdot\cm\,\ri)\, v_m\cdot (\overline w)_n\,\ind_{\cN}({n,k,m}),\quad\text{ with div-free }\vw\,,
\ee
 where $(\overline w)_n$ denotes the Fourier coefficients of $\overline \vw$. It of course also holds for $\ind_\cN\equiv1$, i.e. for $B$.

 It is physically preferable for $\wt B$ to ``inherit'' from $B$ the following properties. 

  \begin{proposition}\label{prop:Bt}
  Consider bilinear form $\wt B(\dcd)$ defined in \eqref{wt:B:full2} with a \emph{generic} $\cNg\subset\big(\bZt\big)^3$. For $\bR^3$-valued, div-free, zero-mean smooth functions $\vU,\vV$, 
  \be\label{realreal}
\text{$\ind_{\cNg}(\dcd,\cdot)$ being an even function \; implies \; $\wt B(\vU,\vV)$ is \,$\bR^3$-valued.}
\ee
We also have
\be\label{c:p}
\text{symmetry $\ind_{\cNg}({n,k,m})=\ind_{\cNg}({m,k, n})$ \; implies \; }\;\big\langle\wt B(\vU,\vV),\vV\big\rangle_{}=0 .\ee
The same properties still hold with $\wt B(\dcd)$  replaced by $\wt B(\tau;\dcd)$ for any $\tau\in\bR$.
  \end{proposition}
\begin{proof} First, due to \eqref{m:B:nr}, \eqref{ad3} and the fact that $\etau$ maps the set of $\bR^3$-valued, div-free, zero-mean functions onto the set itself, it suffices to just prove  \eqref{realreal}, \eqref{c:p}.

By taking the complex conjugate of \eqref{Fourier:s}, we find that
 $g(x)$ is $\bR^3$-valued iff $\overline{g_n}= g_{-n}$ for all $n$.
  Then take the complex conjugate of \eqref{wt:B:full2} with real-valued $\vU,\vV$ to prove \eqref{realreal}.

By self-adjointness of $\leray$, the assumed $\vV=\leray\vV$ and  \eqref{tri:e}, noting $\overline\vV=\vV$, we  show
$
\big\langle\wt B(\vU,\vV),\vV\big\rangle_{}=|\bTt|\sumcnkm  (U_k\cdot\cm\,\ri)\, V_m\cdot V_n\,\ind_{\cNg}({n,k,m})$. 
Switching $m, n$ and adding it back, we prove
\eqref{c:p} by  incompressibility $U_k\cdot (\cn+\cm)=0$ and sum symmetry as in Remark \ref{re:sum}.
\end{proof}


\section{\bf  Restricted Convolution and Integer Point Counting}\label{s:con:L}

The restriction of 3-mode interactions to the NR set $\cN$ leads to the need for estimating restricted convolution. The lemma below is inspired by \cite[Lemma 3.1]{BMN:global} and extends it to a version allowing different factors in the triple product. Also see \cite[Lemma 6.2]{Ch:Book}.

\begin{lemma}\label{restricted:L}
Let  $\ind_\cN( \cdot, \dcd)$ denote the characteristic of set $\cN\subset\big(\bZt\big)^3$ and suppose it is symmetric with respect to all permutations in its arguments. Define
\[
\cNp:=\big\{(n,k,m)\in\cN: |\cn|\ge |\ck| \ge |\cm|\big\}.
\] 
Suppose there exist a constant $\beta\in\big[0,\frac32\big]$ and a constant $C_\cN$ so that the counting condition
\begin{equation}\label{ci}
   \sum_{ \substack{ k \in \bZt }} \ind_{\cNp}(n,k,-n-k)    \le C_\cN \, |\cn|^{\beta},\qquad \forall n \in \bZtno
\end{equation}
holds. 
Then there exists a constant $C$ so that for zero-mean $\vu,\vv\in H^{\beta\over2}(\bTt)$ and $\vw\in L^2(\bTt)$,  
\be\label{restricted}
\begin{aligned}&|\bTt|^{\frac32}\sumcnkm  | u_n|\,| v_k|\,| w_m |  \,\ind_\cN( n, k, m)
 \le C \sqrt {C_\cN}    \left(  \pDn0\vu\,\pDn {\frac {\beta} 2}\vv+\pDn {\frac {\beta} 2}\vu \,\pDn0\vv \right)\pDn0\vw.
 \end{aligned}
 \ee
 The estimate still holds if we permute  $u,v, w$ on the left-hand side. 
\end{lemma}
When applied to standard product $\ip{fg,h}=\int_{\bTt}(fg) \, \overline h\,\rd x$, the counting condition \eqref{ci} holds for $\beta=\frac32$, so the  upper bound in \eqref{restricted} requires $\frac32$ derivatives on either $f$ or $g$. This is in fact a consequence of combining \eqref{tri:inf:fi}(ii) with interpolation, and this is different from the  commoner upper bound $\pDn0{fg}\lesssim\pDn{\frac32+\gamma}f\pDn0g$ for $\gamma>0$ by \eqref{tri:inf:fi}(i). In contrast, as  $\beta=1$ is used in the proof of 2D-like estimates Theorem \ref{thm:2Dlike},  restriction to $\cN$ ``gains'' half a derivative.

\begin{proof}
 It suffices to consider scalar-valued functions. 
Also note that, 
since the following proof only involves absolute values of the summands, 
all nested sums are commutable.

Define\[
\eta (a, b, c):=\begin{cases}
1,\;\;&\text{ if }\, a\ge b \ge c,\\
0,\;\;&\text{otherwise},
\end{cases}\] 
so
\(
\ind_{\cNp}(n,k,m)=\ind_\cN(n,k,m)\,{\eta (|\cn|,|\ck|,|\cm|)}
\). 
 For any wavevectors $n,k,m$, we have
\[\begin{aligned}
1&\le
 \eta (|\cn|,|\ck|,|\cm|) + \eta (|\ck|,|\cn|,|\cm|)+ \eta (|\cm|,|\ck|,|\cn|)\\
 &+\eta (|\cn|,|\cm|,|\ck|) + \eta (|\ck|,|\cm|,|\cn|)+ \eta (|\cm|,|\cn|,|\ck|).
\end{aligned}
\]
Then the summand in  \eqref{restricted} is bounded by the corresponding six terms $|u_nv_kw_m |\ind_{\cNp}(n,k,m)+|u_nv_kw_m |\ind_\cN(n,k,m){\eta (|\ck|,|\cn|,|\cm|)}+|u_nv_kw_m |\ind_\cN(n,k,m){\eta (|\cm|,|\ck|,|\cn|)}+\ldots$ 
Next, permute $n,k,m$ in each term in such a way that the $\eta $ function appears as $\eta (|\cn|,|\ck|,|\cm|)$. For example, switch $n,k$ in the second term, which results in $|u_kv_nw_m |\ind_\cN(k, n,m){\eta (|\cn|,|\ck|,|\cm|)}$. Also 
note the assumed symmetry of $\ind_{\cN}(\dcd,\cdot)$ and symmetry of convolution sum by Remark \ref{re:sum}.  After these steps, all six terms can be rewritten so that $\ind_{\cNp}(n,k,m)$
 is a common factor: 
\be\label{6parts} \sumcnkm | u_n v_k w_m | \,\ind_\cN( n, k, m)\le
\sumcnkm  \big(| u_n v_k w_m |+  | u_k v_n w_m |+  \text{four terms} \big)\ind_\cNp( n, k, m).
\ee

Next introduce the oblate annulus
\[
 \An_i:=\big\{k\in\bZt:2^{i-1}\le |\ck| < 2^i\big\},\qquad i\in\bZ^+.
 \] 
Then 
  the ordering $|\cn| \ge  |\ck| \ge |\cm|$ together with Littlewood-Paley type localisation $k\in\An_i$
 implies  $n=-k-m\in\An_i\cup\An_{i+1}$, effectively localising the sum to be close to the ``diagonal'' where $|\cn|$ and $|\ck|$ are comparable. Consider the first term on the right-hand side of \eqref{6parts}: 
\begin{align}
\sumcnkm  | u_n v_k w_m | \,\ind_{\cNp}( n, k, m)
\label{tri:step0}& 
= \sum_{i \in \bZ^+} \,\sum_{k\in\An_i}\Big(\sum_{n\in\An_i\cup\An_{i+1}} | u_n v_k w_{-k-n} |\ind_{\cNp}\Big)  \\
&\label{tri:step1} = \sum_{i \in \bZ^+} \,\sum_{n\in\An_i\cup\An_{i+1}} \Big(|u_n |\sum_{k\in\An_i}  | v_k   w_{-k-n} |\ind_{\cNp}\Big) .  
\end{align}
Here and below, the shortened  $\ind_{\cNp}$ notation is always understood as $\ind_{\cNp}( n, k, -k-n)$.
Define
\[
a_i:=\sum_{\substack{ n\in\An_i\cup\An_{i+1} }} \Big(\sum_{\substack{ k\in\An_i    }}   | v_k   w_{-k-n}   |  \ind_{\cNp}\Big)^2.
\]and apply Cauchy-Schwartz inequality on the combined $(i,n)$-sum
\begin{align} 
 & \eqref{tri:step1} \,
  {\leq}   \sqrt{\sum_{i \in \bZ^+}\sum_{\substack{ n\in\An_i\cup\An_{i+1}   }} | u_n |^2    }\sqrt{\sum_{i \in \bZ^+}a_i }
 \label{estimate1}   .
\end{align}
 
It remains to estimate:
\begin{align*} 
a_i\,&{\leq} \sum_{\substack{ n\in\An_i\cup\An_{i+1}  }} \bigg[ \Big(\sum_{\substack{ k\in\An_i  }} | v_k   w_{-k-n}  |^{2}   \Big) \Big(\sum_{\substack{ k\in\An_i }} \ind_{\cNp}^2   \Big) \bigg]\\[2mm]
& \le C_\cN  \, \sum_{\substack{  k  \in   \An_i  }}\sum_{\substack{ n\in\An_i\cup\An_{i+1}    }} |\cn|^{\beta} | v_k   w_{-k-n}  |^2    \qquad\qquad \text{ (by $\ind_{\cNp}^2=\ind_{\cNp}$ and \eqref{ci})}  \\[2mm]
& \lesssim C_\cN   \sum_{\substack{ k \in    \An_i  }}\bigg[|\ck|^{{\beta}} |v_k|^2 \sum_{\substack{ n  \in\An_i\cup\An_{i+1}   }} |    w_{-k-n}  |^{2}  \bigg]  \qquad  \text{ (by  definition of }\An_i) \\[2mm]
&\lesssim C_\cN\sum_{\substack{ k  \in\An_i }} |\ck|^{\beta} |v_k|^2\,|\bTt|^{-1}\|\vw \|_{L^{2}}^2
\end{align*}
where in the last step we relaxed the sum in $n$ to all $n\in\bZt$ and used Parseval's theorem.

Finally, combining the above and \eqref{estimate1}, we use  Parseval's theorem on $\vu$ and $\pD^{\beta\over2}\vv$ to prove:
\be\label{re:b1}
\begin{aligned} 
 & \sumcnkm | u_n v_k w_m |  \, \ind_{\cNp}(n,k,m) \lesssim \sqrt{ C_\cN}\,\pDn0\vu \,\pDn {\frac {\beta} 2}\vv \, \pDn0\vw\,|\bTt|^{-\frac32}\,.
\end{aligned}
\ee

Recall that in the right-hand side of \eqref{tri:step0}, for a fixed $i$, the other two indices satisfy $2|\cn|\ge|\ck|$,  allowing us to insert after \eqref{tri:step0}
\[
\lesssim \sum_{i \in \bZ^+}\, \sum_{k\in\An_i}\sum_{n\in\An_i\cup\An_{i+1}}\Big| \big(|\cn|^{\beta/2}u_n\big) \big(|\ck|^{-\beta/2}v_k\big) w_{-k-n} \Big|\,\ind_{\cNp} \,,
\] 
and immediately let
$|\cn|^{\beta/2}u_n$ be denoted by $u_n'$ and $|\ck|^{-\beta/2}v_k$ by $v_k'$. Then we perform the same steps (except with the primes) afterwards until the right-hand side of \eqref{re:b1}; therefore we prove
\[
\sumcnkm | u_n v_k w_m | \,\ind_{\cNp}(n,k,m)   \lesssim \sqrt{C_\cN}\,\pDn {\frac {\beta} 2}\vu \, \pDn0\vv \,\pDn0\vw \,|\bTt|^{-\frac32}\,.
\]
The essence of this estimate and \eqref{re:b1} is that every one of the six terms inside the right-hand side parentheses of \eqref{6parts} can be bounded by the product of three norms where one has the freedom to apply the $H^{\frac {\beta} 2}$ norm on either the factor with $n$ index or the factor with $k$ index. Since at least one of $u$ and $v$ is equipped with either $n$ or $k$ index, the proof is complete.
\end{proof}


\sectionb{``Anti-discretise'' from Integer Point Counting to Integrals}\label{S:R:sum}

With motivation already given in \S\ref{ss:outline}, we start the mathematics right away.

Recall definition \eqref{def:cN} of the NR set $\cN$ and the type of integer point counting used in assumption \eqref{ci} of Lemma \ref{restricted:L}. They inspire us to fix wavevector $n\ne0$,  and define mapping $F_\sigg:\bR^3\backslash\{0,-n\}\mapsto\bR$ for any given pair $(\sigma_1,\sigma_2)\in\{+,-\}^2$ as follows (recall $\cn_3=n_3$, etc.)
\be\label{fpm:Car}
F_\sigg( k):= \sigma_1 \frac{n_3}{|\cn|} + \sigma_2 \frac{k_3}{|\ck|} + \frac{m_3}{|\cm|} ,\qua{ \; for }  m=-n-k.
\ee

Then, for $\cNp$ defined in Lemma \ref{restricted:L}, the condition $\ind_{\cNp}(n,k,-n-k)=1$ implies $k\notin\{0,-n\}$ and at least one of $|F_{n,\pm,\pm}( k)|\le\delta(n,k,-n-k)$ must hold.  Also, $\ind_{\cNp}(n,k,-n-k)=1$ implies $| \cn+\ck| \le | \ck| \le | \cn|$   which further implies  $\frac12| \cn| \leq |\ck| \leq |\cn|$ and $\delta(n,k,-n-k)=\delta(n,n,n)$ under \eqref{de:max}.
So, we   define the following continuous set of real numbers (not just integer points)
\be\label{sym2}
V_\sigg:=\big\{  k\in \bR^3\backslash\{0,-n\} \, : \, |F_\sigg( k)| \le \delta(n,n,n)\text{\,\, and \,\,}\tfrac12|\cn| \leq | \ck| \leq |\cn|\big\},
\ee
and relax the integer counting  in the left-hand side of assumption \eqref{ci} as
\be\label{N:le:V}
 \sum_{{ k \in \bZt }} \ind_{\cNp}(n,k,-n-k)\le \sum_{ (\sigma_1,\sigma_2)\in\{+,-\}^2}\;\sum_{ k \in \bZt}\ind_{V_\sigg}( k).
\ee
Apparently, we can  use the following working assumption for as far as counting is concerned: 
\[
\text{for fixed $n$, the bandwidth $\delta$ used in \eqref{sym2} is constant.}
\]

For simplicity, we omit the subscripts $\sigg$ in $F$ and $V$ until Theorem \ref{thm:int:vol}.

Since $V$ is Lebesgue measurable,  the integral  $\int_{\bR^3}\ind_{V}(k)\,\rd k$ is regarded as  volume $vol(V)$. The heuristic is that integer point counting resembles the Riemann sum of an indicator function which is then linked to the corresponding integral -- note however, all integrals here are Lebesgue integrals. Thus,  much of the work in \S\ref{S:R:sum} is to bound 
$\sum_{ k \in \bZt}\ind_V( k) -\int_{ k \in \bR^3}\ind_V( k)$. 
The estimate of the volume integral itself is given in \S\ref{S:vol}.

 \subsection{Reduction of 3D integer point counting to 2D}We will use Lemma \ref{lem:m:Jordan} for estimating number of {\it planar} integer points. Priori to that, we first ``anti-discretise'' in the vertical dimension since it plays a different role in the dispersion relation. For $k=(k_1,k_2,k_3)$, let
 \[
  k_H = (k_1,k_2) \quad\text{ and } \quad\ck_H=\big(\,{k_1\over\sL_1},{k_2\over\sL_2}\,\big).
  \]  
 Since $\dfrac {\partial F } {\partial k_3}  =  \sigma_2 \dfrac {| \ck_H |^2} {|\ck|^3} -  \dfrac { |\cm_H|^2 } {|\cm|^3}$, we have, 
\[\begin{aligned}
&\dfrac{\pa F}{\pa k_3}=0 \qua{implies} |\ck_H|^{4 \over 3}(k_3^2+2k_3n_3+n_3^2+|\cm_H|^2)=|\cm_H|^{4 \over 3}(k_3^2+|\ck_H|^2).
\end{aligned}
\]
At any fixed $n\in\bR^3\backslash\{0\}$ and $k_H\in \bR^2\backslash\{0,-n_H\}$ (thus fixed $m_H\in \bR^2\backslash\{0\}$), $F$ is defined for all $k_3\in\bR$ and what is in above shows $\frac{\pa F}{\pa k_3}$ either vanishes at most two $k_3\in\bR$, or vanishes everywhere. Then $\bR$ is the union of no more than three $k_3$-monotonic intervals for $F$, so $\big\{k_3\big| \, |F( k)| \le\delta\big\}$ is the union of at most three intervals (an isolated point is counted as an interval). In each interval, the number of integers $k_3$ and the integral of the indicator function in $k_3$ differ by at most $1$.   But even for
 $k_H\in\{0,-n_H\}$, since $F$ is defined in $\bR^3\backslash\{0,-n\}$ and the form of $\frac {\partial F } {\partial k_3} $ is simplified, we can  show the previous statement still holds.
 Therefore, 
\be\label{R:sum:k3}
\sum_{k_3 \in \bZ}\ind_V( k) \le 3+\int_{ \bR}\ind_V( k) \,\rd k_3 \,,\quad\;\forall\,k_H \in  \bR^2.
\ee 
Next, define
\[
S(k_3):=\big\{k_H\in \bR^2:\,(k_H,k_3)\in V
\big\}.
\]
Since any $k\in V$ satisfies $ |\ck_H| \leq |\cn|$, we sum both sides of \eqref{R:sum:k3} for all $k_H\in \bZ^2$ with $|\ck_H| \le |\cn|$ and apply Lemma \ref{lem:m:Jordan} to estimate $\sum_{k_H\in \bZ^2 ,\, |\ck_H| \le |\cn|}3$. This shows  
 \be\label{3Dto2D}
\begin{aligned}
\sum_{ k \in \bZt}\ind_V( k) \le  C_{n,\sL_{1,2}} +  \sum_{k_H\in \bZ^2, \, |\ck_H| \le | \cn|}&\int_{\bR}\ind_V( k) \,\rd k_3   \le C_{n,\sL_{1,2}}+ \int_{\bR} \# \big(\bZ^2\cap S(k_3)\big)  \,\rd k_3,\\
&\text{for constant } C_{n,\sL_{1,2}}\lesssim \sL_1\sL_2|\cn|^2+(\sL_1+\sL_2)|\cn|.
\end{aligned}
\ee
Here, we switched the sum and the integral because the sum only contains finitely many terms.

\subsection{Planar integer point counting with topological singularity}\label{s:planar}
We move to the counting problem $\# \big(\bZ^2\cap S(k_3)\big)$ from \eqref{3Dto2D}. 
By an elementary result due to Jarnik and Steinhaus \cite{JS:integer}, given a closed, rectifiable Jordan curve $\jc\subset\bR^2$ with the open set $S$ being its interior (c.f. Jordan curve theorem), we have  
\[
\big|\#\big(\bZ^2\cap(S\cup\jc)\big) - Area(S)\big| < Len(\jc).
\]
For application to our problem, we prove Lemma \ref{lem:m:Jordan} for {\it multiple disjoint} Jordan curves, which means the set of interest can be in either the interior or the exterior of a given curve. 

Next we focus on $\pa S(k_3)$. A topological subtlety is that  $\pa S(k_3)$ may contain tangentially contacting curves -- see Figure \ref{fig:top} for a 3D version of this. Another subtlety is that the {\it closed} set $S(k_3)$  may contain isolated curve or point which is not covered by Lemma \ref{lem:m:Jordan} according to the type of {\it open} sets it requires.  

\medskip

\begin{center}
\begin{figure}[h]
\hspace{-0.1\textwidth}\includegraphics[width=1.1\textwidth]{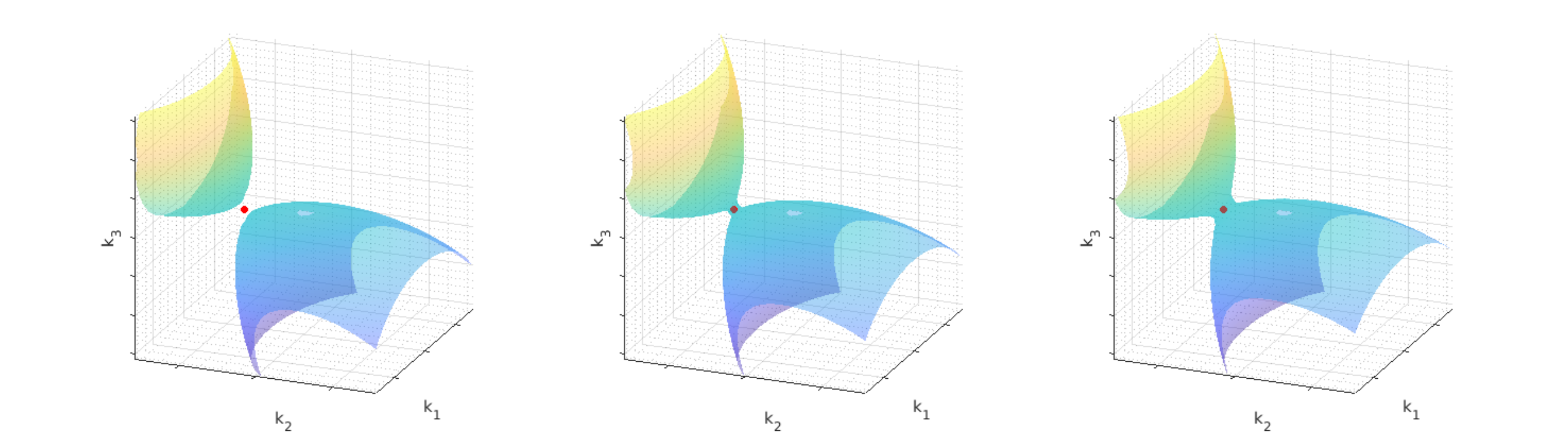}

\caption{{\small Three nearby level sets of $F_{n,+,+}(k)$ with $\sL_1=\sL_2=1$ at values $-\tfrac{28}{313}$, $-\tfrac{25}{313}$ and $-\tfrac{22}{313}$ for $n=(0,312,25)^T$ with the red dots all at $k=(0,-156,-\frac{25}2)^T$. It is a stationary point of $F_{n,+,+}(k)$ and the middle graph shows the level set at the corresponding critical value. There, the red dot is the only overlap of the two tangentially contacting surfaces, thus this level set fails to be manifold(s). The other two level sets are away from the critical values of $F_{n,+,+}$,  so they must be manifolds. In essence, perturbation of the singular case yields regular approximations.}}\label{fig:top}
\end{figure}
\end{center}
\vspace{-5mm}

\begin{lemma}\label{lem:2D:count}
Consider a fixed $n\in\bR^3\backslash\{0\}$,
  $k_3\in(-|\cn|,|\cn|)$ and bandwidth $\delta\ge 0$. Suppose
\be\label{ne0}
  k_3  (k_3+n_3)(2k_3+n_3)\ne0.
\ee
Then there exist constants $C,C'$    independent of $ n, k_3,\delta,\sL_1,\sL_2$ so that
\be\label{Z2:est}
\# \big(\bZ^2\cap S(k_3)\big) \le Area\big( S(k_3)\big) + C(\sL_1+\sL_2)\sqrt{ |\cn|^2-k_3^2 } +C'.
\ee 
where, for $\delta=0$, we have $Area(S(k_3))=0$. 
\end{lemma}

First, introduce  ``oblong'' cylindrical  coordinates  $\left(r_k,   \phi_k, k_3 \right)  \in [0, \infty) \times \bT_{2\pi } \times \bR $ for  $ k$ so that
\[
k_H = (\sL_1r_k\cos \phi_k, \sL_2r_k\sin \phi_k).
\]
Then $|\ck_H|=r_k$ and $|\ck|=\sqrt{r_k^2+k_3^2}$. 
Similar notations apply to wavevector $n$. 

At fixed $n$ and $k_3$, since assumption \eqref{ne0} ensures $F(k_H,k_3)$ is defined for all $k_H\in\bR^2$, we consider the cylindrical form of $F$ which is the mapping $G:[0, \infty) \times \bT_{2\pi }\mapsto\bR$ defined as
\begin{align}
G(r_k,\phi_k)&:=F(\sL_1r_k\cos \phi_k, \sL_2r_k\sin \phi_k, k_3)
\label{def:wt:F}=\sigma_1 \frac{n_3}{|\cn|} + \sigma_2 \frac{k_3}{\sqrt{r_k^2+k_3^2}} + \frac{m_3}{|\cm|},\\[2mm]
\label{kn:ip}\text{where \quad } |\cm|^2&= r_k^2 + 2  r_k r_n\cos (\phi_k - \phi_n) + r_n^2 +  m_3^2 \\
\nonumber\text{ and \quad }m_3&=-(k_3+n_3),
\end{align}
with \eqref{kn:ip} due to $\cm=-\ck-\cn$ and $\ck\cdot \cn=k_3n_3 + \ck_H\cdot\cn_H=k_3n_3 +r_k r_n\cos (\phi_k - \phi_n)$. 
 
  The azimuthal derivative will be of particular importance:  
  \be\label{Fphi}
  \dfrac {\partial G} {\partial \phi_k}  = \dfrac{m_3}{|\cm|^3}  r_n r_k \sin( \phi_k - \phi_n).
\ee
Also, by  the definition of $G$, we have 
\begin{subequations}\label{Fr}
\begin{align}
 &\frac{\pa}{\pa r_k}G(r_k,\phi_k)=0\\
 \label{Fr2}\implies\,&-\sigma_2\frac{k_3r_k}{(r_k^2+k_3^2)^{3/2}} =\frac{m_3( r_k + r_n\cos (\phi_k - \phi_n))}{\big( r_k^2 + 2  r_k r_n\cos (\phi_k - \phi_n) + r_n^2+m_3^2\big)^{3/2}}\\[1mm]
 \implies\,&r_k \text{ is a zero of a polynomial with leading term }
 (m_3^2-k_3^2)r_k^8\ne0\text{ under \eqref{ne0}.} 
 \end{align}
 \end{subequations} 
\begin{proof}[Proof of Lemma \ref{lem:2D:count}]

At fixed $n,k_3$,  
 \eqref{ne0} ensures $F(k_H,k_3)$ is defined for all  $k_H\in \bR^2$.  Let
\[
r_{in}:=\sqrt{\max\{0,  { \tfrac14|\cn|^2-k_3^2}\,\}}\quad\text{ and }\quad
r_{out} :=\sqrt{ |\cn|^2-k_3^2\, }\,,
\] 
that correspond to the inner edge (if $r_{in}>0$) and outer edge of the intersection of 3D annulus $\frac12|\cn|\le|\ck|\le|\cn|$ and the plane $k_3=constant$. If $r_{in}=0$ then the inner edge does not exist.

If $r_n=0$ then  $G(r_k,\phi_k)$ is independent of $\phi_k$ by \eqref{Fphi}. In view of \eqref{Fr}, the boundary $\pa S(k_3)$ consists of at most 11 concentric ellipses (on which either $|F|=\delta$ or $|\ck_H|=r_{in}$ or $r_{out}$), hence proving \eqref{Z2:est} by virtue of Lemma \ref{lem:m:Jordan} -- even if some of these ellipses are possibly not in the closure of the interior of $S(k_3)$. Note, for the case of $\delta=0$, the set $S(k_3)$ is a subset of these ellipses, and thus $Area(S(k_3))=0$.
Thus, for the rest of the proof, we assume 
\be\label{rn:0}
r_n>0 .
\ee
Combining this with \eqref{ne0}, \eqref{Fphi}, we find that $G$ at a fixed $r_k>0$ is strictly monotonic for $\phi_k\in[\phi_n-\pi,\phi_n]$ and $\phi_k\in[\phi_n,\phi_n+\pi]$, changing monotonicity at $\phi_k\equiv\phi_n\text{\;(mod\,}\pi)$. We will refer to this property as {\bf half-period monotonicity}. Immediately, for the case of $\delta=0$, the set $S(k_3)$ intersects with any ellipse of a fixed $r_k$ coordinate at no more than 2 points, which proves $Area(S(k_3))=0$, i.e. zero Lebesgue measure in $\bR^2$, as the last statement of the lemma. 

We recognise that $S(k_3)= S_+\cap S_-$ where
\[\begin{aligned}
S_+:=
\big\{&k_H \in\bR^2\, : \,|\ck_H|\in[r_{in}, r_{out} ] \text{\, and \,}F(k_H,k_3)\le \delta\,\big\},
\end{aligned}\]
and $S_-$ is defined similarly but with  $F$ replaced by $-F$. 

We now focus on the upper bound of $\# \big(\bZ^2\cap S_+\big)$. 
Since $\pa S_+$ is related to level set $F=\delta$,  we have some topological concerns as described before the lemma. To this end, define
\[
{\mathcal T}:=\bigcup_{\iota=0,1}\Big\{G(r_k,\phi_n+\iota\pi): r_k\text{ is a stationary point of }G(r_k,\phi_n+\iota \pi)\text{, or }r_k\in \{0,r_{in},r_{out}\} \Big\}.
\]

Next, we claim: there exist constants $C,C'$   independent of $ n, k_3,\delta,\sL_1,\sL_2$ so that
\be\label{claim:1side}
\text{if \, }\delta\notin{\mathcal T}\cup\big\{ \sigma_1\tfrac{n_3}{|\cn|},0\big\},
\;\;\text{ then  \, } \# \big(\bZ^2\cap S_+\big) \le Area\big( S_+\big) + C(\sL_1+\sL_2)r_{out} +C'.
\ee
The proof of this claim is as follows, where when we say \eqref{claim:1side} is proven ``by virtue of  Lemma \ref{lem:m:Jordan}'', it is understood that the role of $S$ is played by the interior of $S_+$ and a subtlety (related to what was pointed out before) is addressed here: we must show $S_+$ equals the closure of its interior, namely, every point of $S_+$ is a limit point of its interior.
In fact, recall $F(k)$ is continuous and a limit point of limit points is a limit point. Let $S'$ be the set of points in $S_+$ where $F<\delta$. Any point of $S'$ is either away from the edge(s) of the annulus, thus in the interior of $S_+$, or on an edge, thus a limit point of the previous type of points. Next, at every point of $S_+\backslash S'$ we must have  $F=G(r_k,\phi_k)=\delta$, so it is a limit point of $S'$ due to $\delta\notin\mathcal T$ if $\phi_k\equiv\phi_n\text{\;(mod\,}\pi)$, and due to the half-period monotonicity otherwise.

Now, define mapping $H:[0,\infty)\mapsto\bR$ as
\[
H(r_k)=\big(G(r_k,\phi_n) -\delta\big)\big(G(r_k,\phi_n+\pi)-\delta\big).
\]
If $H(r_k)>0$ for all $r_k\in[r_{in},r_{out}]$, then by continuity,  both $G(r_k,\phi_n)$ and $G(r_k,\phi_n+ \pi)$  stay on the same side of $\delta$. By the half-period monotonicity, we find that $S_+$ is either empty or the entire annulus $r_{in}\le r_k\le r_{out}$, which then proves \eqref{claim:1side} by virtue of  Lemma \ref{lem:m:Jordan}. 

Now we prove \eqref{claim:1side} assuming additionally  $H(\cdot)$ in $[r_{in},r_{out}]$ takes non-positive value(s). If there is a zero $r_k^*$ of $H$ 
 so that, e.g. $G(r_k^*,\phi_n) -\delta=0$, then  $G(r_k^*,\phi_n+\pi) -\delta\ne0$ by the half-period monotonicity. Also, assumption $\delta\notin{\mathcal T}$ implies ${\pa\over\pa r_k}G(r_k^*,\phi_n)\ne0$. 
  The argument so far shows $H$ is strictly monotonic at a neighbourhood of any of its zeros. 
 Since $H$ has at most 18 zeros due to \eqref{Fr} and  since $\delta\notin\mathcal T$ implies none of these zeros is in $\{0,r_{in}, r_{out}\}$ and  $H(0)>0$, there exist 
 \[ r_{in}\le\rkj{1}<\rkj{2}<\ldots<\rkj{2J}\le r_{out}\,,
 \qua{\, for some} J\in[1,10],\]
  that consist of all the zeros of $H$ located in $[r_{in},r_{out}]$ and if necessary, also include $r_{in}$ and/or $r_{out}$, so that, at any $r_k\in[r_{in},r_{out}]$,
 \[
H(r_k)\le 0\;\iff\;  r_k\in {\mathbf I}:=\bigcup_{j=1}^J\,\big[\rkj{2j-1},\rkj{2j}\big] . 
\]  
Then by half-period monotonicity, at given $r_k\in[r_{in},r_{out}]$,
\be\label{G:root}
\text{equation }G(r_k,\phi_k)=\delta\text{ has \,}\left\{\begin{aligned}
&\text{exactly one root }\phi_k\in[\phi_n,\phi_n+\pi]\,,&\text{if  }r_k\in {\mathbf I},\\ 
 & \text{no root $\phi_k\in\bT_{2\pi}$}\,,&\text{otherwise}.
\end{aligned}\right.
\ee
This allows us to define function $\Phi:{\mathbf I}\mapsto[\phi_n,\phi_n+\pi]$ to be that one root, namely
\be\label{G:d:I}
G(r_k,\Phi(r_k))=\delta,\qquad\forall\,r_k\in{\mathbf I}.
\ee
We will only be concerned with $r_k\in\mathbf I$ from now on. 
By definitions of $H, \mathbf I$ and $\Phi$, we have
\be\label{2:equiv}
\Phi(r_k^*)\equiv\phi_n\text{\;(mod\,}\pi)\;\iff\;H(r_k^*)=0\;\iff\;r_k^*\in\{\rkj{1},\rkj{2},\ldots,\rkj{2J}\}\backslash\{0,r_{in},r_{out}\},
\ee
where the exclusion of the last small set is due to $\delta\notin\mathcal T$.
  Then, in view of \eqref{Fphi}, \eqref{rn:0} and $m_3\ne0$ by assumption \eqref{ne0}, we apply the implicit function theorem to find
 that 
 \(
 \Phi(r_k)
 \) is differentiable provided that $r_k$ is {\it not} from any of the $r_*$ satisfying \eqref{2:equiv}. On the other hand, for any of the $r_*$ satisfying  \eqref{2:equiv}, 
 we restrict the level set $G=\delta$ in a small neighbourhood of such point $(r_k^*,\Phi(r_k^*))$,  and use assumption $\delta\notin{\mathcal T}$ and the implicit function theorem to find that the $r_k$ coordinate can be expressed  as a differentiable function of $\phi_k$.
In summary, the set
\be\label{def:G:d}
\Gamma_{\delta}:=\big\{(\sL_1r_k\cos\Phi(r_k),\sL_2r_k\sin\Phi(r_k))\in \bR^2:r_k\in {\mathbf I}\big\}
\ee
consists of  differentiable simple\footnote{\,A simple curve is a curve that does not cross itself.} curve(s) on which $F(\dcd,k_3)=\delta$.
 
 The curves of $\Gamma_\delta$ are rectifiable as long as the following improper integral is finite:
\be\label{Len:G}
\begin{aligned}
Len\big(\Gamma_{\delta}\big)&=\int_{{\mathbf I}}
\Big|\big({\sL_1(r_k\cos\Phi(r_k))',\,\sL_2(r_k\sin\Phi(r_k))'}\big)\Big| \,\rd r_k
 \lesssim (\sL_1+\sL_2)\Big(1+\int_{{\mathbf I}}|\Phi'|\,\rd r_k\Big) r_{out}.
\end{aligned}\ee
 Now  $\int_{\rkj{2j-1}}^{\rkj{2j}}|\Phi'|\,\rd r_k$ is the total variation of $\Phi(r_k)$ for $r_k\in(\rkj{2j-1},\rkj{2j})$ and is bounded by 1 plus the number of its extremum points, up to constant factor. We are allowed to assume $r_k\ne0$ as far as the integral \eqref{Len:G} is concerned. By \eqref{kn:ip} we have
\[
2r_n\cos\big(\Phi(r_k)-\phi_n\big)=-r_k+\frac{|\cm|^2-r_n^2-m_3^2}{r_k}\,.
\] Also, note $m_3\ne0$ by \eqref{ne0}. Since $F(\dcd,k_3)\big|_{\Gamma_\delta}=\delta$, we use \eqref{def:wt:F} and the above to find
 \[
2r_n\cos\big(\Phi(r_k)-\phi_n\big)=-r_k+\frac{\frac{|\ck|^2m_3^2}{(C_1|\ck|-k_3)^2}-r_n^2-m_3^2}{r_k}\quad\text{ \, on \, $\Gamma_\delta$ \, with \; } C_1:=\sigma_2\big(\delta-\sigma_1\tfrac{n_3}{|\cn|}\big).
\] 
Then, at an extremum point of $\Phi(r_k)$, the $r_k$ derivative of the above vanishes, i.e.
 \[
0=-1+\frac{1}{r_k}\,{\rd\over\rd r_k}\frac{|\ck|^2m_3^2}{(C_1|\ck|-k_3)^2}
 -\frac{\frac{|\ck|^2m_3^2}{(C_1|\ck|-k_3)^2}-r_n^2-m_3^2}{r_k^2}.
\]
By identity $\frac{1}{r_k}\,{\rd\over\rd r_k}=\frac{1}{|\ck|}\,{\rd\over\rd |\ck|}$ at fixed $k_3$, 
 we calculate the above derivative and substitute $r_k^2=|\ck|^2-k_3^2$ to find that  $|\ck|$ satisfies a polynomial equation with leading term $C_1^3|\ck|^5$ with $C_1\ne0$ under \eqref{claim:1side}. Thus,  we have established an upper bound on the number of extremum points of $\Phi(r_k)$. Since the number of intervals in ${\mathbf I}$ is $J\in[1,10]$, we continue from \eqref{Len:G} to find
\be\label{Phi:tv}
Len\big(\Gamma_{\delta}\big) \lesssim(\sL_1+\sL_2)r_{out}.
\ee

We move on to  $\pa S_+$. 
First, we consider the  half annulus excluding the edge(s), namely $(r,\phi_k)\in(r_{in},r_{out})\times[\phi_n,\phi_n+\pi]$. Any such point belongs to $\pa S_+$ if and only if it belongs to  level set $F=\delta$, which is due to half-period monotonicity and continuity. Then, by \eqref{G:root}, \eqref{def:G:d}, the closure of all such points equals $\Gamma_{\delta}$ which of course is still part of $\pa S_+$. 

Next, consider the half inner edge, namely $r_k=r_{in}$ and $\phi_k\in[\phi_n,\phi_n+\pi]$. Recall that if $r_{in}=0$ then inner edge does not exist. So, let $r_{in}>0$. If $\rkj{1}=r_{in}>0$, then by \eqref{2:equiv} we have $\Phi(r_{in})\in (\phi_n,\phi_n+\pi)$. By half-period monotonicity, we extend $\Gamma_\delta$ by part of the inner edge for either $\phi_k\in [\phi_n,\Phi(r_{in}))$ or $\phi_k\in(\Phi(r_{in}),\phi_n+\pi]$ so that $G(r_k,\phi_k)\le\delta$, and the result of this extension is still a simple curve. If $\rkj{1}>r_{in}>0$, then by \eqref{G:root}, the sign of $G-\delta\ne0$ is fixed on the half inner edge  which therefore is either in $\pa S_+$ or in $\bR^2\backslash\pa S_+$. In the former case, we also add it to $\Gamma_\delta$. Then, for such added curve (if any), \begin{itemize}
\item its length is controlled by the right-hand side of \eqref{Phi:tv};
\item it equals the intersection of $\pa S_+$ and the half inner edge;
\item its endpoints behave in exactly one of the two ways: either one endpoint satisifies $\phi_k\equiv\phi_n\text{\;(mod\,}\pi)$ and the other endpoint on the pre-extended $\Gamma_\delta$, or,  both endpoints satisfy $\phi_k\equiv\phi_n\text{\;(mod\,}\pi)$ and the pre-extended $\Gamma_\delta$ does not intersect the half inner edge.
\end{itemize} 

The same type of construction is used to extend $\Gamma_\delta$ onto the half outer edge.  We name the result of these extensions as $\Gamma_\delta^\sharp$. By construction, it is the intersection $\pa S_+$ with the  {\it closed} half annulus $[r_{in},r_{out}]\times[\phi_n,\phi_n+\pi]$. 
Further, every connected component of $\Gamma_\delta^\sharp$ is a simple curve with endpoints satisfying $\phi_k\equiv\phi_n\text{\;(mod\,}\pi)$, and conversely, by  \eqref{2:equiv} and the disjoint nature of components of ${\mathbf I}$, every point on $\Gamma_\delta^\sharp$ satisfying $\phi_k\equiv\phi_n\text{\;(mod\,}\pi)$ is an endpoint of the aforementioned connected component. 
Then, by  symmetry
\(
G(r_k,\phi_k)=G(r_k,2\phi_n-\phi_k)
\),
we find that performing such symmetry action on $\Gamma_\delta^\sharp$  yields the intersection $\pa S_+$ with the closed half annulus $[r_{in}, r_{out}]\times[\phi_n-\pi,\phi_n]$. The above information on the overlapping of $\Gamma_\delta^\sharp$ and $\phi_k\equiv\phi_n\text{\;(mod\,}\pi)$ ensures that the two sets are joined together to form disjoint Jordan curves. In view of \eqref{Phi:tv},  the proof of \eqref{claim:1side} is complete by virtue of Lemma \ref{lem:m:Jordan}. 

But what about the exclude $\delta$ values? 

 Since \eqref{Fr} ensures ${\mathcal T}$ is a finite set\footnote{\,In fact, it suffices to know that ${\mathcal T}$ has zero Lebesgue measure in $\bR$ by Sard's theorem.}, any non-negative $\delta\in {\mathcal T}\cup\big\{ \sigma_1\tfrac{n_3}{|\cn|},0\big\}$ can be approximated by $\delta_1>\delta$ from outside the previous set, so that $\delta_1$ satisfies \eqref{claim:1side}. In the upper bound in \eqref{claim:1side}, the excessive amount in the $Area(S_+)$ term due to such approximation equals
 \[
Area(S_{\delta,\delta_1})\qua{for}S_{\delta,\delta_1}:=\big\{(k_H\in \bR^2\,\big|\,\delta <f( k_H,k_3)\le \delta_1\text{ \, and \, }r_{in}\le|\ck_H|\le r_{out}\big\}.
 \]
 Since the integrable function $\ind_{S_{\delta,\delta_1}}(k_H)$ is monotonic in $\delta_1$ and
 \(
 \lim_{\delta_1\to\delta}\ind_{S_{\delta,\delta_1}}(k_H)=0
 \),
 we apply the monotone convergence theorem   to show that   $Area(S_{\delta,\delta_1})\to 0$ as $\delta_1\to\delta$. Noting the constants $C,C'$ in \eqref{claim:1side} are independent of $\delta$, we prove the conclusion of \eqref{claim:1side} for any $\delta\ge 0$.
 
 The counting of $\# \big(\bZ^2\cap S_-\big)$ follows the same steps except for $F,G$ replaced by $-F,-G$. Then, we apply inclusion-exclusion principle on $\big(\bZ^2\cap S(k_3)\big)= \big(\bZ^2\cap S_+\big)\cap \big(\bZ^2\cap S_-\big)$ to have
 \[
 \begin{aligned}
 \# \big(\bZ^2\cap S(k_3)\big) \le \,&  C(\sL_1+\sL_2)r_{out} +C'
 +Area\big( S_+\big) +Area\big( S_-\big)  
  -  \# \big(\bZ^2\cap \{r_{in}\le r_k\le r_{out}\}\big).
 \end{aligned}
 \]
 Apply inclusion-exclusion principle on  $S(k_3)=S_+\cap S_-$ to have
 \[
  Area\big( S_+\big) +Area\big( S_-\big)
  =   Area\big( S(k_3)\big)  + Area \big(  \{r_{in}\le r_k\le r_{out}\}\big).
  \]
  Therefore it suffices to prove an upper bound for
  \[
  Area \big(  \{r_{in}\le r_k\le r_{out}\}\big) -  \# \big(\bZ^2\cap \{r_{in}\le r_k\le r_{out}\}\big) \le C(\sL_1+\sL_2)r_{out}.
\]
This is achieved by applying Lemma \ref{lem:m:Jordan} on the complement set of annulus $ \{r_{in}\le r_k\le r_{out}\}$ relative to a rectangle containing that annulus. The proof of  \eqref{Z2:est} is complete.
  \end{proof}

To link back to 3D volume integral, we combine \eqref{N:le:V}, \eqref{3Dto2D} and Lemma \ref{lem:2D:count}, noting that the exclusion of finite number of $k_3$ values   does not affect the Lebesgue integral involved, to prove:

\begin{theorem}\label{thm:int:vol}For the near resonance set $\cN$ in \eqref{def:cN} subject to \eqref{de:max}, and for the set $\cNp\subset\cN$ from Lemma \ref{restricted:L}, consider the sets $V_\sigg$ given in \eqref{fpm:Car}, \eqref{sym2} with fixed wavevector $n\ne0$, thus fixed $\delta\ge 0$. Then there exists  constant $C$ independent of $n, \delta,\sigma_1,\sigma_2,\sL_1,\sL_2$ so that
\[
 \sum_{{ k \in \bZt }} \ind_{\cNp}(n,k,-n-k)\le C|\cn|^2(\sL_1\sL_2+\sL_1+\sL_2)+ \sum_{ (\sigma_1,\sigma_2)\in\{+,-\}^2}vol(V_\sigg),
\]
where, for $\delta=0$, it is understood that $vol( V_\sigg)=0$.
\end{theorem}


\sectionb{Volume Integral Estimates}\label{S:vol}
Continuing from Theorem \ref{thm:int:vol}, we focus our attention on upper bound of the volume $vol(V_\sigg)$, and it suffices to consider $\delta>0$.  The reader may refer to \S\ref{ss:outline} for a bigger picture of this multi-stage proof with technical highlights including the involvement of group theory.
\begin{theorem}\label{thm:V}
Consider $V_\sigg$ given in \eqref{fpm:Car}, \eqref{sym2} for any bandwidth $\delta\in(0,\frac12)$, wavevector $n\ne0$ and sign choice $(\sigma_1,\sigma_2)\in\{+,-\}^2$.  Then 
\be\label{vol:est}
\frac{vol(V_\sigg)}{\sL_1\sL_2|\cn|^3}\lesssim  \delta+\delta\log^+\frac{1}{\delta+\frac{2|n_3|}{|\cn|}}\,,
\ee
where the missing constant factor in the $\lesssim$ notation is independent of $n,\delta,\sigma_1,\sigma_2,\sL_1,\sL_2$.
\end{theorem}
The proof is given in \S\ref{s:proof:V}.

Although the logarithmic factor in \eqref{vol:est} is for volume estimate in  {\it continuum}, it is actually  optimal  even for {\it integer point} counting in the sense detailed in Appendix \S\ref{S:op}. 

To streamline the proof, we first do some housekeeping.
 
Similar to last section's strategy of excluding some rare $\delta$ values from the main argument, it suffices to prove
 \eqref{vol:est} for $n$ in a dense set of $\bR^3$. Assume it indeed holds. For   any $n\ne0$ not in that dense set and any sufficiently small $\ep>0$, there exists  $n^{app}$ in that dense set, so that 
 \[
 \left|\dfrac{n^{app}_{3}}{|\cn^{app}|}-\dfrac{n_3}{|\cn|}\right|<\ep,
 \] so the corresponding $F^{app}_\sigg(k)$ differs from $F_\sigg(k)$ by less than $\ep$. Then, by the (temporary) assumption that the $n^{app},\delta+\ep$ version of \eqref{vol:est} holds, we find that the $n,\delta$ version of LHS \eqref{vol:est} is controlled by the $n^{app},\delta+\ep$ version of RHS \eqref{vol:est}, hence proving the original form of \eqref{vol:est}. 
 
 Henceforth for the rest of \S \ref{S:vol},  we assume
 \be\label{ws1}
n \text{ \, with \, } n_1n_2n_3\ne0\qua{and}\delta\in\big(0,\tfrac12\big)\text{ \, are fixed}.
 \ee

We will use $\chibr{\cdot}$ that takes a Boolean expression as the argument and gives numerical value 1 if the expression is true, and 0 otherwise. The composition of $\chibr{\cdot}$ with its argument is then a mapping from the space where the variables of the Boolean expression reside to $\{0,1\}$. Then, when  $\chibr{\cdot}$ takes inequalities of simple enough nature, the corresponding mapping is Lebesgue measurable. For example, $\chibr{x\le y}$ for $(x, y)\in\bR^2$ is Lebesgue measurable in $\bR^2$.

Recall that the domain-adjusted wavevector $\ck$ is indicated with an accent $\check{\;}$. 
So we rewrite the volume element in \, $\int_{\bR^3}\ind_{V_\sigg}(k)\,\rd k$ \, as \, $\rd k=\sL_1\sL_2\,\rd \ck$, followed by substituting the definitions \eqref{fpm:Car}, \eqref{sym2}, and dropping all accents. This result is: 
\be\label{vol:V}
\frac{vol(V_\sigg)}{\sL_1\sL_2}=\int_{\bR^3\backslash\{0,-n\}}\chibr{\Big|\sigma_1\frac{n_3}{|n|}+\sigma_2\frac{k_3}{|k|}-\frac{n_3+k_3}{|n+k|}\Big|\le\delta}\,\chibr{\tfrac12|n|\le|k|\le|n|}\,\rd k.
\ee
Here the only abuse of notation is renaming  $\cn$ as $n$. The value of the integral does not depend on the choice of $\ck$ or $k$ notation. Henceforth, we discontinue the use of accent $\check{\;}$.

\subsection{Change of coordinates, twice}\label{S5:1}
For estimating the measure of a sublevel set of a function (here, it is the triplet value as function of $k$), we can swap the roles of dependent variable, i.e. the function itself, with one of the coordinates (here, it is the azimuthal angle of $k$). In what follows, the new coordinate is basically $\frac{m_3}{|m|}$ which has a simple affine relation to the triplet value if we fix every quantity except the azimuthal angle of $k$.

The first step is therefore to represent wavevector $ k$ in  spherical coordinates but with rescaled radius $\lambda_k:=\frac{|k|}{|n|}$ i.e.  $\left(\lambda_k,  \theta_k, \phi_k \right)  \in [0, \infty) \times [0, \pi] \times \bT_{2\pi }$ satisfying
\[\begin{aligned}
& (k_1,k_2,k_3)=|n|\lambda_k (\sin\theta_k \cos \phi_k, \sin \theta_k \sin \phi_k, c_k),\quad
\text{\, with}\quad c_k:=\cos\theta_k.
\end{aligned}\]
Similarly notations are used for $n,m$. Note $\lambda_n=1$. 

We impose the following working assumptions until they are addressed above \eqref{V:est1}:
\be\label{w:a:1}
\lambda_k\in\big[\tfrac12,1\big),\qquad\theta_k\in\big(0,\tfrac12\pi\big)\cup\big(\tfrac12\pi,\pi\big)\qqua{and}\lambda_k c_k + c_n \ne 0.
\ee
Naturally, the convolution condition $m=-k-n$ is always assumed.

By $k\cdot n=k_3n_3 + k_H\cdot n_H=|k|\cdot|n|\big(c_k c_n +\sin \theta_k \sin \theta_n\cos (\phi_k - \phi_n)\big)$, we expand $|m|^2=|k+n|^2 $, effectively expressing $\lambda_m$ in terms of $k, n$,
\be\label{lambda_m}
 \begin{aligned}
\lambda_m=\sqrt{ \lambda_k^2 + 1 + 2 \lambda_k  c_k c_n+ 2 \lambda_k  \sin \theta_k \sin  \theta_ n  \cos (\phi_k - \phi_n)  }\;\in\,(0,\lambda_k+1],
\end{aligned} 
\ee
where $\lambda_m>0$ is by the last part of \eqref{w:a:1}, i.e. $m_3\ne0$, and the upper bound due to $|m|\le|k|+|n|$. 

Next,
 define mapping $\scm: (0, \infty) \times [0, \pi] \times \bT_{2\pi }\mapsto \bR$  as
\be\label{c_m:fun}
  \scm(  \lambda_k,\theta_k,\phi_k)  :=  -\frac{ \lambda_k c_k + c_n }{\lambda_m}\;\;\;\left(=-\frac{n_3+k_3}{|n+k|}\,\right).
\ee
Define mapping $\mq:(0, \infty) \times [0, \pi] \times [-1,1]\mapsto \bR$ (that is ``quartic'' in either $\lambda_k, c_k$ or $c_m$) as:
\be\label{def:mq}
\mq(\lambda_k, \theta_k, c_m):=\left[2\lambda_k \sin \theta_k \sin\theta_n  \, c_m^2 \right]^2 - \left[   \left( \lambda_k^2 + 1 + 2 \lambda_k c_n c_k \right)c_m^2 - \left( \lambda_k c_k + c_n \right)^2\, \right]^2.
\ee
Then substitute \eqref{lambda_m} into \eqref{c_m:fun} and rearrange using $\sin^2 (\phi_k - \phi_n )+\cos^2 (\phi_k - \phi_n )=1$ to deduce:
\begin{equation}\label{variable2}
\left(\mq\big(\lambda_k,\theta_k, c_m \big)- \left[ 2\lambda_k \sin \theta_k \sin \theta_n  \sin (\phi_k - \phi_n )\, {c_m^2} \right]^2\,\right)\,\bigg|_{\text{\normalsize$c_m=\scm(  \lambda_k,\theta_k,\phi_k)$}}=0. 
\end{equation}

By ${\partial_{\phi_k}}  \scm(  \lambda_k,\theta_k,\phi_k)  = \lambda_k \sin \theta_k \sin \theta_n\sin( \phi_k - \phi_n)\lambda_m^{-3}(\lambda_k c_k + c_n)$ and under  \eqref{ws1}, \eqref{w:a:1} with fixed $\lambda_k,\theta_k$, we find the sign of $\scm\ne0$ is fixed, and the sign of ${\partial_{\phi_k}}\scm\ne0$ is fixed  for $\phi_k\in(\phi_n-\pi,\phi_n)$ and for $\phi_k\in(\phi_n-\pi,\phi_n)$, so $\scm$ is monotonic in $\phi_k$ in each interval, prompting the next change of coordinates. 
The Jacobian will be the reciprocal of $ {\partial_{\phi_k}}\scm $ as expressed above. We rewrite the Jacobian using \eqref{variable2} and $(\lambda_k c_k + c_n)=-\scm(  \lambda_k,\theta_k,\phi_k)\lambda_m$ due to \eqref{c_m:fun},
\[
\dfrac1{\big|\partial_{\phi_k}  \scm(  \lambda_k,\theta_k,\phi_k) \big|}=\frac{ 2 \lambda_m^{2} \, |c_m| }{ \sqrt{\mq(  \lambda_k,\theta_k, c_m)} }\,\Bigg|_{\text{\normalsize$c_m=\scm(  \lambda_k,\theta_k,\phi_k)$}}<\infty\,,\qua{ for}\phi_k\not\equiv\phi_n\,(\text{mod\,}\pi). 
\] 
By \eqref{lambda_m}, relax $\lambda_m^2$ to $(\lambda_k+1)^2$ in above so that the only $\phi_k$ dependence is via $c_m$, the upcoming new coordinate. 
Then, at fixed $\lambda_k,\theta_k$ coordinates subject to \eqref{w:a:1}, we change the $\phi_k$ coordinate to $c_m$ via $c_m=\scm(  \lambda_k,\theta_k,\phi_k)$ in the following Lebesgue  integral,
\begin{align*}
 & \qquad \int_{\phi_k\in \bT_{2\pi}\backslash\{\phi_n,\phi_n+\pi\}}\chibrB{\big|\sigma_1c_n+\sigma_2c_k+\scm(  \lambda_k,\theta_k,\phi_k)\big|\le\delta} \,\rd\phi_k\\[1mm]
\le \,&2\int_{\min\{\scm(\lambda_k, c_k,\phi_n),\,\scm(\lambda_k, c_k,\phi_n+\pi)\}}^{\max\{\scm(\lambda_k, c_k,\phi_n),\,\scm(\lambda_k, c_k,\phi_n+\pi)\}}\,\chibrB{|\sigma_1c_n+\sigma_2c_k+c_m|\le\delta}\,\frac{ 2 (\lambda_k+1)^2\, |c_m| }{ \sqrt{\mq(  \lambda_k,\theta_k, c_m)} }\,\rd c_m.\nonumber
\end{align*}
Now, the above information on the signs of ${\partial_{\phi_k}}\scm$ and $\scm$ ensures that any $c_m$ in the right-hand side integral domain {\it excluding endpoints} corresponds to a $\phi_k\in (\phi_n,\phi_n+\pi)$ so that $c_m=\scm(\lambda_k, c_k,\phi_k)\ne0$. This has two implications. First $|c_m|=|\scm(\lambda_k, c_k,\phi_k)|\le1$ by \eqref{c_m:fun} and simple geometry. Second,  $\mq(\lambda_k,\theta_k, c_m)>0$ by \eqref{variable2}, \eqref{ws1}, \eqref{w:a:1}. Therefore, we define
\be\label{def:mqs}
\mq^*(  \lambda_k,\theta_k, c_m)=\begin{cases}
\frac1{\sqrt{\mq(  \lambda_k,\theta_k, c_m)}},&\text{ if }\mq(  \lambda_k,\theta_k, c_m)>0,\\
0,&\text{ otherwise},
\end{cases}
\ee
and continue from the previous estimate, 
\begin{align*}
&\quad\int_{\phi_k\in \bT_{2\pi}\backslash\{\phi_n,\phi_n+\pi\}}\chibr{\big|\sigma_1c_n+\sigma_2c_k+\scm(  \lambda_k,\theta_k,\phi_k)\big|\le\delta} \,\rd\phi_k\\
&\le 4(\lambda_k+1)^2\int_{-1}^{1} {\chibr{|\sigma_1c_n+\sigma_2c_k+c_m|\le\delta}}\,\mq^*(  \lambda_k,\theta_k, c_m) \, |c_m| \,\rd c_m\,.
\end{align*}
Recall this estimate holds under assumption \eqref{w:a:1}, which allows us to  apply double integrals
\[
\int_{(0,\frac12\pi)\cup(\frac12\pi,\pi)}\int_{ [\frac12,1)\backslash\{-\frac{c_n}{c_k}\}}\,\sin\theta_k\,\rd\lambda_k \rd\theta_k.
\]
  Compare the LHS of the result (where $\scm$ satisfies \eqref{c_m:fun}) with the RHS of \eqref{vol:V} to find that both have the same integrand  and, when we switch back to un-scaled $|k|=|n|\lambda_k$, the two integral domains differ by a zero-measure set in $\bR^3$. Therefore  
\[
\frac{vol(V_\sigg)}{\sL_1\sL_2|n|^3}\le 16 \int_0^{\pi}\int_{\frac12}^{1}\int_{-1}^{1} {\chibr{|\sigma_1c_n+\sigma_2c_k+c_m|\le\delta}}\,
\mq^*(  \lambda_k,\theta_k, c_m)\,|c_m|\,\sin\theta_k\,\rd c_m\,\rd\lambda_k\, \rd\theta_k,
\]
where we also filled back a zero-measure set in the RHS integral domain. 
Then  by the virtue of Fubini-Tonelli theorem we have
\be\label{V:est1}
\frac{vol(V_\sigg)}{\sL_1\sL_2|n|^3}\le 16  \int_{-1}^{1} \int_0^\pi {\chibr{|\sigma_1c_n+\sigma_2c_k+c_m|\le\delta}} \, \mQ(\theta_k, c_m)\,\sin\theta_k\,\rd\theta_k\,\rd c_m
\ee
where
\be\label{def:mQ}
\mQ(\theta_k, c_m):=|c_m| \,\int_{\frac12}^{1}\mq^*(  \lambda_k,\theta_k, c_m)\,\rd\lambda_k.
\ee

\subsection{Elliptic integrals}\label{S:ei}

Recall \eqref{ws1} so that $c_n$ is fixed and
 $0<|c_n|<1$.
  Noting 
 the coordinates of \eqref{V:est1} are $\lambda_k, c_m,\theta_k$ (equivalently $c_k$),    we impose for the rest of \S \ref{S:ei} the following working assumptions (they are addressed below \eqref{Q:le:sharp}):
\be\label{no:S2}
c_n c_k c_m(1-c_n^2)(1-c_k^2)(1-c_m^2)(c_m^2- c_k^2)(c_k^2-c_n^2)(c_n^2-c_m^2)  \ne 0.
\ee  

To estimate  $\mQ$   with its integrand $\mq^*$ defined in \eqref{def:mqs} having singularities caused by the $\lambda_k$ roots of $\mq$ defined in \eqref{def:mq}, we appeal to the tool of elliptic integrals.  Start with   factorisation  
\be\label{q:cm2}
\begin{aligned}
{\mq}=-&\left[ ( c_m^2  - c_k^2)\lambda_k^2 + 2(    \cos (\theta_k + \theta_n)  c_m^2 - c_n c_k ) \lambda_k +  c_m^2  - c_n^2 \right]\\
 \times &\left[ (  c_m^2  - c_k^2  )\lambda_k^2 + 2(    \cos (\theta_k - \theta_n)  c_m^2 - c_n c_k  ) \lambda_k +  c_m^2  - c_n^2 \right],
\end{aligned}
\ee
in which the discriminants of the two quadratic-in-$\lambda_k$  factors  are, via trigonometry, 
 \[
  4   \sin^2 ( {\theta_k \pm \theta_n}  )  (1-c_m^2 )c_m^2 \ge  0.
 \]
Thus all possible $\lambda_k$ values that make $\mq(\lambda_k,\theta_k, c_m)=0$ are
\[
\dfrac{ c_n c_k -  \cos \left( {\theta_k +\wt\sigma \theta_n} \right) c_m^2+ \wt{\wt\sigma}  \sin \left( {\theta_k +\wt\sigma \theta_n}\right)\sin (\theta_m)\, c_m} {  c_m^2  - c_k^2} ,\quad\text{ for }\,\wt\sigma=\pm1,\; \wt{\wt\sigma}=\pm1.
\]
Since the numerator equals $c_n c_k -\cos ( \theta_k+\wt\sigma\theta_n+\wt{\wt\sigma}\theta_m)c_m$, we denote all $\lambda_k$-zeros of $\mq$ as
\begin{align}
\Lambda^{\sigma_k,\sigma_n} := \dfrac{ c_n c_k -  \cC^{\sigma_k,\sigma_n}c_m } {  c_m^2  - c_k^2} &\qquad\text{with \; } \cC^{\sigma_k,\sigma_n} := \cos \left( \theta_m + \sigma_k \theta_k +\sigma_n \theta_n \right) \label{L:C}\\
&\qquad \text{for \;}  \sigma_k=\pm1,\; \sigma_n=\pm1.\nonumber
\end{align}
Note $\wt\sigma, \wt{\wt\sigma},\sigma_k,\sigma_n$ are independent of $\sigma_1,\sigma_2$ in the NR conditions.

By elementary trigonometric identity
\be\label{e:id}
\cos^2\!\beta-\cos^2\gamma=-\sin(\beta+\gamma)\sin(\beta-\gamma)\,,
\ee
 we go through more elementary trigonometry\footnote{\,Rewrite the numerator of \eqref{L:C} as: $\frac12\cos\left( \sigma_k \theta_k +\sigma_n \theta_n \right) +\frac12\cos\left( \sigma_k \theta_k -\sigma_n \theta_n \right) -\frac12\cos\left( 2\theta_m + \sigma_k \theta_k +\sigma_n \theta_n \right) -\frac12\cos\left(  \sigma_k \theta_k +\sigma_n \theta_n \right) =\sin\left( \theta_m + \sigma_k \theta_k \right)\sin\left( \theta_m +\sigma_n \theta_n \right)  $, and the denominator  as $\sin(\theta_m+\sigma_k\theta_k)\sin(\theta_m-\sigma_k\theta_n)$.} to find
\be\label{L:s}
\Lambda^{\sigma_k,\sigma_n} =
-\frac{\sin(\theta_m+\sigma_n\theta_n)}{\sin(\theta_m-\sigma_k\theta_k)}.
\ee
Interested reader can explore the geometry of the above right-hand side with the help of \eqref{variable2} which shows that $\mq=0$ is linked to $\phi_k\equiv\phi_n\,(\text{mod\,}\pi)$.

In applying the tool of elliptic integral, we rely on the following notations.
For  $(a, b, c, d)\in \bR^4$, define products of interlaced (``IL''),   enclosed (``EC'')  and separated (``SP'') differences
\begin{align*}
\IL(a, b, c, d)&:=(a-c)(b-d),\\
 \EC(a, b, c, d)&:=(a-d)(b-c), \\
\SP(a, b, c, d)&:=(a-b)(c-d).
\end{align*}
They satisfy identity
\be\label{sum:id}
\IL=\EC+\SP.
\ee
Also, they are invariant under two-pair swapping permutations, known as ``double transpositions'', which consist  three members: interlaced ($a\leftrightarrow c$, $b\leftrightarrow d$),  enclosed ($a\leftrightarrow d$, $b\leftrightarrow c$), and separated ($a\leftrightarrow b$, $c\leftrightarrow d$). We call this Klein 4-group symmetry\footnote{\label{fn:4-group}\,Together with identity, they are the only self-inverse, parity-preserving members of the 4-permutation group, and form a  subgroup. Then, necessarily, they are commutative and the three non-identity elements are cyclic under the group operation.   This subgroup  is isomorphic to the Klein 4-group and to $\mathbb Z_2 \times \mathbb Z_2$.} or ``$K_4$-symmetry'' for short.

For {\it ordered} real numbers $a>b>c>d$,  introduce   parameters (e.g. \cite[p8, p102]{e:int})
\begin{align}
\label{def:sk}\sk^2(a, b, c, d) &:= \dfrac{ (a-b)(c-d)}{ ( a-c) (b-d )} =\dfrac{\SP(a, b, c, d)}{\IL(a, b, c, d)}<1,\\
\label{def:sg}\sg(a, b, c, d)&:= \dfrac 2 { \sqrt { (a-c) (b-d) } }=\dfrac 2 {\sqrt{\IL(a, b, c, d)}}, 
\end{align}
with $\sk^2<1$  due to the rearrangement inequality (or \eqref{sum:id}). 
The incomplete  elliptic integral of the first kind for any $0\le \sk<1$ is defined as
\[
  \int_0^ { \Psi} \dfrac 1 { \sqrt{ 1- \sk^2 \sin^2 x}} \, \rd x 
=
\int_0^{ \sin{ \Psi} } \dfrac 1 { \sqrt{ (1- t^2)(1-\sk^2 t^2)}} \,\rd t\,,\qquad 0\le\Psi\le\frac\pi2\,.
\]  
Its relation to $\mQ$ is given in Proposition \ref{prop:ell} with the quartet of singular points being  $\Lambda^{\pm,\pm}$ from \eqref{L:C} rearranged in {\it descending order}. 
 This would pose a lengthy task of tracking $4!=24$ cases, not to mention the inconvenient form of $\Lambda^{\pm,\pm}$. However, we can make great simplification thanks to  affine relation in \eqref{L:C}, and the  following two elegant correspondences.
  
\begin{lemma}[{\bf First Correspondence}]\label{lem:prod}
 For any 4-permutation $\bpi$, the following identity holds, 
 \[
 \Pi\big(\bpi(\cC^{++},\cC^{+-},\cC^{-+},\cC^{--})\big) = 4\,\Pi\big(\bpi(1,c_n^2, c_k^2, c_m^2) \big),\qua{\textnormal{ for}}\Pi\in\left\{\IL,\EC,\SP\right\}.
 \]
 Thus, under working assumptions \eqref{no:S2}, all  $\cC^{\pm,\pm}$ are distinct, and all $\Lambda^{\pm,\pm}$ are distinct.
\end{lemma}
\begin{proof}
 Let us start with the $\bpi=id$ version. Define mapping $\cD^{\,\sigma^*,\,\sigma^{**}}:{  [0,\pi]^3\mapsto[-1,1]}$ as
 \[
 \cD^{\sigma^*,\,\sigma^{**}}(\theta_1,\theta_2,\theta_3) := \cos \left( \theta_1 + \sigma^* \theta_2 +\sigma^{**} \theta_3 \right)\,
 \]
 so that 
 \[
 \cD^{\sigma_k,\sigma_n}(\theta_m,\theta_k,\theta_n)=\cC^{\sigma_k,\sigma_n}.
 \] Then, the case for $\Pi=\SP$ and $\bpi=id$ is due to the elementary identity
 \[
\SP\big((\cD^{++},\cD^{+-},\cD^{-+},\cD^{--})(\theta_1, \theta_2, \theta_3)\big)=4\,\SP\big(1,\cos^2\theta_3,\cos^2\theta_2,\cos^2\theta_1\big),
\]
for $(\theta_1,\theta_2,\theta_3)=(\theta_m,\theta_k,\theta_n)$.
With $(\theta_1,\theta_2,\theta_3)=(\theta_m,\theta_n,\theta_k)$, we find $(\cD^{++},\cD^{+-},\cD^{-+},\cD^{--})$ $(\theta_m, \theta_n, \theta_k)=(\cC^{++},\cC^{-+},\cC^{+-},\cC^{--})$, so the above identity also proves the case for $\Pi=\IL$ and $\bpi=id$. The case for $\Pi=\EC$ and $\bpi=id$ is proven similarly by combining the above identity and $(\cD^{++},\cD^{+-},\cD^{-+},\cD^{--})(\theta_n, \theta_k, \theta_m)=(\cC^{++},\cC^{--},\cC^{-+},\cC^{+-})$.

For cases with general permutations $\bpi$, we just combine the proven case for with $\bpi=id$ with the fact that the composition $\Pi\circ\bpi$ changes the type of $\Pi$ amongst $\IL,\EC,\SP$  and/or changes the sign in the same fashion for both sides of the claimed identity.
\end{proof}

Let $\vec{\cC}_s$ denote the vector resulting from sorting $(\cC^{++},\cC^{+-},\cC^{-+},\cC^{--})$ in descending order, and Let $\vec{\varsigma}_s$ denote the vector resulting from sorting $(1,c_n^2, c_k^2, c_m^2)$ in descending order.

\begin{lemma}[{\bf Second Correspondence}]\label{lem:sort}
For any $\Pi\in\left\{\IL,\EC,\SP\right\}$, we have
\[
 \Pi\big(\vec\cC_s\big) = 4\,\Pi\big(\vec\varsigma_s \big).
 \]
\end{lemma}
\begin{proof}

For  $\vec v=(a, b, c, d)$ with mutually distinct real numbers $a, b, c, d$, define $\bpi_{\vec v}$ to be the unique  permutation  that sorts $\vec v$ in descending order, namely 
\be\label{def:s:p}\vec v_{s}=\bpi_{\vec v}(\vec v).\ee  

Let $K_4$ denote the Klein 4-group and $S_4$ the 4-permutation group. When  $K_4$ is viewed as a subgroup of $S_4$, it consists of identity and the double transpositions given below \eqref{sum:id}. By group theory, $K_4$ is a normal subgroup of $S_4$, so the quotient group $S_4/K_4$ exists.

Now, take any 4-permutations $\bpi_1,\bpi_2$ so that $\bpi_2\in K_4\bpi_1 $  with coset $K_4\bpi_1 \in S_4/K_4$. By  the $K_4$-symmetry below \eqref{sum:id}, the three signs of $\EC(\cdot),\SP(\cdot),\IL(\cdot)$ acting on $\bpi_1\vec v$ and $\bpi_2 \vec v$ are identical. It turns out the reverse is also true: for a given $\vec v$, the three signs of $\EC(\vec v),\SP(\vec v),\IL(\vec v)$ suffice to uniquely determine the element of $S_4/K_4$ that contains   $\bpi_{\vec v}$.  In particular $\bpi_{\vec v}\in K_4\bpi^*$ with $\bpi^*$ being the ``special representative'' of the coset  that holds the first element fixed, known as a stabiliser. (In fact, choosing stabilisers as representatives for the cosets is  possible due to the isomorphism between $S_4/K_4$ and $S_3$.) See Table \ref{table:sort}. 
  
Let  $\vec\cC=(\cC^{++},\cC^{+-},\cC^{-+},\cC^{--})$  and   $\vec{\varsigma} = (1,c_n^2, c_k^2, c_m^2)$ so that  $\vec\cC_{s} = \bpi_{\vec\cC}\,(\vec\cC)$ and $\vec\varsigma_{s} = \bpi_{\vec\varsigma}\,(\vec\varsigma)$ by notation \eqref{def:s:p}. Using the one-to-one relation proven above and the $\bpi=id$ version of Lemma \ref{lem:prod},  we prove
  $\bpi_{\vec\cC}\in K_4\bpi_{\vec\varsigma}$. 
  Combine it  with the $K_4$-symmetry of the $\Pi$'s and  the $\bpi=\bpi_{\vec\varsigma}$ version of Lemma \ref{lem:prod} to complete the proof.
\end{proof}

{\small  \begin{table}[h]
\centering
\textrm{\textit{\begin{tabular}{c|c|c|c}
\hline
\rule{0pt}{5mm}
$\EC(\vec v)$&$\SP(\vec v)$&$\IL(\vec v)$& $\bpi_{\vec v}\in K_4\bpi^*$ where $\bpi^*=\ldots$\\[1mm]
\hline
\rule{0pt}{5mm}
$+$&$+$&$+$&$(a),b, c, d$\\[1mm]
$-$&$-$&$-$&$(a),d, c, b$\\[1mm]
$+$&$-$&$-$&$(a),d, b, c$\\[1mm]
$-$&$+$&$+$&$(a),c, b, d$\\[1mm]
$-$&$+$&$-$&$(a),c, d, b$\\[1mm]
$+$&$-$&$+$&$(a),b, d, c$\\[1mm]
$+$&$+$&$-$&impossible\\[1mm]
$-$&$-$&$+$&impossible\\[1mm]
\hline
\end{tabular}}}
\bigskip

\caption{\small All possible descending order permutations of $\vec v=(a, b, c, d)$ are divided into 6 cosets, distinguishable by the three signs $\vec v$ produces on the left columns. The ``special representative'' $\bpi^*$ of each coset   holds the first  element $(a)$ fixed. In other words, $a$ is always the largest entry in $\vec v$. Thus the signs of $\EC(\vec v),\SP(\vec v),\IL(\vec v)$ are simply computed as  the signs of $(b-c),(c-d),(b-d)$.}\label{table:sort}\vspace{-5mm}
\end{table}}

Let  $\vec{\Lambda}_s$ denote the vector resulting from sorting $(\Lambda^{++},\Lambda^{--},\Lambda^{+-},\Lambda^{-+})$ in descending order. 
Since    $\Lambda^{\pm,\pm}$   and $\cC^{\pm,\pm}$ are linked by  an affine transform as in \eqref{L:C}, we  recall definitions \eqref{def:sk}, \eqref{def:sg} 
and apply Lemma \ref{lem:sort} to obtain
\be\label{sg:so}
\sg(\vec{\Lambda}_s) =\dfrac{ \left| c^2_m - c^2_k \right|}{ |c_m| } \, \frac1{ \sqrt{ \IL(\vec{\varsigma}_{s})  }}
\ee
and  (note \eqref{sum:id})
\be\label{sk:so}
 1-\sk^2(\vec{\Lambda}_s) =\frac { \EC(\vec{\varsigma_s})} {   \IL(\vec{\varsigma_s}) }.
\ee
Then, noting the first entry of $\vec\varsigma_s$ is always 1,
we have reduced the bookkeeping task of $4!=24$ cases to $3!=6$ cases that account the ordering of $c_n^2,c_k^2,c_m^2$. In fact, we show   in the argument below \eqref{Q:le:sharp} that there are  only 2 essentially distinct scenarios. But first, let us estimate $\mQ$. 

\begin{lemma}
\label{lem:Q:est}
Consider $\mQ(\theta_k,c_m)$ from \eqref{def:mQ}. Under working assumptions \eqref{no:S2}, we have
\begin{align}
\label{cI:lesssim} \mQ & \lesssim \frac{1+\log\sqrt{ { \IL(\vec{\varsigma_s})}/ {   \EC(\vec{\varsigma_s}) }}}{\sqrt{\IL(\vec\varsigma_{s})}},
\end{align}
and a (conditional) sharper estimate:
\be\label{cI:sharp}
|c_n|<\min\{|c_k|,|c_m|\}\quad\implies\quad\mQ \lesssim \frac{1 }{\sqrt{\IL(\vec\varsigma_{s})}}.
\ee
\end{lemma}
\begin{proof}Denote elements of the descending-ordered vectors 
\[
\vec \varsigma_{s} = (1,\varsigma_{2}, \varsigma_{3}, \varsigma_{4})\qua{\, and \,}\vec{\Lambda}_s=\left(\Lambda_{1},\Lambda_{2},\Lambda_{3},\Lambda_{4}\right).
\]  

Combine \eqref{def:mQ}, \eqref{def:mqs} and \eqref{def:mq}, noting the coefficient $-(c_m^2-c_k^2)^2$ of $\lambda_k^4$ therein, to show
\be\label{mQ:L}
\mQ(\theta_k, c_m)=\frac{|c_m|}{|c_m^2-c_k^2|}\int_{\left((\Lambda_4,\Lambda_3)\cup (\Lambda_2,\Lambda_1)\right) \atop \bigcap\left[\frac12,1\right]}\frac{\rd\lambda_k}{\sqrt{-(\lambda_k-\Lambda_1)(\lambda_k-\Lambda_2)(\lambda_k-\Lambda_3)(\lambda_k-\Lambda_4)}}\,.
\ee

(i) For proving the general bound \eqref{cI:lesssim}, we simply drop the $\left[\frac12,1\right]$ restriction of the above integral and combine elementary estimate \eqref{int:yz:est} with identities \eqref{sg:so}, \eqref{sk:so}.

(ii) For proving the sharper bound \eqref{cI:sharp}, by the geneneral bound  \eqref{cI:lesssim} we just proved, it suffices to assume $ { \IL(\vec{\varsigma_s})}/ {   \EC(\vec{\varsigma_s}) } > 9$ for the rest of the proof. Then 
\be
\label{k:est2} {\varsigma_{2} - \varsigma_{4}}>9( {\varsigma_{2} - \varsigma_{3}}).
\ee
Next, by \eqref{e:id}, \eqref{L:s},  we have
\begin{align}
\label{inv:pm:b}\Lambda^{+,\sigma_n}\Lambda^{-,\sigma_n} \,\big(c_m^2-c_k^2\big)&<0  ,\\
\nonumber\Lambda^{\sigma_k,+}\Lambda^{\sigma_k,-} \,\big( c_m^2-c_n^2\big)&<0,
\end{align}
the latter of which together with the ordering in \eqref{cI:sharp} implies 
\be\label{L:0}
\Lambda_1>\Lambda_2>0>\Lambda_3>\Lambda_4.
\ee
 Then in view of the restriction $\lambda_k\in[\frac12,1]$ in the integral \eqref{mQ:L}, it suffices to consider 
\be\label{L:2:1}
0<\Lambda_{2} < 1,
\ee 
and we can relax the integral domain of \eqref{mQ:L} to $[\Lambda_2,1]$. Then, by \eqref{sg:so} and  Proposition \ref{prop:ell}, the sharper bound  \eqref{cI:sharp} would hold if \eqref{int:yz:est:sharp:c} holds for $(a,b,c)=(\Lambda_1,\Lambda_2,\Lambda_3)$ and $y=1$. Since $1-\Lambda_2<1$ by \eqref{L:2:1}, it suffices to bound $\Lambda_{2}-\Lambda_{3}$ and $\Lambda_{1}-\Lambda_{2}$ from below by positive constants with the latter bound strictly greater than 1. The rest of the proof is for this purpose.
\begin{itemize}[leftmargin=*]
\item For bounding $\Lambda_{2}-\Lambda_{3}$ from below, we use affine relation \eqref{L:C} and Lemma  \ref{lem:prod}
 together with $\EC(1,c_n^2, c_k^2, c_m^2)<0$ and $\IL(1,c_n^2, c_k^2, c_m^2)<0$ thanks to the  ordering in \eqref{cI:sharp},  to have 
\[
\EC\left(\Lambda^{++},\Lambda^{+-},\Lambda^{-+},\Lambda^{--}\right)<0
\quad\text{  
and 
 }\quad
\IL\left(\Lambda^{++},\Lambda^{+-},\Lambda^{-+},\Lambda^{--}\right)<0.
\] 
Then, using simple logic\footnote{\,if the largest and smallest elements of $x, y, z, w$ form set $\{x,w\}$ or  set $\{y,z\}$, then $(x-y)(z-w)>0$.}, we find that $\{\Lambda_1,\Lambda_4\}$ can {\it not} equal to any of these sets: $\left\{\Lambda^{++},\Lambda^{-+}\right\}$, $\left\{\Lambda^{+-},\Lambda^{--}\right\}$, $\left\{\Lambda^{++},\Lambda^{--}\right\}$, $\left\{\Lambda^{+-},\Lambda^{-+}\right\}$. Therefore there are only ${4\choose2}-4=2$ choices for the set $\left\{\Lambda_{1},\Lambda_{4}\right\}$ and equivalently only 2 choices for $\left\{\Lambda_{2},\Lambda_{3}\right\}$. Then by \eqref{L:s}, we find
\[
\begin{aligned}
\Lambda_{2}-\Lambda_{3}&=\min\Big\{\big|\Lambda^{++}-\Lambda^{+-}\big|\,,\,\big|\Lambda^{--}-\Lambda^{-+}\big|\Big\}
=\frac {2|c_m|\sin\theta_n} {|c_m|\sin\theta_k + |c_k|\sin\theta_m}>\frac {2|c_m|} {|c_m| + |c_k|},
\end{aligned}
\]
where we used the ordering in \eqref{cI:sharp} and the positivity of all $\sin\theta$'s. Now, relax $\varsigma_4<0$ in \eqref{k:est2} and use it together with the ordering in  \eqref{cI:sharp} to find either $c_k^2<c_m^2$ or $c_k^2 < {\frac98}c_m^2$. Then continue from the above to prove a positive lower bound for 
$\Lambda_{2}-\Lambda_{3}$.
\item For bounding $\Lambda_{1}-\Lambda_{2}$ from below by a constant strictly greater than 1, there are two cases.

$\circ$ If $1 > |c_m|>|c_k|>|c_n|$, we observe that each pair of roots in either of the two quadratic-in-$\lambda_k$ factors of \eqref{q:cm2} multiply to  $\frac{c_m^2-c_n^2}{c_m^2-c_k^2}$ which is greater than $9$ by  \eqref{k:est2}.  Then, by the sign information of \eqref{L:0}, we have
  \(
  \Lambda_1\Lambda_2 >9
  \). This together with \eqref{L:2:1} shows 
$ \Lambda_1-\Lambda_2>8$.

$\circ$  If $1 > |c_k|>|c_m|>|c_n|$, then 
 \eqref{inv:pm:b} and \eqref{L:0} together imply $\left\{\Lambda_{1},\Lambda_{2}\right\}=\left\{\Lambda^{+,\sigma_n}, \Lambda^{-,\sigma_n}\right\}$ for $\sigma_n=1$ or $-1$, so we use \eqref{e:id}, \eqref{L:s} to obtain
\be\label{L12:est}
\begin{aligned}
\Lambda_{1}+\Lambda_{2}& = -\sin(\theta_m+\sigma_n\theta_n)\cdot\frac{\sin(\theta_m+\theta_k) + \sin(\theta_m-\theta_k)}{c_k^2-c_m^2} \\
&=\frac{c_m^2-c_n^2}{\sin(\theta_m-\sigma_n\theta_n)}\cdot \frac{2\,c_k\sin\theta_m}{c_k^2-c_m^2}>\frac{ 2\sin\theta_m}{\sin \theta_n + \sin \theta_m}\cdot\frac{c_m^2-c_n^2}{c_k^2-c_m^2},
\end{aligned}
\ee
where the last step was to due to $\Lambda_{1}+\Lambda_{2}>0$ and $|\!\sin(\theta_m-\sigma_n\theta_n)|<(\sin \theta_n + \sin \theta_m)\,|c_k|$ thanks to the ordering assumption. Then in view of \eqref{k:est2}, we  find
\[
 \Lambda_{1}+\Lambda_{2}>  \frac{ 16\sin\theta_m}{\sin \theta_n + \sin \theta_m}.
\]
By relaxing $c_k^2<1$ and rewriting $c_m^2-c_n^2=\sin^2\theta_n-\sin^2\theta_m$ in \eqref{L12:est}, we also  have
\[
 \Lambda_{1}+\Lambda_{2}>\frac{2(\sin\theta_n-\sin\theta_m)}{\sin\theta_m}.
\]
Since $\sin\theta_n>\sin\theta_m>0$ and the above two lower bounds have obvious monotonicity in $\sin\theta_n$, we can show $\Lambda_{1}+\Lambda_{2}>4$ which, with \eqref{L:2:1}, proves $\Lambda_{1}-\Lambda_2>2$.
\end{itemize}
\end{proof}

\subsection{Proof of Theorem \ref{thm:V}}\label{s:proof:V} 
We claim every NR condition with $\delta \in(0,\frac12)$ implies
\be\label{s4}
\sqrt{\varsigma_{4}}  \le \tfrac34,
\ee
regardless of the signs $\sigma_1,\sigma_2$.
Suppose not, so $1\ge\sqrt{\varsigma_2}\ge\sqrt{\varsigma_3}\ge \sqrt{\varsigma_4}>\frac34$. Note that, up to a harmless sign change, a triplet value is either $\sqrt{\varsigma_2}+\sqrt{\varsigma_3}\pm \sqrt{\varsigma_4}$ or $\sqrt{\varsigma_2}-\sqrt{\varsigma_3}\pm \sqrt{\varsigma_4}$. In the former case, the $\sqrt{\varsigma_2}>\frac34$ term is too large to make the entire expression fall within $[-\delta,\delta]$. In the latter case, the first two terms combine to a value in  $(0,\frac14)$, but $\sqrt{\varsigma_4}>\frac34$ so $\sqrt{\varsigma_2}-\sqrt{\varsigma_3}\pm \sqrt{\varsigma_2}\in [-\delta,\delta]$ is not possible. Contradiction has been reached. 

Applying $\log x \le x-1$ to \eqref{cI:lesssim} shows $\mQ  \lesssim \frac{1}{\sqrt{\EC(\vec\varsigma_{s})}}$. Combining this with \eqref{s4} to find
\be\label{Q:le:sharp}
 \mQ \lesssim  \frac1{\sqrt{\varsigma_{2}-\varsigma_{3}}} . 
\ee

Now, inspired by the ordering concern in Lemma \ref{lem:Q:est}, introduce shorthand notations
\[\begin{aligned}
\chi_0&:=\chibr{c_k c_m(1-c_k^2)(1-c_m^2)(c_m^2- c_k^2)(c_k^2-c_n^2)(c_n^2-c_m^2)=0},\\
\chi_1&:=(1-\chi_0)\cdot\chibr{|c_k|<\min\{|c_n|,|c_m|\}},\\
 \chi_2&:=(1-\chi_0)\cdot\chibr{|c_m|<\min\{|c_n|,|c_k|\}}, \\
  \chi_3&:=(1-\chi_0)\cdot\chibr{|c_n|<\min\{|c_k|,|c_m|\}}.
\end{aligned}\]
Since their sum is at least 1,
we perform simple change $c_k=\cos\theta_k$ in \eqref{V:est1} and have
\be\label{est:cmck}
\frac{vol(V_\sigg)}{\sL_1\sL_2|n|^3}\lesssim \sum_{j=0}^3\int_{(-1,1)^2}{\chibr{|\sigma_1c_n+\sigma_2c_k+c_m|\le\delta}} \cdot\chi_j \cdot  \mQ(\arccos c_k, c_m) \,\rd c_m\,\rd c_k.
\ee
 Note that we can view $\varsigma_2,\varsigma_3,\varsigma_4$ and $\chi_0,\ldots,\chi_3$ as measurable functions\footnote{\,e.g. $\min\{a, b\}=\frac12(a+b)-\frac12|a-b|$ and $\varsigma_3=c_n^2+c_k^2+c_m^2-\min\{c_n^2,c_k^2,c_m^2\}+\min\{-c_n^2,-c_k^2,-c_m^2\}$.} of $c_n, c_k, c_m$.

 The $j=0$ summand of \eqref{est:cmck} vanishes due to the zero measure of the support of $\chi_0$.

 For the $j=1$ summand of \eqref{est:cmck}, we use \eqref{Q:le:sharp} and change of coordinates from $(c_m, c_k)$ to $(c_m,\delta')$ with $\delta'=\sigma_1c_n+\sigma_2c_k+c_m$ to bound it as
 \[
\lesssim \int_{-1}^1\int_{-3}^{3} \chibr{|\delta'|\le\delta}\cdot\frac{\chi_1}{\sqrt{\varsigma_{2}-\varsigma_{3}}}\,\rd\delta'\,\rd c_m
= \int_{-\delta}^{\delta}\bigg(\int_{-1}^{1}\frac{\chi_1}{\sqrt{ |c_m^2-c_n^2|}} \,\rd c_m\bigg)\,\rd\delta',
 \]
where $ \varsigma_{2}-\varsigma_{3} = |c_m^2-c_n^2|$ was due to $\chi_1=1$.  Also,
\[
\chi_1=1\;\implies\;(-|c_n|,|c_n|)\ni -\sigma_2c_k=c_m-\delta'+\sigma_1c_n.
\]
Then,  the $\chi_1$ factor in the above right-hand side integrand  allows us to restrict the $c_m$-integral to an interval of length $2|c_n|$.   Therefore, by elementary Calculus,  the above inner integral is uniformly bounded. Thus the double integral is $\lesssim\delta$, which is consistent with Theorem \ref{thm:V}.

The $j=2$ summand of \eqref{est:cmck} is  treated similarly to $j=1$ since  \eqref{Q:le:sharp} depends on ordering, not labels of the $c$'s, and since $c_k,c_m$ play  symmetric roles in $\chi_1,\chi_2$. Also $c_k, c_m$ are symmetric in the NR condition in the sense of $ |\sigma_1c_n+\sigma_2c_k+c_m| = |\sigma_1\sigma_2 c_n+\sigma_2c_m+c_k| $.

For the $j=3$ summand of \eqref{est:cmck}, estimate \eqref{cI:lesssim} or \eqref{Q:le:sharp} is not sharp enough, especially since we will argue optimality in \S \ref{S:op}. Instead \eqref{cI:sharp} is used for this case. Then, by different but simple considerations for the cases of $\varsigma_2<(\frac78)^2$ and $\varsigma_2\ge(\frac78)^2$ (also note \eqref{s4}), we have
  \be\label{j3} \begin{aligned}
& \text{ the $j=3$ summand of \eqref{est:cmck}}\\
\lesssim & \int_{(-1,1)^2}{\chibr{|\sigma_1c_n+\sigma_2c_k+c_m|\le\delta}} \cdot \max\Big\{\frac{\chi_3}{\sqrt{\varsigma_2 - \varsigma_4}},\,\frac{\chi_3}{\sqrt{1- \varsigma_3}}\Big\}\,\rd c_m\,\rd c_k.
\end{aligned} \ee
 We prove the following estimates on the contributions of the two arguments of $\max$.\vspace{1mm} 
 \begin{itemize}[leftmargin=*]
 \item By definition of $\chi_3$, we have $\frac{\chi_3}{\sqrt{1-\varsigma_3}}\le \frac1{\sqrt{1-c_k^2}}$, so the contribution due to $\frac{\chi_3}{\sqrt{1-\varsigma_3}}$  is
 \[
 \lesssim \int_{-1}^1\int_{-1}^1{\chibr{|\sigma_1c_n+\sigma_2c_k+c_m|\le\delta}} \cdot  \frac{1}{\sqrt{1- c_k^2}}\,\rd c_m\,\rd c_k \le \int_{-1}^12\delta\cdot\frac{1}{\sqrt{1- c_k^2}}\,\rd c_k\lesssim\delta,
 \] 
  which is consistent with Theorem \ref{thm:V}. \vspace{1mm}
 \item
 For the contribution to  \eqref{j3} due to $\frac{\chi_3}{\sqrt{\varsigma_2 - \varsigma_4}} $, we divide it into two sub-cases using
 \[
 \chi_3=\chi_3\cdot\chibr{|c_k|< |c_m|}+\chi_3\cdot\chibr{|c_k|>|c_m|}\stackrel{def}=\chi_{3,1}+\chi_{3,2}.
 \]
 Thanks again to the $c_k, c_m$ symmetry, it suffices to treat the case  $\chi_{3,1}=1$.
 Then by change of coordinates from $(c_m, c_k)$ to $(c_m,\delta')$ with $\delta'=\sigma_1c_n+\sigma_2c_k+c_m$, we bound the contribution to  \eqref{j3} due to the double integral of $\chibr{\ldots}\cdot\frac{\chi_{3,1}}{\sqrt{\varsigma_2 - \varsigma_4}}\,\rd c_m \,\rd c_k$  as
  \be\label{chi3less}
\lesssim \int_{-1}^1\int_{-3}^{3} \chibr{|\delta'|\le\delta}\cdot\frac{\chi_{3,1}}{\sqrt{\varsigma_{2}-\varsigma_{4}}}\,\rd\delta'\,\rd c_m
= \int_{-\delta}^{\delta}\bigg(\int_{-1}^{1}\frac{\chi_{3,1}}{\sqrt{ c_m^2-c_n^2}} \,\rd c_m\bigg)\,\rd\delta',
 \ee
 with  $\varsigma_{2}-\varsigma_{4}=c_m^2-c_n^2$ by $\chi_{3,1}=1$.
 Also
  $\chi_{3,1}=1$  implies
   \[
(-|c_m|,|c_m|)\bigcap\Big((-1,-|c_n|)\cup(|c_n|,1)\Big)\ni    -\sigma_2c_k=c_m-\delta^*,\qua{ for} \delta^*:=\delta'-\sigma_1c_n.
   \]
   Now, for brevity, it suffices to integrate over $c_m\in(0,1)$ as the case of  $c_m\in(-1,0)$ is treated similarly.
   Thus, we have (with the convention that $(b, a)=\emptyset$ if $b\ge a$)
   \[\begin{aligned}
   \chi_{3,1}=1\text{ \, and \, }c_m\in(0,1)\;&\implies\;\delta^*-c_m<c_m<\delta^*+c_m
   \text{\, and \,}(0,\delta^*-|c_n|)\cup(\delta^*+|c_n|,1)\ni c_m\\
   &\implies\;\delta^*>0   \text{\, and \,} (\tfrac12\delta^*,\delta^*-|c_n|)\cup(\delta^*+|c_n|,1)\ni c_m.
   \end{aligned}\] 
    Combining them with  the simple estimates
\[
\int_b^a  \frac{\rd x}{\sqrt{x^2-c_n^2}}=\log\big(x+\sqrt{x^2-c_n^2}\big)\bigg|_b^a<\log\frac{2a}b,\qquad\text{for \; }a>b>|c_n|,
 \]
 and
 \[
\int_b^a  \log\frac1x\,\rd x= (x-x\log x)\Big|_b^a<(a-b)\Big(1+\log\frac1a\Big),\qquad\text{for \; }a>b\ge 0,
 \] 
 we find
 \[\begin{aligned}
 \eqref{chi3less}
 & \lesssim \delta +\max_{\sigma_1,\sigma_2}\int_{-\delta}^{\delta}\chibr{\delta'-\sigma_1c_n>0}\cdot\log^+\frac1{\delta'-\sigma_1c_n+|c_n|}\,\rd\delta'\\[1mm]
&\lesssim\begin{cases}
 \delta+\delta\log^+\frac{1}{\delta+2|c_n|},&\text{ if \, }\delta\le|c_n|,\\[1mm]
 \delta+\max\left\{(\delta+|c_n|)\log^+\frac{1}{\delta+2|c_n|},\,(\delta-|c_n|)\log\frac{1}{\delta}\right\},&\text{ if \, }\delta>|c_n|.
 \end{cases}
 \end{aligned}\]
 The latter bound can be relaxed using $\delta+|c_n|<2\delta$ and $\log\frac1\delta
 <\log 3+\log\frac1{\delta+2|c_n|}$,  hence merged into the former estimate. 
 \end{itemize}

We have completed the proof of  Theorem \ref{thm:V} (also see below \eqref{vol:V} regarding the notational change between $\cn$ and $n$).



\section{\bf Proofs of the Main Theorems}\label{s:mainproofs}

We are ready to prove the 2D-like estimates for bilinear form $\wt B$ resulting from our proposed NR approximation under suitable conditions on the bandwidth $\delta$.
 \begin{proof}[Proof of Theorem \ref{thm:2Dlike}]
 
 By \eqref{tri:e}, we find
 \[
 \ip{\pD^s \wt B( \vu,\vv),\pD^s\vw}=|\bTt|\sumcnkm |\cn|^{s}\, (u_k\cdot\cm\,\ri)\, (v_m\cdot w_n)\,|\cn|^{s}\,\ind_{\cN}({n,k,m})\,.
 \]
 Relax one $|\cn|^s$ factor using \eqref{cn:2s}.
 Note the symmetry of $\ind_\cN(\cdot,\cdot,\cdot)$ with respect to argument permutations is obvious by definition. Then we prove \eqref{tri:est} using Lemma \ref{restricted:L} 
 where counting condition \eqref{ci}  with $\beta=1$ is satisfied due to Theorems \ref{thm:int:vol}, \ref{thm:V} and the choice of $\delta$ in \eqref{de:up}.
 
Next, expand the left-hand side of \eqref{comm:est} using \eqref{tri:e},  then add it with the result of switching $n,m$, use incompressibility $u_k\cdot\cm=-u_k\cdot\cn$ and use the symmetry of $\ind_\cN(n,k,m)$ and the symmetry of convolution sum from Remark \ref{re:sum} to show 
\begin{align}
\label{tri:nm}
&\;2\Big|\ip{\pD^s\wt B(\vu,\vw),\pD^s\vw}\Big|\le |\bTt|\sumcnkm \big|u_k\cdot\cm\big|\, \big|w_m\cdot w_n\big| \,\big||\cn|^{2s}-|\cm|^{2s}\big|\,\ind_{\cN}({n,k,m}). 
\end{align}
Now, by the mean value theorem and $\big||\cn|-|\cm|\big|\le |\cn+\cm|$, we have
\be\label{cn:cm:s}
\big||\cn|^{s}-|\cm|^{s}\big|\le s\max\big\{|\cn|^{s-1},|\cm|^{s-1}\big\}\,|\ck|\,,\qquad\forall\,\text{real \,}s\ge0.
\ee 
We also have $|\cn|^{s}+|\cm|^{s}\le 2\max\big\{|\cn|^{s},|\cm|^{s}\big\}$. Then the first case of \eqref{comm:est} for $s\in(0,1]$ is due to the estimates so far, the simple fact below \eqref{comm} and Lemma \ref{restricted:L} with $\beta=1$. The second case for $s>1$ is done similarly with the additional estimate (by \eqref{cn:2s}):
\[
\max\big\{|\cn|^{s-1},|\cm|^{s-1}\big\}\le 2^{s-1}\min\big\{|\cn|^{s-1},|\cm|^{s-1}\big\}+2^{s-1}|\ck|^{s-1}.
\]
 \end{proof}
Global well-posedness result Theorem \ref{thm:NR}   shows the propagation of $H^s(\bTt)$ regularity of the initial data to all $t\ge 0$. When $s=0$, it coincides with the notion of weak solutions {\it a la} Leray (\cite{Leray}). Since various versions of weak solutions can be found in literature, let us  formalise it precisely for our purpose. First, we make sense of $\wt B$ when its arguments are only in  $L^2(\bTt)$.
  \begin{definition}\label{def:L2wk}
 Given $\bC^3$-valued $\vu,\vv\in L^2(\bTt)$ with $\vu$ div-free, we define $\wt B(\vu,\vv)$ as a functional on $\bC^3$-valued
smooth functions defined on $\bTt$, namely, for any such function $\vh(x)$, define
\[
\ipwk{\wt B(\vu,\vv),\vh}:=-\sum_{n\in\bTt}\left(\int_{\bTt}\sum_{k\in\bTt}e^{-\ri\cn\cdot x}(u_k\otimes v_{-n-k}):\nabla\leray\overline\vh(x)\,\rd x\right)\!\ind_\cN(n,k,-n-k),
\]
where $\otimes$ denotes tensor product and $:$ denotes the dot product between tensors.
\end{definition}
The above sum converges absolutely, because $\vh(x)$ is smooth (making $|h_n|\lesssim |\cn|^{-10}\pDn{20}\vh$),  and  by Cauchy-Schwartz inequality, we have 
\be\label{conv:L2L2}
{|\bTt|}\sum_{k\in\bTt}|u_k|\, |v_{-n-k}|\le \pDn0{\vu}\pDn0{\vv}\,,\quad\text{\, at a fixed $n$}.
\ee 
By a similar argument, the usual $L^2$ pairing can be recovered as (for real $s\ge 0$),
\be\label{ipwk:ip}
\ipwk{\wt B(\vu,\vv),\pD^{2s}\vh}=\big\langle{\pD^s\wt B(\vu,\vv),\pD^s\vh}\big\rangle,\, \quad\forall\,\vu,\vv\in H^{s+1}(\bTt)\text{ and smooth }\vh(x).
\ee

\begin{definition}\label{def:wk} We call $\wt\vU\in \scc^0([0,\infty); L^2(\bTt))\cap L^2([0,\infty); H^1(\bTt))$ a global weak solution of the NR approximation \eqref{NR} if $\wt\vU(t,\cdot)$ is div-free for all $t\ge 0$ and if, for any $\bC^3$-valued, smooth function $\vpsi(t,x)$, the following weak formulation holds  for any $T\ge 0$,
\[\begin{aligned} 
 &\ip{\wt\vU(T,\cdot),\vpsi(t,\cdot)}-\ip{\wt\vU(0,\cdot),\vpsi(0,\cdot)}+\int_0^{T}\ipwk{\wt B(\wt\vU(t,\cdot),\wt\vU(t,\cdot)),\vpsi(t,\cdot)}\,\rd t \\
 =\,&\int_0^{T}\int_\bTt\left(\wt\vU\cdot\partial_t\overline\vpsi+\Omega\cL\wt\vU\cdot\overline\vpsi-\mu\nabla\wt\vU:\nabla\overline\vpsi\right)(t,x)\,\rd x\,\rd t .
\end{aligned}
\]
 \end{definition}
Note the continuity in time for $\wt\vU(t,\cdot)$ with values in $L^2(\bTt)$ is a stronger requirement than the commonly seen  $L^\infty$-in-$t$ requirement. Also, as Theorem \ref{thm:NR} also covers more regular spaces, we note that for  weak solutions in $\scc^0([0,\infty); H^{s}(\bTt))\cap L^2([0,\infty); H^{s+1}(\bTt))$,  they coincides with  other notions of solutions as $s$ increases. For example, with $s\ge1$, both   $\Om\cL\tvU$ and $\mu \Delta \tvU$ belong to $L^2([0,\infty); H^{s-1}(\bTt))$.  One can show $\wt B(\tvU,\tvU)$ also belongs to this space using 
 \be\label{cn:2s}
 |\cn|^s\le (|\cm|+|\ck|)^s\le 2^{(s-1)^+}\big(|\cm|^s+|\ck|^s\big),\quad\text{\, for real \,}s\ge 0,
 \ee
(by minimising $1/(\alpha^s+(1-\alpha)^s)$ for $\alpha\in[0,1]$) and \eqref{tri:inf:fi}(ii) with $h$ being a test function. Then, the  result of setting $\vpsi=\vpsi(x)$ in  Definition \ref{def:wk}
 \[
 \tvU(T,\cdot)-\tvU(0,\cdot)=\int_0^{T}-\wt B(\tvU , \tvU) + \Om\cL\tvU + \mu \Delta \tvU\,\rd t
 \]  is a Bochner integral with integrand in $L^2([0,\infty); H^{s-1}(\bTt))$. Therefore, $\partial_t\tvU$ is also in this space and importantly, the two sides of PDE \eqref{NR} are regarded as an identical element in this space. In short, under the $s\ge1$ regularity, the weak formulation and such identity in the PDE form are equivalent, so one can manipulate the latter without involving the former. Similar argument shows, with $s\ge2$, the PDE is satisfied in the
$\scc^0([0,\infty); H^{s-2}(\bTt))$ space which,  if $s>{7\over2}$, is  embedded in $\scc^0([0,\infty)\times \bTt)$, the last case called classical solutions.

Next, let $\cPl$ denote projection  $\cPl\vu=\sum_{n\in\bZ^3,|\cn|<R}e^{\ri\cn\cdot x}u_n$, known as  low-pass filter.  
 
 \begin{remark}\label{re:cPl}$\cPl$ preserves the realness of its argument, is self-adjoint with respect to the $\ip{,}$ inner product, and commutes with  $\leray,\cL,\pD$, differentiation and integration.
\end{remark}

 \begin{remark}\label{re:Ut}  For a weak solution from Definition \ref{def:wk}, using $\vpsi=e^{\ri\cn\cdot x}$ and \eqref{conv:L2L2}, we find Fourier coefficients of $\tvU(t,\cdot)$ are smooth functions of $t$, hence $\cPl\tvU\in \scc^\infty([0,T]\times \bTt)$. 
 \end{remark} 
 \begin{lemma}[{\bf Stability of NR approximate dynamics}]\label{lem:stable}
 For the NR approximation \eqref{NR} under the same assumptions on $\delta, \cN$ as in Theorem \ref{thm:NR}, suppose it admits two weak solutions with  regularity $\wt\vU,\wt\vU'\in \scc^0([0,T]; H^s(\bTt))\cap L^2([0,T]; H^{s+1}(\bTt))$ for some real  $s\ge0$ and $T>0$. Then, for $E=\int_0^T\pDn{s+1}{(\tvU,\tvU')}^2\,\rd t$ and  constant $C=C(s,|\bTt|)$, we have,
 \[
 \pDn{s}{\tvU-\tvU'}^2(t)\le e^{C\mu^{-1}E}\pDn{s}{\tvU-\tvU'}^2(0),
 \]
 for any $t\in[0,T]$, and
 \[
 \int_0^T\pDn{s+1}{\tvU-\tvU'}^2(t)\,\rd t\le \mu^{-1}\big(1+C\mu^{-1}Ee^{C\mu^{-1}E}\,\big)\pDn{s}{\tvU-\tvU'}^2(0),
 \]
 \end{lemma}
 \begin{proof}  
 Let $\vw=\tvU-\tvU'$ and $\vw_R=\cPl\vw$.  
By taking the difference of the weak formulation in Definition \ref{def:wk}  for $\tvU$ and $\tvU'$, setting $\vpsi=\pD^{2s}\vw_R$ (smooth by Remark \ref{re:Ut}) and noting the term involving $\cL$ vanishes, we have, for any $ t_1\in[0,T]$,
 \begin{align}
\label{e:m} \pDn{s}{\vw_R}^2(t_1)-&\pDn{s}{\vw_R}^2(0)+I_R=\int_{0}^{t_1}\frac12{\rd\over\rd t}\pDn{s}{\vw_R}^2-\mu\pDn{s}{\nabla\vw_R}^2\;\rd t,\\
\nonumber \text{ with \; }I_R=&\int_{0}^{t_1}\ip{\wt B(\tvU,\vw_R)+\wt B(\tvU,\vw-\vw_R)+\wt B(\vw,\tvU'),\pD^{2s}\vw_R}_{wk}\,\rd t\,.
 \end{align}
Since $\tvU(t,\cdot),\tvU'(t,\cdot)\in H^{s+1}(\bTt)$ for almost every $t\in[0,T]$, and since $I_R$ is an integral in $t$, we can transform its integrand using \eqref{ipwk:ip}. Then, by Theorem \ref{thm:2Dlike}, we find
\[\begin{aligned}
|I_R|\lesssim_{|\bTt|,s}\!\int_{0}^{t_1}\pDn{s+1}\tvU\,\pDn{s}{\vw_R}\,\pDn{s+1}{\vw_R}+
&\big(\pDn s\tvU\,\pDn{s+1}{\vw-\vw_{R}}+\pDn{s}{\vw}\,\pDn{s+1}{\tvU'}\big)\pDn{s+1}{\vw_R}+\\
&\big(\pDn {s+1}\tvU\,\pDn{s+1}{\vw-\vw_{R}}+\pDn{s+1}{\vw}\,\pDn{s+1}{\tvU'}\big)\pDn{s}{\vw_R}\,\rd t.
\end{aligned}\]
Then,  take the limit of \eqref{e:m} as $R\to\infty$. The assumed regularity on $\tvU,\tvU'$ ensures  Lebesgue dominated convergence theorem can be applied to the above right-hand side and also ensures $\pDn{s+1}{\vw-\vw_{R}}(t)\to 0$ as $R\to\infty$ for almost every $t\in[0,T]$. So
\[\begin{aligned}
\pDn{s}{\vw}^2(t_1)-\pDn{s}{\vw}^2(0)+2\int_{0}^{t_1}\mu\pDn{s}{\nabla\vw}^2\,\rd t
\lesssim& \int_{0}^{t_1}\pDn{s+1}{(\tvU,\tvU')}\,\pDn{s}{\vw}\,\pDn{s+1}{\vw}\,\rd t.
\end{aligned}\]
Using the assumed regularity and the dominated convergence argument again, we can replace $(\tvU,\tvU')$ by $\cPl(\tvU,\tvU')$ up to arbitrary error. Then, further applying Young's inequality shows
\[
\pDn{s}{\vw}^2(t_1)-\pDn{s}{\vw}^2(0)+\int_{0}^{t_1}\mu\pDn{s}{\nabla\vw}^2\,\rd t
\lesssim \ep_R+ \int_{0}^{t_1}\mu^{-1}\pDn{s+1}{\cPl(\tvU,\tvU')}^2(t)\,\pDn{s}{\vw}^2(t)\,\rd t,\;\;\forall t_1\in[0,T],
\]
with $\ep_R\to 0$ as $R\to\infty$. 
Dropping the left-hand side integral gives a Gr\"onwall's inequality for $\pDn{s}{\vw}^2(t)$ with $\pDn{s+1}{\cPl(\tvU,\tvU')}^2(t)$ being smooth due to Remark \ref{re:Ut}. This then proves the first part of the conclusion. Substituting it into the above proves the second part.
 \end{proof}

\subsection{Proof of the global well-posedness result Theorem \ref{thm:NR}}\label{ss:mainproof}
The stability statement of the theorem is self-explained. Together with estimate \eqref{global:Hs}, it implies uniqueness.
 Let us prove global existence and \eqref{NR:en:L2}, \eqref{global:Hs}. 
   
  First,  
  fix large $\rho>0$ (to appear in \eqref{vVR}). For any $R>\rho$, define bilinear form 
\[
\wt B^R(\vU,\vV):=\cPl\wt B(\cPl\,\vU,\cPl\vV).
\] 
Its eigen-basis expansion is similar to \eqref{wt:B:full2} but with  an additional $\ind_{\{\max\{|\cn|,|\ck|,|\cm|\}<R\}}(n,k,m)$ factor in the summand. 
Consider the following equation
 \be\label{vVR}
 \partial_t \vVR   + \wt B^R(\vVR , \vVR) = \Om\cL\vVR + \mu \Delta \vVR,
\quad \text{\, with \, }\vVR(0,\cdot)=\cP_{\!<\rho} \tvU_0.
 \ee
  Define also $E^{\vV}_{j0}:=\pDn{j}{\vVR(0,\cdot)}^2$ for any $j\ge 0$. Apparently $E^{\vV}_{j0}$ is finite and independent of $R$.

The collection of Fourier modes $\tVR_n(t)$ with $|\cn|<R$ satisfy a self-contained, finite dimensional ODE system. (A side note: nonlinearity $\wt B^R$ can cause $\tVR_n(t)\ne0$ for modes $\rho\le |\cn|<R$ at $t>0$, even for initial data as in \eqref{vVR}.) We refer to it as the ``$R$-low-pass ODE'' and prove it is globally solvable as follows. Apparently $\vVR(t,\cdot)$ is div-free. By Remark \ref{re:cPl} and Proposition \ref{prop:Bt}, we find $\ipb{\wt B^R(\vVR , \vVR),\cPl\vVR}=0$ and also $\ipb{\cPl(\Om\cL\vVR+\mu\Delta\vVR),\cPl\vVR}\le0$. Therefore, applying $L^2$ energy method on \eqref{vVR} shows,  if the  ODE is solvable for $t\in[0,T]$, then
\be\label{l2:ineq}
\sum_{n\in\bTt,|\cn|<R}|\tVR_n|^2(t) \le \sum_{n\in\bTt,|\cn|<R}|\tVR_n|^2(0)= |\bTt|^{-1}E_{00}^\vV\,,\qquad\forall \;t\in[0,T].
\ee
Now, the $R$-low-pass ODE is autonomous. Its right-hand side  function is Lipschitz continuous in variables $\{\tVR_n\}_{|\cn|<R}$  with respect to the max norm, and the Lipschitz constant is controlled by $\max\big\{|\tVR_n|\big\}_{|\cn|<R}$. 
Then, apply Picard iteration on the $R$-low-pass ODE in an interval $t\in[0,t_1]$ for $t_1$ small enough so that the result of each iteration has its max norm bounded by $2\sqrt{|\bTt|^{-1}E_{00}^\vV}$\,. Note that this bound holds for the base step of the iteration because we know \eqref{l2:ineq} holds at $t=0$. The aforementioned Lipschitz continuity implies that we can decrease $t_1$ (if need be) to make the iteration mapping a contraction with respect to the max norm, so by the Banach fixed-point theorem,  the solution exists  for $t\in[0,t_1]$. We can choose $t_1$ to depend only on $R$ and $|\bTt|^{-1}E_{00}^\vV$, the proof of which is elementary, thus omitted. The solvability proven so far implies \eqref{l2:ineq} holds at $t=t_1$. 
 Then, by the same Picard iteration and fixed-point argument, we prove the $R$-low-pass ODE is solvable for $t\in[t_1,2t_1]$, and for $t\in[2t_1,3t_1]$, and so on.\vspace{1mm}
 
On the other hand, the choice of initial date in \eqref{vVR} and the definition of $\wt B^R$ with $R>\rho$ implies $\tVR_n(t)\equiv 0$ if $|\cn|\ge R$. 
Then, the information on $\tVR_n(t)$  we have shown so far together with \eqref{vVR} implies $\vVR(t,x)$ is smooth in $[0,\infty)\times\bTt$. So we take the inner product $\pD^{j}\eqref{vVR}$ with $\pD^{j}\vVR$, knowing Calculus rules for smooth functions apply, to obtain
  \[
\frac12\frac{\rd}{\rd t}\pDn {j}\vVR^2+\mu\pDn {j+1}\vVR^2=-\big\langle\,\pD^j\wt B^R(\vVR,\vVR),\pD^j\vVR\,\big\rangle, \qquad \forall\,\text{real \,}j\ge 0.
  \]
  The proof of 2D-like estimate \eqref{comm:est} can be repeated verbatim for $\wt B^R$ except that we apply the $n,m$ symmetry of $\ind_\cN(n,k,m)\ind_{\{\max\{|\cn|,|\ck|,|\cm|\}<R\}}(n,k,m)$ to obtain a version of \eqref{tri:nm} where $\wt B^R$ replaces $\wt B$  and there is an additional harmless $\ind_{\{\max\{|\cn|,|\ck|,|\cm|\}<R\}}(n,k,m)$ factor in the upper bound. Therefore, for any real $j\ge0$, we have
\be\label{G:ineq}\begin{aligned}
&\quad\frac12\frac{\rd}{\rd t}\pDn{j}\vVR^2+\mu\pDn {j+1}\vVR^2
\le \sqrt{2C_{j}} \, \pDn1\vVR \pDn {j}\vVR\, \pDn {j+1}\vVR \\&\le \frac{C_{j}}{\mu}\pDn1\vVR^2\,\pDn {j}\vVR ^2 +\frac\mu2\pDn {j+1}\vVR ^2\,,\qquad\qquad\text{\, 
for \, }0\le C_{j}\lesssim 4^{j}j^2 \,|\bTt|^{-1}.
\end{aligned}\ee 
Since setting $j=0$ gives zero value after the first $\le$, we obtain the analogue of equality \eqref{NR:en:L2},
\be\label{NR:en:L2:V}
\|\vVR(T,\cdot)\|_{L^2}^2+2\mu\int_0^{T}\|\nabla\vVR\|_{L^2}^2\,\rd t =  \|\vVR(0,\cdot)\|_{L^2}^2\,, \qquad\forall \,T\ge 0.
\ee
Solving Gr\"onwall's inequality \eqref{G:ineq} for general $j$ gives
\[
\pDn {j}{\vVR(T,\cdot)} ^2\le \exp\left(\frac{2C_{j}}\mu\int_0^{T}\pDn1{\vVR(t,\cdot)}^2\,\rd t\right)E_{j0}^\vV\,,\qquad\forall\,T\ge 0\text{ and }j\ge 0.
\]
Then, integrate \eqref{G:ineq} in time and relax every $\pDn1\vVR^2$  to the above right-hand side, followed by applying $\int_0^\infty\pDn1{\vVR(t,\cdot)}^2\,\rd t\le\frac1{2\mu}E_{00}^\vV$ (by \eqref{NR:en:L2:V}), to obtain an upper bound on $\int_0^\infty\pDn {j+1}{\vVR}^2\,\rd t$. In summary, we prove \eqref{global:Hs} with every $\tvU$ replaced by $\vVR$, $s$ replaced by $j$,   $E_{s0}$ replaced by $E_{j0}^\vV$ and $E_{00}$ replaced by $E_{00}^\vV$.  As any order of time derivatives of $\vVR$ can be estimated using \eqref{vVR}, \eqref{cn:2s}, \eqref{tri:inf:fi}(ii), it is easy to show that all $\{\vVR\}_{R>\rho}$ have their $H^{j}([0,\infty)\times\bTt)$ bounds finite and independent of $R$. Obviously $j$ can be arbitrarily large. Therefore by compact embedding, for any positive $j,T$, there exist $\vV_{\!\!j,T}\in H^{j}([0,T]\times\bTt)$ and a sequence of $R$ values (indices omitted for brevity) tending to $\infty$ so that
\[
\lim_{R\to\infty}\big\|\vVR-\vV_{\!\!j,T}\big\|_{H^{j}([0,T]\times\bTt)}=0.
\]
Using the standard technique of repeatedly extracting subsequences of $R$ values (again, indices omitted for brevity), we show there exists a smooth function $\vV(t,x)$ so that
\[
\lim_{R\to\infty}\big\|\vVR-\vV\big\|_{H^{j}([0,T]\times\bTt)}=0,
\]
for any positive $j,T$. 
By Sobolev embedding, this implies convergence in $\scc^{j-3}([0,T]\times\bTt)$. Then for very positive $j$, every linear term of \eqref{vVR} converges to its $\vV$ analogue in $H^{j}([0,T]\times\bTt)$ and the initial condition part also converges in $H^j(\bTt)$. For the bilinear term of \eqref{vVR}, split the following difference into three parts,
\[\begin{aligned}
&\wt B^R(\vVR , \vVR)-\wt B(\vV , \vV)= \big(\wt B^R(\vVR , \vVR)-\wt B^R(\vV , \vV)\big)+\\
&\big(\wt B^R(\vV , \vV)-\wt B^R(\cP_{\!\!<{R / 2}}\vV , \cP_{\!\!<{R / 2}}\vV)\big)+
\big(\wt B(\cP_{\!\!<{R / 2}}\vV , \cP_{\!<{R / 2}}\vV)-\wt B(\vV , \vV)\big),
\end{aligned}\]
where the two terms involving $\cP_{\!\!<{R / 2}}$ cancel due to the definition of $\wt B^R$. 
Each difference is concerned with the same type of bilinear form and can be estimated similarly. For example, by Definition \ref{def:L2wk}, \eqref{conv:L2L2} and choosing any smooth $\vh(x)$ (making $|h_n|\lesssim |\cn|^{-10}\pDn{20}\vh$), we have
\be\label{ipwk:stable}
 \ipwk{\wt B(\vu,\vu) - \wt B(\vu',\vu'),\vh}\lesssim \big(\pDn0{\vu}+\pDn0{\vu'}\big)\pDn0{\vu-\vu'}\pDn{20}\vh.
\ee
At any fixed $t\in[0,\infty)$, set  $\vh=\pD^{2j}\big(\wt B(\cP_{\!\!<{R / 2}}\vV , \cP_{\!\!<{R / 2}}\vV)-\wt B(\vV , \vV)\big)$, $\vu=\cP_{\!\!<{R / 2}}\vV$, $\vu'=\vV$, and use the smoothness of $\vV$ as proven above, noting \eqref{ipwk:ip}, to find
$\big\|\wt B(\cP_{\!\!<{R / 2}}\vV , \cP_{\!\!<{R / 2}}\vV)-\wt B(\vV , \vV)\big\|_j$ to be essentially controlled by $\pDn0{\cP_{\!\!<{R / 2}}\vV-\vV}$\,, hence tending to 0 as $R\to\infty$. 

 In short, we have just proven that every term in \eqref{vVR} converges strongly to its $\vV$ analogue in $\cC^0([0,T];H^j(\bTt))$ as $R\to\infty$ for any $j$, and in particular, the strong limit of $\wt B^R(\vVR , \vVR)$ is $\wt B(\vV , \vV)$ with the bilinear form also converging.  To finish the proof, we rename $\vU^{(\rho)}:=\vV$ to highlight its dependence on $\rho$. Then
\be\label{vVrho}
 \partial_t \vU^{(\rho)}   + \wt B(\vU^{(\rho)} , \vU^{(\rho)}) = \Om\cL\vU^{(\rho)} + \mu \Delta \vU^{(\rho)}\,,
\quad \text{\, with \, }\vU^{(\rho)}(0,x)=\cP_{\!<\rho}\tvU_0\,.
 \ee
 Recall we have also proven \eqref{NR:en:L2}, \eqref{global:Hs} with $\tvU$ is replaced by $\vV^{R}$, and by taking the limit, we obtain their $\vU^{(\rho)}$ analogue.
Then, by the stability Lemma \ref{lem:stable}, we have $\{\vU^{(\rho)}\}_{\rho\in\bZ^+}$ form a Cauchy sequence in $\scc^0([0,\infty); H^s(\bTt))$ and $\{\nabla\vU^{(\rho)}\}_{\rho\in\bZ^+}$ form a Cauchy sequence in $L^2([0,\infty); H^{s}(\bTt))$. Both are Banach spaces, so each sequence converge strongly in the respective space to, say, $\tvU$ and ${\bs W}$. Therefore, for any $3\times3$ tensor-valued smooth function $\vh(t,x)$, we have
 \[
 \int_0^T\int_\bTt \tvU\cdot(\nabla\cdot\vh)-{\bs W}:\vh\,\rd x\,\rd t=
 \lim_{\rho\to\infty} \int_0^T\int_\bTt \vU^{(\rho)}\cdot(\nabla\cdot\vh)-\nabla\vU^{(\rho)}:\vh\,\rd x\,\rd t=0.
 \]
 By choosing the smooth function $\vh=\cP_{<R}(\nabla\tvU-{\bs W})$ for any $R>0$, we show $\nabla\tvU={\bs W}$. 
 
The strong convergence of $\vU^{(\rho)}\to\tvU$  in the above sense implies every term in the weak formulation ({\it a la} Definition \ref{def:wk}) of \eqref{vVrho} converges to the desired limit. Note: convergence of the term with $\ipwk{,}$ is due to \eqref{ipwk:stable}.
  The global existence part of the Theorem has been proven. Also,  \eqref{NR:en:L2} and \eqref{global:Hs} follow from the strong convergence of $\vU^{(\rho)}\to\tvU$ together with argument just below \eqref{vVrho} and the fact that 
 $\pDn0{\vU^{(\rho)}(0,x)}^2=\pDn0{\cP_{\!<\rho}\tvU_0}^2\to E_{00}$ as $\rho\to\infty$.
 
The proof of  Theorem \ref{thm:NR} is complete.

\subsection{Proof of the error estimate Theorem \ref{thm:local:e}}\label{ss:errorproof}

First, we show nonlinear estimates in a generic setting. For functions $f,g,h$ in $\bTt$, define
\[
f_{abs}:=\sum_n |f_n| \, e^{\ri\cn\cdot x},\text{ \; and similarly define }g_{abs},h_{abs}.
\]
Since ${g_{abs} h_{abs}}=\sum_n\Big(\sum_k |g_{k}|\,|h_{n-k}|\Big)e^{\ri\cn\cdot x}$ which shows that the coefficient of  $e^{\ri\cn\cdot x}$  in the Fourier series of $\overline{g_{abs} h_{abs}}$ is $\sum_k |g_{k}|\,|h_{-n-k}|$ (note the minus sign in $-n$), we obtain 
\begin{align*}
|\bTt|\sumcnkm|f_n|\,|g_k|\,|h_m|&=|\bTt|\sum_n\Big( {|f_n|}\overline{\sum_k|g_{k}|\,|h_{-n-k}|}\Big)\\&\stackrel{\eqref{P:thm}}{=}\big\langle f_{abs},\overline{g_{abs} h_{abs}}\big\rangle\le \|f_{abs}  \|_{L^2}\,\| g_{abs} h_{{abs}} \|_{L^2}.
\end{align*}
Therefore,  applying H\"older's inequality on $\|g_{abs} h_{abs} \|_{L^2}$ followed by applying Sobolev inequalities, we show, for {\bf generic} set $\cNg\subset\big(\bZtno\big)^3$ which surely satisfies $\ind_{\cNg}(n,k,m)\le1$, that
\be
\label{tri:inf:fi}
\sumcnkm |f_n|\,|g_k|\,|h_m|\ind_{\cNg}(n,k,m) \lesssim
\begin{cases}
C_\gamma \,\pDn {\frac32+\gamma}f \, \pDn0g  \, \pDn0h\,, &\text{for real \,}\gamma>0,\\
C_\beta\, \pDn {\frac32-\beta}f  \, \pDn {\beta}g \, \pDn0h\,,&
\text{for  real \,}\beta\in(0,\tfrac32).
\end{cases}
\ee
By the same approach together with \eqref{c:p}, \eqref{cn:cm:s}, \eqref{cn:2s}, we obtain the following on $\bR^3$-valued, div-free, zero-mean functions  $\vu,\vv\in H^{s+1}(\bTt)$ (see e.g. \cite[Lemma A.1]{KM:limit} for integer $s$) 
\begin{align}
\nonumber
&\qquad\Big|\big\langle \pD^{s} B( \vu,\vv),\pD^s\vv\big\rangle\Big| =\Big|\big\langle \pD^{s} B( \vu,\vv) - B(\vu,\pD^{s}\vv),\pD^s\vv\big\rangle\Big| \\
& \lesssim \begin{cases}C_{s,\gamma}\pDn {\max\{s,\frac52+\gamma\}}\vu\,\pDn {s}\vv^2,&\text{ for real  \,} s \ge1\text{ and }\gamma>0,\\
C_s\big(\pDn {s+1}\vu\,\pDn {3\over2}\vv+\pDn {3\over2}\vu\,\pDn {s+1}\vv\big)\,\pDn {s}\vv,&\text{ for real  \,}s>0,
\end{cases}
\label{comm}\end{align}
where for $s\in(0,1)$, we also used $|u_k\cdot\cm| \le |u_k|\min\{|\cm|,|\cn|\}$ by incompressibility $u_k\cdot\cm=-u_k\cdot\cn$.

Following from the above, we remark on local-in-time existence and estimates of solutions. For this purpose, 
the formalisms of the original RNS equations \eqref{NS} and our NR approximation \eqref{NR} are regarded in the same way -- in fact, the proof works for any {\it generic} set $\cNg\subset(\bTt)^3$. First, inequality \eqref{comm}(i) can be shown to also hold for $\wt B$ with a generic $\cNg$. Then with
 the initial data for $\tvU$ and $\vU$ given in Theorem \ref{thm:local:e}, one can use the same construction as in the proof of Theorem \ref{thm:NR}, but with 3D-like estimates \eqref{tri:inf:fi}, \eqref{comm} rather than 2D-like estimates of Theorem \ref{thm:2Dlike}, and show local-in-time existence and estimates for both solutions as,
 \be\label{classic:e:e}
\max_{t\in[0,c/E_0^{\frac12}]}\big(\pDnb s{\tvU}^2(t)+\pDnb {s^\flat}{\vU}^2(t)\big) +\mu \int_0^{c/E_0^{\frac12}}\pDnb {s+1}{\tvU}^2(t)+\pDnb {s^\flat+1}{\vU}^2(t)\,\rd t\lesssim_{s,s^\flat} E_0,
\ee
with constants $c,C,s,s^\flat$ described in Theorem \ref{thm:local:e}. Note in particular $s,s^\flat>\frac52$. Also note that, in proving the above, one should change the counterpart of the {\it first} inequality in \eqref{G:ineq} by raising the $H^1$ norm to $H^{\frac52+\gamma}$ for $\gamma>0$ and lowering the $H^{j+1}$ norm to $H^j$. The second inequality of \eqref{G:ineq} is irrelevant here. 
 The resulting Gr\"onwall's inequality is controlled by a Ricatti-type ODE for which the solution remains bounded in time interval $[0,c/E_0^{\frac12}]$ as claimed in \eqref{classic:e:e}.  (Note: for Navier-Stokes equations, see standard results such as \cite[Ch.\,17, Theorem\,4.1, (4.16), Props.\,4.2,\,4.3]{Taylor:3}.)

\begin{proof}[Proof of Theorem \ref{thm:local:e}]
We will use the normal preserving property \eqref{norm:pr} of $\etau$ without reference. Then $\pDn s{\tvU}=\pDn s{\tvu}$. Also, property \eqref{ip:uvw} implies, for any real $s\ge0$, that
\be\label{Ds:B:tau}
\big\langle  \pD^{s}B(\tau;\vu,{ \vv}) ,\pD^{s}\vw\big\rangle 
 =  \big\langle  \pD^{s} B( \etau\vu,{ \etau\vv}) ,\pD^{s}\etau\vw\big\rangle.
\ee
So, tri-linear estimates \eqref{tri:inf:fi}, \eqref{comm} will be used in conjunction with the above as needed.

 Subtract the two transformed systems \eqref{m:NS},  \eqref{m:NR} to find
\be\label{w:eq}
-\pa_t(\vu-\tvu)+\mu\Delta(\vu-\tvu)=B(\Omega t;\vu,\vu-\tvu) + B(\Omega t;\vu-\tvu,\wt \vu) +B(\Omega t;\tvu,\tvu)-\wt B(\Omega t;\tvu,\tvu).
\ee
Recall \eqref{B:full1} and \eqref{wt:B:full1} and define $\cN^c:=\big\{(n, k, m)\in(\bZtno)^3:n+k+m=0\big\}\backslash\cN$. Then rewrite the ``residual''
\begin{align*}
&\qquad B(\Omega t;\tvu,\tvu)-\wt B(\Omega t;\tvu,\tvu)\\
&= \sum_{(n,k,m)\in\cN^c}\sum_{\vec\sigma}\cP_{-n}^{\sigma_1}B(\cP_k^{\sigma_2}\tvu,\cP_m^{\sigma_3}\tvu)\exp(\ri \omvs\Omega t)=\Omega^{-1}\pa_t\vr_\delta-\Omega^{-1}\vr_1,\\
\text{where}\quad \vr_\delta&:=  \sum_{(n,k,m)\in\cN^c}\sum_{\vec\sigma}\cP_{-n}^{\sigma_1}B(\cP_k^{\sigma_2}\tvu,\cP_m^{\sigma_3}\tvu)\,{\exp(\ri \omvs\Omega t)}(\ri \omvs)^{-1}, \\
\text{and}\quad \vr_1 &:=  \sum_{(n,k,m)\in\cN^c}\sum_{\vec\sigma}\cP_{-n}^{\sigma_1}{\pa\over\pa t}\big(B(\cP_k^{\sigma_2}\tvu,\cP_m^{\sigma_3}\tvu)\big){{\exp(\ri \omvs\Omega t)}(\ri \omvs)^{-1}} .
\end{align*}
Let $
\vw:=\vu-\tvu+\Omega^{-1}\vrd
$ 
and recast \eqref{w:eq} into
\be\label{vw:eq}\begin{aligned}
&\quad-\pa_t\vw+\mu\Delta\vw\\
&=B(\Omega t;\vu,\vw)+B(\Omega t;\vw,\wt \vu) 
-\Omega^{-1}\big(\vr_1+{B(\Omega t;\vu,\vrd)} + {B(\Omega t;\vrd,\wt \vu)} -\mu\Delta\vr_\delta\big).
\end{aligned}\ee
By definition \eqref{def:cP} where constant vector $\rvec\ns$ is of unit length, we have, for vector fields $\vu,\vv$,
\[
|\cP_n^\sigma\vu|\le|u_n|\qua{ and }\big|B(\cP_k^{\sigma}\vu,\cP_m^{\sigma'}\vv)\big|\le |u_k|\,|v_m|\,|\cm|.
\]
Recall definition \eqref{def:cN} of   $\cN$ so that, on the complement set $\cN^c$, we have $\left|\om^\vsigma_{ nkm}\right| >\delta=\delta(n,k,m)$. Then by the {\it lower} bound \eqref{de:lower} on $\delta\in(0,1)$,  we find, for any real $\gamma>0$,
\[\begin{aligned}
\Big|\frac{1}{\omvs}\Big|& 
\lesssim_\gamma (|\ck|+|\cm|)^{1+\gamma}\,,\qquad \forall\, (n,k,m)\in\cN^c\,.
\end{aligned}\]
Now, 
for any $\bC^3$-valued $\vh\in L^2(\bTt)$ and any real $ j\ge 0$,  we expand  $\langle\pD^j\vrd,\vh\rangle$ using  the adjoint property \eqref{ad2}; then, using  the above inequalities, we find 
\begin{align*}
 \langle\pD^{j}\vrd,\vh\rangle&\lesssim_{j,\gamma}\hat c_2^{-1}\sumcnkm|\cn|^j(|\ck|+|\cm|)^{1+\gamma}|\wt u_k|\,|\wt u_m|\,|\cm|\,|h_n|.
\end{align*}
Next, switch $k,m$ and add the result to the above, 
applying \eqref{cn:2s} and relaxing the result into two  sums that are equal thanks to $k,m$ symmetry, to find
\begin{align*}
\langle\pD^{j}\vrd,\vh\rangle&\lesssim_{j,\gamma}\hat c_2^{-1}\sumcnkm |\wt u_k|\,|\wt u_m|\,|\cm|^{j+2+\gamma}\,|h_n|,\quad\text{ for real $\gamma>0$}.
\end{align*}
Similarly, we use respectively the definition of $\vr_1$ and the right-hand side of \eqref{m:NR} for $\pa_t\tvu$, and simplify using $k,m$ symmetry where possible, to find, respectively, for $j\ge 0$
\begin{align*}
\langle\pD^{j}\vr_1,\vh\rangle&\lesssim_{j,\gamma}\hat c_2^{-1}\sumcnkm\Big(|(\wt u_t)_k|\,|\wt u_m|\,|\cm|^{  j+2+\gamma} + |\wt u_k|\,|(\wt u_t)_m|\,|\cm|^{  j+2+\gamma}\Big)|h_n|\,,\\
\nonumber\langle\pD^{j}\pa_t\tvu,\vh\rangle&\lesssim_{  j}\sumcnkm |\wt u_k|\,|\wt u_m|\,|\cm|^{j+1}\,|h_n|+\mu\pDn {j+2}\tvu\,\pDn0{\vh}\,.
\end{align*} 
By the above three inequalities, \eqref{tri:inf:fi} and treating $\vh$ as test function, we obtain, for  $s'\in[0,s-3)$,
\begin{subequations}\label{subeq:e}
\begin{align}
\label{e1}\pDn {s'+1}\vrd&\lesssim \hat c_2^{-1}\pDn {\frac32}\tvu\, \pDn {s}\tvu\,,\\
\nonumber\pDn {s'}{\vr_1}&\lesssim \hat c_2^{-1}\pDn{  \frac32}{\pa_t\tvu}\,\pDn {  s-1}\tvu + \pDn {  s-1}{\pa_t\tvu}\,\pDn{  \frac32}\tvu\\
&\lesssim \hat c_2^{-1}\Big(\pDn{  2} \tvu\, \pDn {\frac52}\tvu+\mu\pDn { \frac72}\tvu\Big)\, \pDn {  s-1}\tvu
+\hat c_2^{-1}\Big( \pDn{  2}\tvu\, \pDn s\tvu+\mu\pDn {  s+1}\tvu\Big)\pDn{ \frac32}\tvu\nonumber\\&\le2\hat c_2^{-1}\Big( \pDn{ 2} \tvu\, \pDn s\tvu+\mu\pDn {  s+1}\tvu\Big)\pDn{ \frac32}\tvu,\qquad\text{(by interpolation)}.&
\label{e4}
\end{align}
\end{subequations}
 Here and below the missing constants in the $\lesssim$ notation depend on $s',s$.

For the last three bilinear terms in \eqref{vw:eq}, by the same combination of testing  function, \eqref{cn:2s} (in the spirit of ``endpoint'' estimates) and \eqref{tri:inf:fi}, we  obtain, for real ${s'\in[0},s-3)$ and $\gamma>0$,
\begin{subequations}\label{subeq:B}\begin{align}
\label{B:vw:0}
\pDn {s'}{B(\Omega t;\vw,\wt \vu)}&\lesssim_{s',\gamma} \pDn {s'}\vw\,\pDn {\max\{s'+1,\frac52+\gamma\}}\tvu\,,\\
\pDn {s'}{B(\Omega t;\vu,\vrd)}&\lesssim_{s',\gamma} \pDnb {\max\{s',\frac32+\gamma\}}{\vu}\,\pDn {s'+1}\vrd \,,\\
\pDn {s'}{ B(\Omega t;\vrd,\wt \vu)}&\lesssim_{s',\gamma} \pDn {s'}\vrd\,\pDn {\max\{s'+1,\frac52+\gamma\}}\tvu\,.
\end{align} 
\end{subequations}
(Note: in proving them, one can consider the cases $s'<\frac32,s'=\frac32,s'>\frac32$.)

Take the $L^2$ inner product  of $ \pD^{s'}\!\eqref{vw:eq}$ with $\pD^{s'}\!\vw$ for $s'$ in the range prescribed in Theorem \ref{thm:local:e}. In the resulting right-hand side,  the first term vanishes if $s'=0$ and is bounded using \eqref{comm}(i) if $s'\ge1$. Also, bound the rest using \eqref{subeq:e}, \eqref{subeq:B}. Then we arrive at
\[\begin{aligned}
\frac{\rd}{\rd t}\pDn {s'}{\vw}^2+2\mu\pDn {s'+1}{\vw}^2&\lesssim \big(\pDn {s}\tvu+\pDn {s^\flat}\vu\big)\,\pDn {s'}{\vw}^2+\\
&\;\;\hat c_2^{-1}\Omega^{-1}\big(\pDn {s}\tvu+\pDn {s^\flat}\vu\big)\pDn{s}\tvu\pDn{\frac32}\tvu\pDn {s'}{\vw}+\hat c_2^{-1}\Om^{-1}\mu \pDn {s}\tvu\pDn{\frac32}\tvu\pDn {s'+1}{\vw}.
\end{aligned}
\]
Relax both $O(\Om^{-1})$ terms using Young's inequality so that part of the result cancels the $O(\mu)$ term on the left and part of the result is absorbed into the first term on right,
\be\label{w:d:ineq}
{\rd\over\rd t}\pDn {s'}{\vw}^2
\lesssim\big(\pDn {s}\tvu+\pDn {s^\flat}\vu\big)\pDn {s'}{\vw}^2 +\hat c_2^{-2} \Omega^{-2}\Big[\big(\pDn {s}\tvu+\pDn {s^\flat}\vu\big)\pDn{s}\tvu^2 +\mu\pDn {s}\tvu^2\Big]\pDn{\frac32}\tvu^2.
 \ee
Recall $\vw=\vu-\tvu+\Omega^{-1}\vrd$, so by \eqref{e1} we have
$
\pDn {s'}{\vw(0,\cdot)}^2-\pDn{s'}{\vU_0-\tvU_0}^2\lesssim \Omega^{-2}.
$ 
Then,  integrate the above in time and apply local estimates \eqref{classic:e:e} to show 
\[
\pDn {s'}{\vw(t,\cdot)}^2-\pDn{s'}{\vU_0-\tvU_0}^2\lesssim_{E_0}\int_0^t \pDn {s'}{\vw(t_1,\cdot)}^2\,\rd t_1+\hat c_2^{-2}\Omega^{-2}\,,\qua{\, for}t\in[0,c/E_0^{\frac12}].
\] 
Solve this Gr\"onwall's inequality and apply \eqref{e1} again to prove the required bound on $\vu-\tvu=\vw-\Omega^{-1}\vrd$ which of course holds also for $\vU-\tvU$.
\end{proof}

\begin{remark} \label{re:global:sol}It then requires minimal effort to prove the global-in-time solvability of RNS equations for any rotation rate $\Omega$ above a threshold value that only depends on viscosity $\mu$ and the size and type of the norm of the initial datum $\vU_0$. First, for $\vU_0$ from a sufficiently regular function space, we combine assumptions of Theorems \ref{thm:NR}, \ref{thm:local:e} so that the bandwidth $\delta(n,k,m)$ now satisfies two-sided bounds \eqref{de:up} and \eqref{de:lower}. Further assume 
\[
s^\flat=s'\qqua{and}\vU_0=\tvU_0 \in H^{s+1}(\bTt)\,,
\] 
so by Theorem \ref{thm:NR}, $\pDn {s+1}{\tvu(t,\cdot)}$ is uniformly bounded for all $t\ge 0$.
Now, for as long as
\be\label{vu:100}
\pDn{s'}\vu\le 100\pDn{s}{\tvu},
\ee we deduce from \eqref{w:d:ineq} that
\[
{\rd\over\rd t}\pDn {s'}{\vw}^2
\lesssim101\pDn {s}\tvu\pDn {s'}{\vw}^2 + \Omega^{-2}\Big[101\pDn{s}\tvu^3 +\mu\pDn {s}\tvu^2\Big]\pDn{\frac32}\tvu^2.
\]
By \eqref{NR:en:L2} and $\|\tvu\|_{L^2} \le \|\nabla\tvu\|_{L^2}$ due to $\tvu$ being zero-mean, we have $\|\tvu(t,\cdot)\|_{L^2}$ decays exponentially.  Interpolating this and the above uniform $H^{s+1}(\bTt)$ bound, we find the factor $101\pDn {s}\tvu$ of the above differential inequality can be integrated for $t\in[0,\infty)$ when used as an integrating factor.  Therefore, in view of \eqref{e1}, the assumed $\vU_0=\tvU_0$ and $\vw=\vu-\tvu+\Omega^{-1}\vrd$, we can find a small enough threshold for $\Omega^{-1}$ below which the upper bound on $\vw$ stays low enough to ensure \eqref{vu:100} stays valid for all $t\ge0$, hence making the RNS equations solvable in the $H^{s'}(\bTt)$ space for all $t\ge0$. Secondly, if $\vU_0$ is not regular enough to be covered by the first case, we use the local-in-time smoothing property to bootstrap the solution's regularity. For instance, if $\vU_0\in H^s(\bTt)$ for $s>\frac12$, then by the usual $H^s(\bTt)$ energy method followed by applying \eqref{comm}(ii) on the nonlinear term and interpolating the $H^{\frac32}$ norm in the result, we have
\[
{\rd\over\rd t}\pDn {s}{\vU}^2+2\mu\pDn{s+1}{\vU}^2\le C_s \pDn{s+1}{\vU}^{1+\beta}\pDn{s}{\vU}\pDn{0}{\vU}^{1-\beta},
\]
for some $\beta=\beta(s)\in(0,1)$.
Apply Young's inequality to relax the right-hand side to the sum of $\mu\pDn{s+1}{\vU}^2$ and some finite power of the rest, noting $\pDn{0}\vU$ is non-increasing in time, to obtain a finite bound on $\sup_{t\in[0,t_1]}\pDn{s}\vU$ for some $t_1>0$. This then implies a finite bound on $\mu\int_0^{t_1}\pDn{s+1}\vU^2\rd t$ and thus a finite bound on $\pDn{s+1}\vU^2(t')$ for some $t'\in[0,t_1]$. All these bounds only depend on $\mu,s$ and $\pDn{s}{\vU_0}$. Repeat such process to achieve a sufficiently regular solution $\vU(t,\cdot)$ at some $t>0$, after which time the proof is the same as the above first case.
\end{remark}

\appendix

 \section{\bf Counting Integer Points Bounded by Jordan Curves}\label{s:j:c}

 Inspired by \cite{JS:integer} on estimating integer points inside a planar Jordan curve, we provide the following self-contained proof on such counting problem in a set separated by multiple disjoint Jordan curves. Then, the set in question may overlap either the interior or exterior of an individual Jordan curve. 
 
 Let $cl(\cdot)$ denote the set closure operator.  A curve is simple if it does not cross itself. A curve is rectifiable if it has  finite length. Recall the area of a bounded open set in $\bR^2$, i.e. its Lebesgue measure, is well defined because it is Borel measurable and hence Lebesgue measurable.

A Jordan curve is a simple closed curve in the plane $\bR^2$, namely, it is the image of an injective continuous map of a unit circle into the plane. The Jordan curve theorem, seemingly intuitive but requiring nontrivial proofs,  states that  for any planar Jordan curve $\jc$, its complement $\bR^2\backslash\jc$ consists of two path-connected open subsets called ``exterior'' and ``interior'' regions so that the exterior region consists of all points path-connected to points that are arbitrarily and sufficiently far away from $\jc$, and also, $\jc$ is the boundary of each region.  Note that the ``topological exterior'' of a general set is a different notion which we do not use here. An even more nontrivial Jordan-Schoenflies theorem further asserts that for a Jordan curve as a continuous injection  from unit circle to $\bR^2$, the domain of that injection can be extended so that it is a homeomorphism  of $\bR^2$.
 
 Note: we will use notation $\jc$ whenever the curve is explicitly known as a Jordan curve.
 
 We call  $\gamma$ an ``arc'' of a Jordan curve $\jc$ if it is the image of a closed, positive-lengthed, sub-interval of the unit circle under the same mapping that defines $\jc$.
  
Now, consider 
\be\label{S:intro}
\left.\begin{aligned}&\text{bounded open set }\emptyset\ne S \subset\bR^2 \text{ \,  so that \, } \pa S =\bigcup\limits_{\jc\in{\boldsymbol{\mathcal J}}}\jc\,,\\ 
&\text{where finite set $\Jset$ consists of {\bf disjoint}, rectifiable Jordan curves.}
\end{aligned}\right\}
\ee
 For an integer point $\bp=(x, y)\in\bZ^2$, define the open square box
\[
B _{\bp}:=(x-\tfrac12, y+\tfrac12)\times (x-\tfrac12, y+\tfrac12).
\]
 For any Jordan curve $\jc\in\Jset$, an arc $\gamma\subset \jc\cap cl(B _\bp)$ is called ``maximal'' relative to $B _\bp$ if 
\be\label{g:q12}
\gamma\cap B _\bp \text{ is non-empty, path-connected}\qua{ and }\gamma\cap\pa B _\bp  =\{\bq_1,\bq_2\}  ,
\ee
for points $\bq_1, \bq_2$ which we shall call the  endpoints of $\gamma$. It is possible $\bq_1=\bq_2$. 
Note that a given box $B_\bp$ can break up a Jordan curve $\jc$ into several maximal arcs, but our proof below will be localised entirely within $B_\bp$, assuming no knowledge of such break-ups. See Figure \ref{fig:J}.

 Any point of $B _\bp\cap\pa S $ that is not on a maximal arc   must be on an ``interior'' Jordan curve i.e. one that is entirely in $B _\bp$. 
 In other words,  there exists a set  $\Gset_{\!\bp}$ of curves 
so that
\be\label{J:Q}
\begin{aligned}
&B _\bp\cap\pa  S  = \bigcup_{\gamma \in \Gset_{\!\bp}}\Big(\gamma\backslash\big\{\text{its endpoint(s)}\big\}\Big)  \\
&\text{where each $\gamma\in\Gset_{\!\bp}$ is either a maximal arc or an interior Jordan curve.}
\end{aligned}
\ee
There are countably many\footnote{\,Consider summability and $\bR^+=\cup_{j\in \bZ}[2^{-j},2^{-j+1})$.  } 
such $\gamma$'s, and their total length $Len(B _\bp\cap\pa S )$ is defined and finite.
An element of $\Gset_{\!\bp}$ can only intersect $\pa B _\bp$ at 0,1 or 2 points. This then excludes any positive-length part of $\pa B _\bp \cap \pa S$ from our proofs, which is OK as far as the {\it upper bound} in Lemma \ref{lem:m:Jordan} is concerned. The finiteness of $\Jset$ however may be lost. For example, consider $\bp=(0,0)$ and this modified version of ``topologist's sine curve'' 
$y=\big|\sin(1/x)\big|\,x^2-\tfrac12$ in the $x-y$ plane.

Figure \ref{fig:J} helps visualisation of the following notions. For points $\bp,\bq$, let $\overline{\bp\bq}$ denote the {\it straight} line segment bounded by and including $\bp,\bq$. For any two points { $\bq_1,\bq_2\in \pa B _\bp$}, let $\wh{\bq_1\bq}_2$ denote the {\it piecewise linear} segment of the box's edge $\pa B _\bp$ that is bounded by and including $\bq_1,\bq_2$ so that its length is less than or equal to that of the rest of $\pa B _\bp$. 
For  a maximal arc   $\gamma$ relative to $B _\bp$ with endpoints $\bq_1,\bq_2$,   let $\Omega_\gamma$ denote the {interior} region of the {\it concatenated} Jordan curve $\gamma \cup \wh{\bq_1\bq}_2$. For  an interior Jordan curve $\jc$ in $B_\bp$, naturally let  $\Omega_\jc$ denote the interior region of $\jc$. 

\begin{proposition}\label{prop:L1}
For any $\gamma\in\Gset_{\!\bp}$, the following two statements hold.
\be\label{A:L}
Len(\gamma) < 1\qqua{implies}Area(\Omega_\gamma) <  Len(\gamma).
\ee
 \be\label{ma:notL1}
\bp \in cl(\Omega_\gamma) \; \text{ for a maximal arc }\gamma   \qua{implies} Len(\gamma) \ge1.
 \ee
\end{proposition}
Intuitively, \eqref{A:L} helps the proof of Lemma \ref{lem:m:Jordan}  in the sense that if we don't already have $Len(\gamma) \ge1$ to cover the 1 integer count of $\bp$, then for every   $Area(\Omega_\gamma)$ removed from $Area(B_\bp)=1$ (thus not contributing to $Area(S)$), compensation is provided via the contribution of $Len(\gamma)$ to $Len(\pa S)$. Also, \eqref{ma:notL1} highlights the challenging case of $\bp \notin cl(\Omega_\gamma)$.
\begin{proof}
 
Statement \eqref{A:L} for an {\it interior} Jordan curve $\gamma$ follows from the isoperimetric inequality.

Consider  maximal arc  $\gamma$ relative to $B _\bp$ with endpoints $\bq_1,\bq_2$. By simple geometry, we have
\[
Len(\wh{\bq_1\bq}_2) \le \sqrt 2\, Len(\overline{\bq_1\bq}_2) \le \sqrt 2 \, Len(\gamma).
\]
Then by the isoperimetric inequality and $Len(\pa\Om_\gamma)=Len(\gamma)+Len(\wh{\bq_1\bq}_2) $, the above leads to 
\[
Area(\Omega_\gamma) < \tfrac{(1+\sqrt2)^2}{4\pi}\left(Len(\gamma)\right)^2,
\] hence proving \eqref{A:L}.

For the proof of \eqref{ma:notL1}, assume $\bp \in cl(\Omega_\gamma)$ 
but suppose instead $Len(\gamma) < 1$. Then $\wh{\bq_1\bq}_2$ overlaps with at most 2 sides of $\pa B _\bp$,  with the understanding that overlapping with a corner of $B$ is on 2 sides.
Let $\bq_{opp}$ be the point that is symmetric to $\bq_1$ about the centre $\bp$. See Figure \ref{fig:J}. The overlapping argument we just made then implies $\bq_{opp} \in \pa B _\bp \backslash \wh{\bq_1\bq}_2$, so $\bq_{opp}$ is in the exterior region of  the concatenated Jordan curve { $\pa\Omega_\gamma$} since one can path-connect it to any point outside the box without touching  { $\pa\Omega_\gamma$}. Then, by $\bp\in cl(\Omega_\gamma)$  and the Jordan curve theorem, the closed line segment $\overline{\bp\bq_{opp}}$ includes a point $\bq_{arc}\in\overline{\bp\bq_{opp}}\cap \gamma\cap B_\bp$ (may not be unique). See Figure \ref{fig:J}. Then 
\[
Len(\overline{\bq_1\bq_{arc}})=Len(\overline{\bq_1\bp})+Len(\overline{\bp\bq_{arc}}).
\]  Helped by the geodesic nature of straight lines  and $dist(\bp,\pa B_\bp)=\frac12$, this implies: 
\[
Len(\gamma)\ge Len(\overline{\bq_1\bq_{arc}}) + Len(\overline{\bq_{arc}\bq_2})\ge Len(\overline{\bq_1\bp}) + Len(\overline{\bp\bq_2}) \ge 1.
\]
Contradiction to $Len(\gamma)<1$ supposed before!
\end{proof}

\begin{figure}[h]
\centering
\hspace{-20mm}
\includegraphics[width=0.45\textwidth]{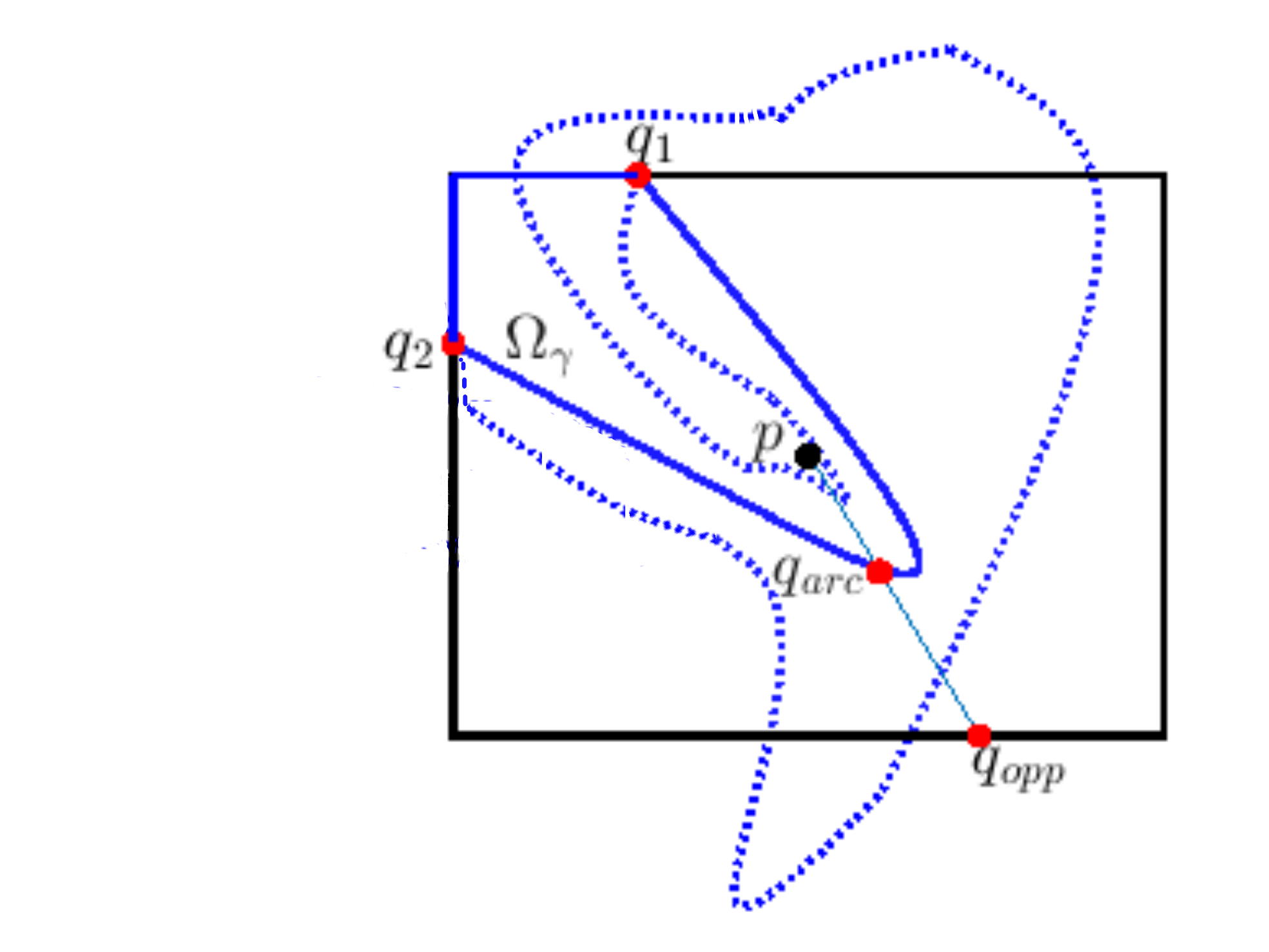}
\caption{\small The boundary of $\Omega_\gamma$ is in solid blue. The maximal arc $\gamma$ (solid blue curve inside the box) and the dotted blue curve are joined to form a Jordan curve from $\Jset$. The position of the dotted blue curve is immaterial to our proofs concerning $\Omega_\gamma$.  }
\label{fig:J}
\end{figure}

In the following main result on planar integer point counting, we have to address exceptional cases where an integer point is surrounded by a possibly very  small Jordan curve.

\begin{lemma}\label{lem:m:Jordan}
For open set $ S $ introduced in \eqref{S:intro}, we have  
\[
\# \big\{\bZ^2\cap cl( S )\big\} \le Area( S ) +  Len(\pa S)+| E|
\]
with the exceptional set  
\[
\begin{aligned}
E:=\big\{\bp\in \bZ^2\cap cl( S )\,: \,&\textnormal{ there exists an \underline{interior} Jordan curve }\jc\in\Jset \textnormal{ so that} \\
&\;\,\bp\in cl(\Omega_\jc)\textnormal{\, and \,}Area(\jc)+Len(\jc) <  1 \big\}.
\end{aligned}
\]
\end{lemma}
\begin{proof}

It suffices to show, for any $\bp\in\big(\bZ^2\cap cl( S )\big)\backslash E$ and $B:=B _\bp$, we have
\[
B \cap \pa S\ne\emptyset\qqua{implies}Area(B  \cap  S ) + Len(B \cap\pa S ) \ge 1.
\]
 Recall \eqref{J:Q} and the text leading to it. Let $\Gset:=\Gset_{\! \bp}$. The above then amounts to:
\be\label{AL1}
B \cap \pa S\ne\emptyset\qqua{implies}Area(B  \cap  S ) +\sum_{\gamma \in \Gset} Len( \gamma ) \ge 1.
\ee 

  If $\bp\in cl(\Omega_\jc)$ for an {\it interior} Jordan curve $\jc\in\Gset$, then  $Area(\jc)+Len(\jc)\ge1$ due to $\bp\notin E$, which proves \eqref{AL1}.
If $\bp\in cl(\Omega_\gamma)$ for maximal arc $\gamma\in\Gset$,  then  \eqref{AL1} follows from \eqref{ma:notL1}.

Therefore from now on, for such $\bp\in\bZ^2\cap cl( S )$, we assume  
\be\label{bp:notin}
\bp \notin  cl(\Omega_\gamma)\qua{ for any} \gamma \in \Gset\,,\qua{ so}\bp\in S .
 \ee
 
 Define
 \[
 S_{path}=\big\{\,\bp_{in} \in B  \backslash \pa  S \big|\,\exists \text{ a path, disjoint from $\pa S  \cup \pa B $, that connects $\bp$ to $\bp_{in}$}\big\}\cup\big\{\bp\big\},
 \] 
which is open since any closed path in $\bR^2\backslash(\pa S \cup \pa B)$ is positively distanced from $\pa  S \cup \pa B $. 
For any point $\bp_{out}\notin  S $, we consider any path that connects $\bp$ to $\bp_{out}$ and is parametrised by $\bpsi:[0,1]\mapsto\bR^2$. Then $\bpsi(0)=\bp\in S$ and the openness of $S$ implies that  $\big\{t\in(0,1]\,\big|\,\bpsi(t)\subset  S \big\}\ne \emptyset$, so its supremum corresponds to a point on $\pa S $. By the arbitrariness of $\bpsi$, we find $\bp_{out}\notin S_{path}$,  thus
\be\label{A1}
S_{path}\subset  B \cap S.
\ee

 Next, by the definition of $S_{path}$, \eqref{bp:notin} and Jordan curve theorem,  we have
\be\label{S:B}
S_{path} \subset B  \backslash \Big(\bigcup_{\gamma \in \Gset }cl(\Omega_\gamma)\Big).
\ee
 On the other hand, it is an intuitive but non-trivial statement to assert that the two sides above are equal, as suggested by the existence of ``lakes of Wada'', three non-overlapping, connected open sets that share the same boundary! To this end, take any point 
 \be\label{bp1:BA}
 \bp_1\in B \backslash \Big(\bigcup_{\gamma \in \Gset }cl(\Omega_\gamma)\Big)
 \ee
 and we prove $\bp_1\in S_{path}$ as follows. 
 In fact, parametrise the line segment $\overline{\bp\bp_1}$ as
 \[
 \bpsi(t)=(1-t)\bp+t\bp_1\qua{ for } t\in[0,1].
 \]
Since $\bpsi(0)=\bp\in S_{path}$ which is open, the following  is defined:
\[
t^-:=\sup\big\{t\in(0,1)\,\big|\,\bpsi\big([0,t]\big)\cap \pa  S  = \emptyset\big\}>0.
\]  
Then $\bpsi(t^-)$ satisfies one and only one of the following.
\begin{itemize}[leftmargin=*]
\item[(i)]
$\bpsi(t^-)\notin \pa  S $. Then the maximality of $t^-$ and the closedness of $\pa S$ implies  $t^- = 1$. Thus, $\bp$ and $\bp_1$ are connected by line segment $\bpsi\big([0,1]\big)$ which does not intersect $\pa  S  \cup \pa B $, proving $\bp_1\in S_{path}$.\medskip

\item[(ii)] $\bpsi(t^-)\in \gamma^\bpsi $ for some $\gamma^\bpsi  \in \Gset$. Then by definition of $S_{path}$ and \eqref{S:B}, we have
\be\label{t-:e}
\text{$\bpsi\big([0,t^-)\big)\subset S_{path}$,\, \; thus it is in the exterior of Jordan curve }{ \pa\Omega_{\gamma^\bpsi}}.
\ee
By  \eqref{bp1:BA}, we have $\bpsi(1)=\bp_1\in \bR^2\backslash cl(\Omega_{\gamma^\bpsi})$ which is open,  so 
\[
t^+:=\inf\big\{t\in(0,1)\,\big|\,\bpsi\big([t,1]\big)\cap cl(\Omega_{\gamma^\bpsi }) = \emptyset\big\}<1.
\] Then $\bpsi(t^\pm)\in \gamma^\bpsi $ and $t^-\le t^+$. {  Let $\gamma^\bpsi_{sub}\subset \gamma^\bpsi$ denote the closed sub-curve (possibly degenerate to a point) of $\gamma^\bpsi$ enclosed between $\bpsi(t^\pm)$ inclusive. 
 By the definition of interior Jordan curve and definition \eqref{g:q12} of maximal arc, and the fact that $\bpsi(t^-),\bpsi(t^+)\subset B$, we have the closed sets $\gamma^\bpsi_{sub}$ and $\pa B$ are disjoint, hence
\[
d:= dist\big(\gamma^\bpsi_{sub},\pa B\big)>0.
\]  Let $B_1$ denote the set of all points in $B$ that are more than $\frac12 d$ away from $\pa B$. Obviously any point not in $B_1$ is at least $\frac12 d$ away from $\gamma^\bpsi_{sub}$. Any point in $B_1\cap(\pa S\backslash\gamma^\bpsi)$ must belong to some   $\gamma\in\Gset$ satisfying \eqref{J:Q} and distinct from $\gamma^\bpsi$, which implies the closed set $\gamma\cap\overline{B_1}$ is disjoint from $\gamma^\bpsi_{sub}$. Also, since $\gamma$ is either an interior Jordan curve or a maximal arc of length no less that $d$ (due to $\gamma\cap B_1\ne\emptyset$), the total number of such $\gamma$'s is finite. Then,
\[
dist\big(\gamma^{\bpsi}_{sub}\,, \,\pa B \cup(\pa S\backslash\gamma^\bpsi)\big)>0 .
\]}

Now, by the Jordan-Schoenflies theorem, let ${\mathfrak h}$ be a homomorphism of $\bR^2$ that maps the Jordan curve $ \pa\Omega_{\gamma^{\bpsi}}$ to the unit circle while preserving the respective exterior (interior) region.  By the local uniform continuity of ${\mathfrak h}^{-1}$ and the above positivity, we have
\[
dist\big({\mathfrak h}({ \gamma^{\bpsi}_{sub}})\,,\, {  {\mathfrak h}(\pa B\cup(\pa S\backslash\gamma^\bpsi))}\big)>0 
\] 
Recall \eqref{t-:e} and also that definition of $t^+$ implies $\bpsi\big((t^+,1]\big)$ is in the exterior of $\pa\Omega_{\gamma^{\bpsi}}$. 
But  $\bpsi(t^\pm)\in { \gamma^{\bpsi}_{sub}}$. Since topological properties remain valid under ${\mathfrak h}$, we use  continuity of ${\mathfrak h}\circ \bpsi(t)$ to find sufficiently small $\ep>0$ so that 
 {  ${\mathfrak h}\circ\bpsi(t^--\ep)$ and ${\mathfrak h}\circ\bpsi(t^++\ep)$} are   in the exterior of ${\mathfrak h}(\pa\Omega_{\gamma^{\bpsi}})$ and are sufficiently close to ${\mathfrak h}({ \gamma^{\bpsi}_{sub}})$ which is an arc of the circle ${\mathfrak h}(\pa\Omega_{\gamma^{\bpsi}})$.
 Thus, the above positive distance and the simple  geometry of circle ensures  
 ${\mathfrak h}\circ\bpsi(t^\pm\pm\ep)$ can be connected by a path disjoint from ${\mathfrak h}(\pa S\cup \pa B)$. {  This path can possibly intersect the part of $\bpsi(t)$ where $t\in[0,t_--\ep]\cup[t_++\ep,1]$, which is OK as far as path-connectedness is concerned.} 
 Then, in view of \eqref{t-:e}, we  show \(\bpsi(t^++\ep) \in S_{path}\).
\end{itemize}

 If need be, iterate the above starting from below \eqref{bp1:BA}, but focusing on the line segment between $\bpsi(t^++\ep)$ and $\bp_1$ so that
 the first step is the definition of a new $\bpsi(\cdot)$  with the new $\bpsi(0)$ being $\bpsi(t^++\ep) \in S_{path}$ from the previous iteration and $\bpsi(1)$ always being $\bp_1$.   Since {  all $\gamma^\bpsi$ involved in the iteration form a finite set (because it is either an interior Jordan curve or a maximal curve of length no less than $2dist(\bp_1,\pa B)$)}, and since  the definition of $t^+$ prevents  the iteration from intersecting the same $\gamma^\bpsi\in\Gset$ twice, we must have the iteration of item (ii) halt in finite steps, after which we must reach item (i), hence showing $\bp_1\in S_{path}$. 
 
 Therefore we have shown
the two sides of \eqref{S:B} are indeed identical, implying
\[
Area(S_{path}) \ge 1- \sum_{\gamma \in \Gset }Area(\Omega_\gamma).
\] 
Finally, either  $\sum_{\gamma \in \Gset }  Len(\gamma)\ge1$ which trivially proves \eqref{AL1}  or otherwise, every $\gamma\in \Gset $ satisfies $Len(\gamma)<1$, thus proving \eqref{A:L} of Proposition \ref{prop:L1}.   Substitute the result into the above inequality and apply \eqref{A1} to prove \eqref{AL1}.
 \end{proof}

 
 \section{\bf Elliptic Integrals}\label{s:e:i}
 
 The following result on elliptic integrals  is elementary but convenient.

\begin{proposition}\label{prop:ell}
Consider $a>y>b>c>d$ and parameters $\sk^2(a, b, c, d)$, $\sg(a, b, c, d)$ defined in \eqref{def:sk}, \eqref{def:sg}. 
Then, 
\begin{align}
\label{int:yz}
&\int_{(b,y]}\;\frac {\rd \lambda} {\sqrt{-(\lambda-a)(\lambda-b)(\lambda-c)(\lambda-d)}} = \sg \int_0^{y_0} \frac {\rd \lambda_0} {\sqrt{(1-\lambda_0^2)(1-\sk^2\lambda_0^2)}},\\
&\qquad
\nonumber\text{ for \; \; } y_0=\sqrt{ \frac {(a-c)(y-b)} {(a-b)(y-c)} }<1 .
\end{align}
Also, there exists positive constant $C_1$ independent of $a, b, c, d$ so that
\be\label{int:yz:est}
\int_{(d,c)\cup(b, a)}\;\frac {\rd \lambda} {\sqrt{-(\lambda-a)(\lambda-b)(\lambda-c)(\lambda-d)}} \le C_{1}\,\sg\cdot\big(1-\log\sqrt{1-\sk^2}\,\big).
\ee

If, in addition, there exist positive constants $C_2,C_3$ so that 
\be\label{int:yz:est:sharp:c}
\frac{y-b}{a-b}\le C_2<1 \text{ \; and \; } \frac{y-b}{b-c}\le C_3,
\ee 
then a refined estimate
\be\label{int:yz:est:sharp}
\eqref{int:yz} \le C_4\,\sg.
\ee
holds for a positive constant $C_4=C_4(C_2,C_3)$ that is otherwise independent of $a, b, c, d,y$.
\end{proposition}
\begin{proof}
By change of coordinates
\[
\lambda_0=\sqrt{ \frac {(a-c)(\lambda-b)} {(a-b)(\lambda-c)} } = \sqrt { \frac{\EC(a,\lambda, b, c)}{\IL(a,\lambda, b, c)} }\qua{for any}\lambda\in(b, a),
\]
we use identity \eqref{sum:id} to have
\[
(1-\lambda_0^2)(1-\sk^2\lambda_0^2)=\frac{\SP(a,\lambda, b, c)}{\IL(a,\lambda, b, c)}\left(1-\frac{\SP(\lambda, b, c, d)}{\IL(\lambda, b, c, d)}\right)=\frac{(a-\lambda)(\lambda-d)} {(a-b)(b-d)}\left( \frac{b-c} {\lambda-c} \right)^2.
\]
Also, apparently
\[
\rd \lambda_0 = \frac 12 \sqrt{ \frac {a-c} {a-b} } \, \sqrt{ \frac {\lambda-c} {\lambda-b} }\, \frac {b-c} {(\lambda-c)^2}\,.
\]
Therefore we prove \eqref{int:yz}. The $y_0<1$ part is due to the rearrangement inequality (or \eqref{sum:id}).

Next, in   \eqref{int:yz}, we let $y\nearrow a$ or equivalently $y_0\nearrow 1$, relax the factor $(1+\lambda_0)(1+\sk \lambda_0)\ge1$ on the right-hand side and then treat the cases of $0< \sk^2<\frac12$ and $\frac12\le \sk^2<1$ differently to prove 
\[
\int_{(b, a)}\;\frac {\rd \lambda} {\sqrt{-(\lambda-a)(\lambda-b)(\lambda-c)(\lambda-d)}} \le\text{ right-hand side of }\eqref{int:yz:est}.
\] Apply this result to $-d>-c>-b>-a$, noting $\sk^2(-d,-c,-a,-b)=\sk^2(a, b, c, d)$ and   $\sg(-d,-c,-a,-b)=\sg(a, b, c, d)$, to show the entire   \eqref{int:yz:est}.

 \eqref{int:yz:est:sharp}  follows from  relaxing $1-\sk^2\lambda_0^2>1-\lambda_0^2$ in \eqref{int:yz}, and then using identity and estimate
\[
1-y_0^2=  { \frac{1-\frac{y-b}{a-b}} {\frac{y-b}{b-c}+1} }\ge {  \frac{1-C_2} {C_3+1} }>0.
\]
\end{proof}
 
 
\section{\bf Optimality of Choice of Bandwidth $\delta$ and Integer Count}\label{S:op}

We establish {\it lower} bounds for number of $k$ associated with $\cN$ via direct counting (not via volume estimate).
If one accepts that estimate \eqref{ci} with $\beta=1$ is a condition that is not only sufficient but also technically necessary for the proof of Theorem \ref{thm:2Dlike} to work,  then the ``optimality of \eqref{de:up}'' means that \eqref{de:up} being valid asymptotically for arbitrarily large wavenumbers is necessary for \eqref{ci} with $\beta=1$ to hold for all $|\cn|$. 

Set $\sL_1=\sL_2=\sL_3=1$ for simplicity.   Consider set $\cNp\subset(\bZ^3)^3$ given in Lemma \ref{restricted:L} for $\cN$ defined in \eqref{def:cN} with $\delta$ subject to \eqref{de:max}.  By Theorems \ref{thm:int:vol} and  \ref{thm:V},  the count of integer points $k$ as in  $\sum_{ \substack{ k \in \bZt }} \ind_{\cNp}(n,k,-n-k)$ is shown to have  an {\it upper} bound   at the order  of
\be\label{up:also:low}
|n|^2+|n|^3\,\delta+|n|^3\,\delta\log^+\dfrac1{\delta+\frac{2|n_3|}{|n|}}\,,\quad\text{\, for any large \,}|n|.
\ee  
In the following main results of this appendix, we provide examples to show that the same type of integer point count has {\it lower} bounds of the same order for $n_3=0$ and $n_3\ne0$ (with a change in the latter case that is justified in the remarks), hence showing that  if $\delta$ violates \eqref{de:up} then \eqref{ci} with $\beta=1$ can not hold.
 \begin{proposition}\label{prop:le1}
 For $\delta\in(0,\frac12)$ and $n\in\bZt$, 
 the number of  all $k\in\bZt$ subject to constraints
 \begin{align}\label{prop:c1}
& |n|\ge|k|\ge|n+k|\ge\frac13|n|,\\
\intertext{(so all 3 wavenumbers are ``localised''), and}
 \label{prop:c2}& \left|\frac{n_3}{|n|}+\frac{k_3}{|k|}-\frac{n_3+k_3}{|n+k|}\right| \le \delta,
 \end{align}
for large enough $|n|$ is at least of order:
 \begin{itemize}
\item[{\textnormal{(i)}}] $|n|^2+|n|^3\,\delta\log\frac1\delta$ under additional constraints $n_3=0$ and $k_3\ne0$; 
\item[{\textnormal{(ii)}}] $|n|^2\,\min\big\{1,(\delta|n|)^{3\over2}\big\}+ |n|^3\,\delta\min \big\{\log\frac1\delta,\delta|n|\log|n|\big\}$ under additional constraints $n_3k_3(n_3+k_3)\ne0$ and $1\lesssim\delta|n|^2$.
 \end{itemize}
 
 The first lower bound also holds for $\delta=0$ if $0|\log0|=0$ is understood. 
 \end{proposition}
 
 We make some remarks. By dispersion relation \eqref{om:eig} for the large operator $\Omega\cL$ that is responsible for fast waves, any term in eigen-expansion \eqref{e:eig:exp} with purely horizontal wavenumber, i.e. $n_3=0$, is {\it independent} of the fast time $\tau$, so it is called a slow mode in the dichotomic framework used in most literature whereas any mode with $n_3\ne0$ is regarded fast.  In the case of three slow/horizontal modes,  the  $n_3=k_3=m_3=0$ condition makes the counting elementary, thus we omit the detail. There does not exist any triplet of one fast and two slow  modes due to  $n_3+k_3+m_3=0$. Then Proposition \ref{prop:le1} addresses both triplets of slow-fast-fast interaction and fast-fast-fast interaction -- we made an effort on this separate treatment as it appears frequently in literature. 
The estimate for either case {\it alone} can justify the use of logarithmic factor in \eqref{de:up}. Also note the result for former case attests to the optimality of upper bound \eqref{up:also:low} for every $\delta\in[0,\frac12)$ and the result for the latter case deviates from \eqref{up:also:low} for $|n|^{-2}\lesssim\delta\ll |n|^{-1}$. Such degeneracy is however expected. In fact, it can be rigorously argued in the limiting case $\delta=0$ for which it was shown in \cite[Lemma 4.1]{KY:number} that the number of integer solution to the Diophantine {\it equation} associated with purely fast interaction has an  $O(|n|^{1+\ep})$ upper bound for any $\ep>0$, therefore it is not possible to close the gap with \eqref{up:also:low}. We should note that this result does not apply to all aspect ratios of the domain (\cite[Remark 4.2]{KY:number}).
\begin{proof} 
 For case (i), let $N$ be any large positive integer and $n=(-2N,0,0)$.
 Consider set 
 \be\label{def:Sd}\bigg\{(k_1 ,k_2,k_3)\in\bZt \, :\, k_1 \in\Big[N,N+\frac N3\min\big\{\frac{\delta |n|}{2k_3},\, 1\big\}\Big]\text{\, and \,}\max\{|k_2|,|k_3|\}\le    N,\,k_3\ne0\bigg\}.
 \ee
 It is elementary to show   any  $k$ from this set  satisfies \eqref{prop:c1} and  also, for $m:=-n-k$, we have
\be\label{km:k+m}
 |k|\,|m|\ge k_1 m_1=k_1(2N-k_1)\ge \frac89 N^2,\qua{ and } |k|+|m|\ge |n|.
 \ee
 Then any such $k$ satisfies \eqref{prop:c2} since
 \be\label{nkm:c1}
-\frac{n_3}{|n|}-\frac{k_3}{|k|}-\frac{m}{|m|}  =\frac{k_3(|k|-|m|)}{|k|\,|m| } 
 = \frac{4k_3N(k_1-N)}{|k|\,|m|\,(|k|+|m|)}\in\big[-\frac34\delta,\frac34\delta\,\big].
 \ee
At a fixed  $k_3$, the count of integer pairs $(k_1,k_2)$ admissible by \eqref{def:Sd} is  
\[
(2N+1)\left(1+\left\lfloor\frac N3\min\{\frac{\delta N}{k_3},1\}\right\rfloor\right).
\]
With $1+\lfloor x \rfloor >\frac12(1+x)$ for $x\ge 0$, 
 its ${k_3}$-sum for  $|k_3|\in[1,N]$ can be bounded from below by 
\[\begin{aligned}
  (2N+1)N+\frac{(2N+1) N}3 \int_1^{N+1} \min\big\{\frac{\delta N}{k_3},\,1\big\}\,\rd k_3.
\end{aligned}\]
The integral equals $\delta N\log(N+1)$ if  $0<\delta N< 1$, and equals
$\delta N\log\frac{N+1}{\delta N}+(\delta N-1)$ if $\delta N\ge 1$. 
By using $x\log x\ge -e^{-1}$, we lower the integral value in the $0<\delta N<1$ case as  
 \[ \delta N\log (N+1)>\delta N \log\frac1\delta+{\delta N}\cdot\log(\delta N)\ge \delta N\log\frac1\delta-{e^{-1}},\]
which apparently also works as lower bound for the $\delta N\ge1$ case. Noting $|n|=2N$,  we prove the first case of the proposition which also holds for $\delta=0$ if $0|\log0|=0$ is understood.  
 
For case (ii), one may consider a proof by treating it as perturbation of case (i). In a nutshell, let $n'=(-2N,0,\alpha)$ and $k'=(k_1,k_2,k_3-\alpha)$ for any  $(k_1,k_2,k_3)$ from \eqref{def:Sd}. With $m'=-n'-k'$,  estimate the difference between \eqref{nkm:c1} and $-\frac{n_3'}{|n'|}-\frac{k_3'}{|k'|}-\frac{m_3'}{|m'|}$ at $\alpha=1$ by considering the $\alpha$ derivative and showing that the difference is   $O({|n|}^{-1})$. Then,  using the first case, one can prove the second case for  as long as $1\lesssim \delta|n|$. To cover a wider parameter range as claimed in the proposition, we adopt the following more refined approach.

  For large integer $N$, let $n=(-2N,0,-1)$ and consider  set  
 \be\label{def:Ssd}
 \begin{aligned}
\bigg\{&(k_1 ,k_2,k_3)\in\bZt \, :  \; \, \frac{|n|}{1+\delta|n|}\le |k|\le |n|,\\
&\quad k_1 \in\Big[N, N+\frac N3\min\big\{\frac{\delta |n|}{2(k_3-1)},\, 1\big\} \Big]
\text{ \, and \, } k_3\in \Big[2,\frac N3\min\big\{\sqrt{\delta |n|},1\big\}\Big]\;\bigg\}.
\end{aligned}
\ee
It is elementary to show   any $k$ from this set  satisfies \eqref{prop:c1}.
Next, with  $m:=-n-k$, consider
\[
-\frac{n_3}{|n|}-\frac{k_3}{|k|}-\frac{m_3}{|m|}=f^\flat(k)+f^\sharp(k):=\Big(\frac{1}{|n|}-\frac1{|k|}\Big) +\Big( \frac{k_3-1}{|m|}-\frac{k_3-1}{|k|} \Big).
\]
By \eqref{prop:c1}, we find $f^\flat\le0$ and $f^\sharp\ge 0$. By the first inequality of \eqref{def:Ssd}, we have $f^\flat\ge-\delta$.
For the $f^\sharp(k)$ term, we use the  $\delta$-dependent upper bounds of $k_1,k_3$ in \eqref{def:Ssd} to have \[
f^\sharp(k)=(k_3-1)\cdot\frac{|k|-|m|}{|k|\,|m|} =(k_3-1)\cdot\frac{4N(k_1-N)+2k_3-1}{|k|\,|m|\,(|k|+|m|)} \le\frac{ \frac46 N^2\delta |n|+\frac29 N^2\delta |n|}{|k|\,|m|\,(|k|+|m|)} .
 \]
 Since it is easy to show \eqref{km:k+m} still holds, we find \( f^\sharp\le\delta\), hence proving \eqref{prop:c2}.

We move on to counting of \eqref{def:Ssd} where the constraints on $k_1,k_3$ are independent of $k_2$. Then, at any such admissible $k_1,k_3$, 
 the range of $|k_2|$ 
  is of length
\[\begin{aligned}
 &\qquad\sqrt{|n|^2 -k_1^2-k_3^2} - \sqrt{ \left(\tfrac{|n|^2}{(1+\delta|n|)^2}-k_1^2-k_3^2\right)^+}      \\
 & >\begin{cases}\sqrt{\frac{19}{36}}\,|n|,&\text{if }\frac{|n|^2}{(1+\delta|n|)^2}\le k_1^2+k_3^2,\\
 \frac{|n|^2-\frac{|n|^2}{(1+\delta|n|)^2}}{ \sqrt{|n|^2 -k_1^2-k_3^2} +\sqrt{ \tfrac{|n|^2}{(1+\delta|n|)^2}-k_1^2-k_3^2} }\ge\delta|n|^2\,\dfrac{\frac{2+\delta|n|}{(1+\delta|n|)^2}}{\sqrt{0.75}+\sqrt{(1+\delta|n|)^{-1}-0.25}},&\text{if }\frac{|n|^2}{(1+\delta|n|)^2}> k_1^2+k_3^2,
 \end{cases}
\end{aligned}\]
where we used 
$\frac{17}{36}|n|^2>k_1^2+k_3^2>\frac14|n|^2$. This last bound also implies that the assumption for the second case above is only possible when $\delta|n|<1$, a condition that can be used to show that the second lower bound is $O(\delta|n|^2)$. Note there are at least $(b-a)$ integers in $[a, b]$. In summary, at any $(k_1,k_3)$ admissible by \eqref{def:Ssd}, the number of integer $k_2$ has a lower bound
\[
\min\{1,\delta|n|\}\,O(|n|),\quad\text{  provided \; $1\lesssim\delta|n|^2$}.
\]

Independently, we estimate $\left(\sum_{k_3}\sum_{k_1}1\right)$, the count of integer pairs $(k_1,k_3)$ admissible by \eqref{def:Ssd}. When $\delta|n|\ge1$, we use $1+\lfloor x \rfloor > \frac12(1+x)$ for $x\ge 0$ to find a lower bound as
\[\begin{aligned}
 &\sum_{k_3=2}^{\left\lceil \delta|n|/2 \right\rceil}\Big(1+\big\lfloor\tfrac N3 \big\rfloor\Big)
+ \sum_{k_3=\left\lceil \delta|n|/2 \right\rceil+1}^{\lfloor N/3\rfloor}\Big(1+\big\lfloor\tfrac N3\,\tfrac{\delta |n|}{2(k_3-1)}\big\rfloor\Big)\\
&>\frac12\lfloor N/3\rfloor+ \frac N6\Big(\frac12\delta|n|-1\Big)^++\frac N{12}\Big(\int_{\left\lceil \delta|n|/2 \right\rceil+1}^{\left\lfloor N/3\right\rfloor+1}\frac{\delta|n|}{k_3-1}\,\rd k_3\Big)^+.
\end{aligned}\] 
 At large $N$, the right-hand side is of order $N+ N^2 \delta +  N^2\delta\log^+\frac1{3\delta}$ which has a lower bound of  order $|n|+|n|^2\delta\log\frac1\delta$ since $|n|^2=4N^2+1$ and $\delta\in(0,\frac12)$. The case when $\delta|n|<1$ is estimated similarly with a lower bound on $\left(\sum_{k_3}\sum_{k_1}1\right)$ to be
\[
\sum_{k_3=2}^{\left\lfloor(N/3)\sqrt{\delta|n|}\right\rfloor}\Big(1+\big\lfloor\tfrac N3\,\tfrac{\delta |n|}{2(k_3-1)}\big\rfloor\Big)>\frac12\lfloor N/3\sqrt{\delta|n|}\rfloor+\frac N{12}\int_2^{\left\lfloor N/3\sqrt{\delta|n|}\right\rfloor+1}\frac{\delta|n|}{k_3-1}\,\rd k_3.
\]
Using $\sqrt{\delta||n|}\log\sqrt{\delta|n|}\ge -e^{-1}$, we can show the right-hand side has an asymptotic lower bound of order $N\sqrt{\delta|n|}+N\delta|n|\log\frac N6$. The proof of case (ii) is complete.
\end{proof}

\bibliographystyle{alpha}

\begin{thebibliography}{abcde}

\bibitem{AL:Nature}
K. D. Aldridge and L. I. Lumb. {\it Inertial waves identified in the Earth's fluid outer core.} Nature 325 (1987).

\bibitem{ALH:earth}
K. D. Aldridge, L. I. Lumb and G. A. Henderson. {\it A Poincar\'e model for the Earth's fluid core.} Geophysical \& Astrophysical Fluid Dynamics 48 (1989).

\bibitem{e:int} Paul F. Byrd and Morris D. Friedman. Handbook of elliptic integrals for engineers and physicists. Springer, 2013.

\bibitem{BMN:global:0} Anatoli Babin, Alex Mahalov and Basil Nicolaenko. {\it Regularity and integrability of 3D Euler and Navier-Stokes equations for rotating fluids.} Asymptotic Analysis 15 (1997).

\bibitem{BMN:global} Anatoli Babin, Alex Mahalov and Basil Nicolaenko. {\it Global regularity of 3D rotating Navier-Stokes equations for resonant domains.} Indiana University Mathematics Journal (1999).

\bibitem{Ch:Book} Jean-Yves Chemin, Benoit Desjardins, Isabelle Gallagher and Emmanuel Grenier. Mathematical Geophysics: An introduction to rotating fluids and the Navier-Stokes equations. Clarendon Press, Oxford, 2006.

 \bibitem{Cheng:SIMA:2012} Bin Cheng. {\it Singular Limits and Convergence Rates of Compressible Euler and Rotating Shallow Water Equations.} SIAM J. Mathematical Analysis, 44 (2012).
 
 \bibitem{Cheng:SIMA:2014} Bin Cheng. {\it Improved Accuracy of Incompressible Approximation of Compressible Euler Equations.} SIAM J. Mathematical Analysis 46 (2014).
 
 \bibitem{CJS:3}
Bin Cheng,  Qiangchang Ju and Steve Schochet. {\it Three-scale singular limits of evolutionary PDEs.} Archive for Rational Mechanics and Analysis 229 (2018): 601-625.

\bibitem{CJS:MHD}
Bin Cheng,  Qiangchang Ju and Steve Schochet. {\it Convergence rate estimates for the low Mach and Alfvén number three-scale singular limit of compressible ideal magnetohydrodynamics.} ESAIM: Mathematical Modelling and Numerical Analysis 55 (2021).
 
 \bibitem{ChMa:zonal} Bin Cheng and Alex Mahalov. {\it Euler equation on a fast rotating sphere -- Time-averages and zonal flows.} European Journal of Mechanics-B/Fluids, 37 (2013).
  
  \bibitem{EmMa:CPDE} Pedro Embid and Andrew Majda. {\it Averaging over fast gravity waves for geophysical flows with arbitrary potential vorticity.} Comm. Partial Differential Equations 21 (1996).
 
 \bibitem{Galtier}
 S\'ebastien Galtier. {\it Weak inertial-wave turbulence theory.} Physical Review E 68 (2003).

\bibitem{Gallagher} 
Isabelle Gallagher. {\it Asymptotic of the solutions of hyperbolic equations with a skew-symmetric perturbation.} J. Differential Equations 150 (1998).

\bibitem{Gallagher:parabolic} 
Isabelle Gallagher. {\it Applications of Schochet’s methods to parabolic equations.} J.
Math. Pures Appl. 77 (1998).

\bibitem{Greenspan}
Harvey P. Greenspan.  The Theory of Rotating Fluids. Cambridge University Press, Cambridge, 1968.

\bibitem{HW:pa}
Terry Haut and Beth Wingate. {\it An asymptotic parallel-in-time method for highly oscillatory PDEs.} SIAM
J. Scientific Computing, 36 (2014).

\bibitem{KM:limit} 
Sergiu Klainerman and Andrew Majda. {\it Singular limits of quasilinear hyperbolic systems with large parameters and the incompressible limit of compressible fluids.} Comm.  Pure and Applied Mathematics 34 (1981).

\bibitem{LS:zonal}
Youngsuk Lee and Leslie M. Smith. {\it On the formation of geophysical and planetary zonal flows by near resonant
wave interactions.} J. Fluid Mechanics, 576 (2007).

\bibitem{Leray} Jean Leray. {\it Sur le mouvement d’un liquide visqueux remplissant l’espace.} Acta Mathematica, 63 (1934).

\bibitem{Newell}
Alan C. Newell. {\it Rossby wave packet interactions.} J. Fluid Mechanics, 35 (1969).

\bibitem{KY:number} 
Nobu Kishimoto and Tsuyoshi Yoneda. {\it Global solvability of the rotating Navier-Stokes equations with fractional Laplacian in a periodic domain.} Mathematische Annalen 372 (2018).

\bibitem{Pedlosky} 
Joseph Pedlosky. Geophysical Fluid Dynamics. 2nd ed. Springer-Verlag, New York, 1987.

\bibitem{Poincare} 
M. H. Poincar\'e. {\it Sur la pr\'ecession des corps d\'eformables.} Bulletin Astronomique, Serie I 27 (1910).

\bibitem{RV:shell}
M. Rieutord and L. Valdettaro. {\it Inertial waves in a rotating spherical shell.} J. Fluid Mechanics 341 (1997).

\bibitem{Schochet:JDE} 
Steven Schochet. {\it Fast singular limits of hyperbolic PDEs.} J. differential equations 127 (1994).
 
\bibitem{JS:integer} 
Hugo Steinhaus. {\it Sur un th\'eor\`eme de M. V. Jarn\'ik.} Colloquium Mathematicum 1 (1947).

\bibitem{Taylor:3}
Michael Taylor. Partial Differential Equations III: Nonlinear Equations. Vol. 117. Springer-Verlag, 1996.

\bibitem{Temam} Roger Temam. Navier-Stokes Equations, Theory and Numerical Analysis. North-Holland, Amsterdam, 1984.

\bibitem{SmLe:NR}
Leslie M. Smith and Youngsuk Lee. {\it On near resonances and symmetry breaking in forced rotating flows at
moderate Rossby number.} J. Fluid Mechanics, 535 (2005).

\bibitem{WW:fast}
Jared Whitehead and Beth Wingate. {\it The influence of fast waves and fluctuations on the evolution of the
dynamics on the slow manifold.} J. Fluid Mechanics, 757 (2014).

\end{thebibliography}

\end{document}